%% file: DAHA.tex
\documentclass[12pt]{amsart}
\usepackage{amssymb}
\usepackage{amsbsy}
\usepackage{amscd}
\usepackage{lineno}

\usepackage[mathscr]{eucal}
\usepackage{nicefrac}

\DeclareSymbolFont{rsfs}{U}{rsfs}{m}{n}
\DeclareSymbolFontAlphabet{\mathrsfs}{rsfs}

\usepackage{a4wide}
\usepackage{verbatim}
\usepackage{version}
\usepackage{color}
\newenvironment{NB}{
\color{red}{\bf NB}. \footnotesize
}{}
\newenvironment{NB2}{
\color{blue}{\bf NB2}. \footnotesize
}{}
\excludeversion{NB}
\excludeversion{NB2}
\makeatletter

\usepackage{ifmtarg}

\usepackage{xr-hyper}
\usepackage[
pdftex,
bookmarks=true,
colorlinks=true,
unicode,
pdfnewwindow=true]{hyperref}
\usepackage{xspace}

\renewcommand{\thesubsection}{\thesection(\@roman\c@subsection)}
\newcounter{number}
\makeatother
\newtheorem{Theorem}[equation]{Theorem}

\newtheorem{Lemma}[equation]{Lemma}
\newtheorem{Proposition}[equation]{Proposition}

\theoremstyle{definition}
\newtheorem{Definition}[equation]{Definition}
\newtheorem{Example}[equation]{Example}

\theoremstyle{remark}
\newtheorem{Remark}[equation]{Remark}

\numberwithin{equation}{section}

\newcommand{\thmref}[1]{Theorem~\ref{#1}}
\newcommand{\secref}[1]{\S\ref{#1}}
\newcommand{\lemref}[1]{Lemma~\ref{#1}}
\newcommand{\propref}[1]{Proposition~\ref{#1}}
\newcommand{\subsecref}[1]{\S\ref{#1}}
\newcommand{\remref}[1]{Remark~\ref{#1}}

\newcommand{\defeq}{\overset{\operatorname{\scriptstyle def.}}{=}}
\newcommand{\CC}{{\mathbb C}}
\newcommand{\ZZ}{{\mathbb Z}}

\newcommand{\SU}{\operatorname{\rm SU}}
\newcommand{\GL}{\operatorname{GL}}

\newcommand{\gl}{\operatorname{\mathfrak{gl}}}

\newcommand{\ve}{\varepsilon}
\renewcommand{\MR}[1]{}

\newcommand{\bM}{\mathbf M}
\newcommand{\bN}{\mathbf N}

\newcommand{\tslabar}{\mathbin{
\setbox0=\hbox{/\!\!/\!\!/}\rule[0.4\ht0]{\wd0}{.3\dp0}\kern-\wd0\box0}}

\newcommand{\la}{\lambda}

\makeatletter
\newcommand{\cA}[1][{}]{%
  \@ifmtarg{#1}%
  {\mathcal A}
  {\mathcal A(#1)}
}
\newcommand{\cAh}[1][{}]{%
  \@ifmtarg{#1}%
  {\mathcal A_\hbar}
  {\mathcal A_\hbar(#1)}
}
\makeatother

\newcommand{\ft}{\mathfrak t}
\newcommand{\gr}{\operatorname{gr}}

\newcommand{\Res}{\operatorname{Res}}

\newcommand{\po}{\ar@{}[dr]|{\text{\pigpenfont R}}}
\newcommand{\pb}{\ar@{}[dr]|{\text{\pigpenfont J}}}
\newcommand{\pp}{\ar@{}[dr]|{\text{\pigpenfont P}}}

\newcommand{\SDAHA}{\mathbf S\mathbf H}
\newcommand{\DAHA}{\mathbf H}

\newcommand{\sfu}{{\mathsf{u}}}
\newcommand{\bc}{\mathbf c}
\newcommand{\bt}{\mathbf t}
\newcommand{\cB}{\mathcal B}

\newcommand{\cM}{\mathcal M}
\newcommand{\EE}{h}

\begin{document}

\title[Quantized Coulomb branches and cyclotomic Cherednik algebras]
{Quantized Coulomb branches of Jordan quiver gauge theories and cyclotomic rational Cherednik algebras}
\author[R.~Kodera]{Ryosuke Kodera}
\address{Department of Mathematics,
Kyoto University, Kyoto 606-8502,
Japan}
\email{rkodera@math.kyoto-u.ac.jp}
\author[H.~Nakajima]{Hiraku Nakajima}
\address{Research Institute for Mathematical Sciences,
Kyoto University, Kyoto 606-8502,
Japan}
\email{nakajima@kurims.kyoto-u.ac.jp}

\subjclass[2000]{}
\begin{abstract}
  We study quantized Coulomb branches of quiver gauge theories of
  Jordan type.  We prove that the quantized Coulomb branch is
  isomorphic to the spherical graded Cherednik algebra in the unframed
  case, and is isomorphic to the spherical cyclotomic rational
  Cherednik algebra in the framed case.  We also prove that the
  quantized Coulomb branch is a deformation of a subquotient of the
  Yangian of the affine $\gl(1)$.
\end{abstract}
\maketitle
\setcounter{tocdepth}{2}

\input{introduction}
\input{Cherednik}
\input{cyclotomic}
\input{Yangian}
\input{auto}
\input{Yangian2}

\bibliographystyle{myamsalpha}
\bibliography{nakajima,mybib,coulomb,DAHA}

\end{document}

%% file: introduction.tex
\section{Introduction}

\subsection{}

Let $G$ be a complex reductive group and $\bM$ be its finite
dimensional symplectic representation. To such a pair, physicists
consider a $3d$ $\mathcal N=4$ supersymmetric gauge theory, and
associate a hyper-K\"ahler manifold possibly with a singularity with
an $\SU(2)$-action rotating complex structures, called the
\emph{Coulomb branch} of the gauge theory. Its physical definition
involves quantum corrections, hence is difficult to be justified in a
mathematically rigorous way.
The second named author proposed an approach towards a mathematically
rigorous definition of the Coulomb branch
\cite{2015arXiv150303676N}. 
When the symplectic representation $\bM$ is of a form
$\bN\oplus\bN^*$, its definition as an affine variety with a
symplectic form on the regular locus, together with a
$\CC^\times$-action was firmly established in \cite{2016arXiv160103586B} written
by the second named author with Braverman and Finkelberg.

As a byproduct of the definition in
\cite{2015arXiv150303676N,2016arXiv160103586B}, the Coulomb branch
$\cM$ has a natural quantization, i.e., we have a noncommutative
algebra $\cAh$ over $\CC[\hbar]$ such that its specialization
$\cAh\otimes_{\CC[\hbar]}(\CC[\hbar]/\hbar\CC[\hbar])$ at $\hbar = 0$
is isomorphic to the coordinate ring $\CC[\cM]$, and the Poisson
bracket associated with the symplectic form is given by the formula
$\{\ ,\ \} = \frac{[\ , \ ]}\hbar \bmod \hbar$. We call $\cAh$ the
\emph{quantized Coulomb branch}.

By a general result proved in \cite{2016arXiv160103586B}, the
quantized Coulomb branch $\cAh$ is embedded into a localization of the
quantized Coulomb branch associated with the maximal torus $T$ of $G$
and the zero representation. The latter quantized Coulomb branch is
the ring of difference operators on the Lie algebra of $T$, and the
localization is the complement of a finite union of hyperplanes in
$\operatorname{Lie}T$.

This embedding was further studied for a framed quiver gauge theory of
type $ADE$ in two appendices of \cite{2016arXiv160403625B}, written by
the present authors with Braverman, Finkelberg, Kamnitzer, 
Webster and Weekes. In particular, generators of the quantized Coulomb
branch are given by explicit difference operators, and it was shown
that the quantized Coulomb branch is a quotient of shifted Yangian
introduced in \cite{kwy} (when the dominance condition is satisfied).

In this paper, we study the quantized Coulomb branch of a
framed quiver gauge theory of Jordan type, i.e., the gauge group is
$\GL(N)$, and its representation $\bN$ is the direct sum of the
adjoint representation $\gl(N)$ and $l$ copies of the vector
representation $\CC^N$, where $l$ is a nonnegative integer.
The Coulomb branch of the gauge theory was known in \cite{MR1454291}:
it is the $N$th symmetric power of the surface $\mathcal S_l$ in
$\CC^3$ given by the equation $xy = z^l$.  (For example,
$\mathcal S_0$ is $\CC\times\CC^\times$.) The definition in
\cite{2015arXiv150303676N,2016arXiv160103586B} reproduces this answer
\cite[\S 3(vii)]{2016arXiv160403625B}. We have its quantization,
namely the spherical part $\SDAHA_{N,l}^{\operatorname{cyc}}$ of the
cyclotomic rational Cherednik algebra $\DAHA_{N,l}^{\operatorname{cyc}}$
associated with the wreath product $\mathfrak S_N\ltimes(\ZZ/l\ZZ)^N$
if $l > 0$, and the spherical part $\SDAHA_N^{\operatorname{gr}}$ of
the graded Cherednik algebra $\DAHA_N^{\operatorname{gr}}$ associated
with $\mathfrak S_N$ if $l=0$. Here the cyclotomic rational Cherednik algebra
$\DAHA_{N,l}^{\operatorname{cyc}}$ with $l=1$ is understood as the
rational Cherednik algebra associated with $\mathfrak S_N$, which is
the rational degeneration of $\DAHA_N^{\operatorname{gr}}$.

Our first main result is to show that the quantized Coulomb branch is isomorphic to this quantization:
\begin{Theorem}\label{thm:main1}
    The quantized Coulomb branch $\cAh$ of the gauge theory $(G,\bN) =
    (\GL(N), \gl(N)\oplus (\CC^N)^{\oplus l})$ is isomorphic to
    $\SDAHA_{N,l}^{\operatorname{cyc}}$ if $l > 0$, and to
    $\SDAHA_N^{\operatorname{gr}}$ if $l=0$. The parameters $\bt$,
    $\hbar$ of the quantized Coulomb branch and the Cherednik algebra
    are the same, and others are matched by
  \begin{equation*}
     z_k = - \frac1{l}\left(
       (l-k)\hbar + \sum_{m=1}^{l-1} \frac{c_m(1 - \ve^{mk})}{1 - \ve^m}
       \right)
  \end{equation*}
  when $l > 0$.
\end{Theorem}

\begin{Remark}
  The quantized Coulomb branch has the parameter $z_1$, \dots, $z_l$
  corresponding to equivariant variables for the additional
  $(\CC^\times)^l$-action. However the overall shift
  $z_1 \to z_1 + c$, \dots, $z_l \to z_l + c$ is irrelevant. Therefore
  our convention $z_l = 0$ does not loose generality.
\end{Remark}

While we are preparing the paper, we notice that Losev shows that any
filtered quantization of $\operatorname{Sym}^N \mathcal S_l$ is
$\left.\SDAHA_{N,l}^{\operatorname{cyc}}\right|_{\hbar=1}$ for some
choice of parameters \cite{2016arXiv160500592L}.
In order to calculate parameters, we can use localization of the
quantized Coulomb branch as in \cite[\S3(ix)]{2016arXiv160403625B} to
reduce the cases $l=0$ or $N=1$. (We learn the argument through
discussion with Losev.) The case $l=0$ is easy to handle, as we will
do in \secref{sec:no_frame}. The proof for the case $N=1$ requires
only \propref{prop:parameters}, thus our calculation is reduced about
to the half. We think that our proof is elementary and interesting its own way.

The second main result is the relation of $\cAh$ to the Yangian
$Y(\widehat{\gl}(1))$ of the affine $\gl(1)$, which appeared in
cohomology of moduli spaces of framed torsion free sheaves on $\CC^2$
(\cite{2012arXiv1211.1287M,MR3150250}). We will use its presentation
in \cite{MR3077678} (see also \cite{2014arXiv1404.5240T}) by
generators and relations.

\begin{Theorem}[\thmref{thm:shiftedYangian} for detail]
  The quantized Coulomb branch $\cAh$ is a deformation of a
  subquotient of $Y(\widehat{\gl}(1))$. A little more precisely, we
  consider a subalgebra of $Y(\widehat{\gl}(1))$ generated by elements
  $D_{0,m}$ \textup($m\ge 1$\textup), $e_n$, $f_{n+l}$
  \textup($n\ge 0$\textup) and deform it by replacing the relation
  \eqref{eq:DE} by \eqref{eq:7}. Let $Y_l(\vec{z})$ denote the
  resulted algebra. We have a surjective homomorphism
  $Y_l(\vec{z})\to \cAh$.
\end{Theorem}

This result is not surprising as $Y(\widehat{\gl}(1))$ is introduced
as a limit of $\SDAHA^{\operatorname{gr}}_N$ as $N\to\infty$ in
\cite{MR3150250}. But we will give a self-contained proof starting
from the presentation in \cite{MR3077678} so that it also gives the
result with $l > 0$. We call $Y_l(\vec{z})$ the \emph{shifted Yangian}
of $\widehat{\gl}(1)$, as it is an analog of the shifted Yangian of
$\gl_n$ \cite{BruKle-alg,BruKle} and a finite dimensional simple Lie
algebra \cite{kwy}.

\subsection{}

All necessary computation for the quantized Coulomb branch in the
proofs of two main theorems is already given in appendices of
\cite{2016arXiv160403625B}. The remaining steps are to relate the
spherical cyclotomic rational Cherednik algebra and the affine Yangian of
$\gl(1)$ with the ring of difference operators on
$T = (\CC^\times)^N$. These steps are completely independent of the
quantized Coulomb branch. Let us formulate them as results on those
algebras.

Let $\cAh[T,0]$ be the ring of $\hbar$-difference operators on the Lie
algebra of the torus $T$. Taking coordinates, we represent it as the
$\CC[\hbar]$-algebra with generators $w_i$, $\sfu_i^{\pm 1}$
($1\le i\le \dim T$) with relations
\begin{equation*}
  [w_i, w_j] = 0 = [\sfu_i,\sfu_j], \qquad
  \sfu_i^{-1} \sfu_i = 1 = \sfu_i \sfu_i^{-1}, \qquad
  [\sfu_i^{\pm 1}, w_j] = \pm\delta_{i,j} \hbar \sfu_i^{\pm 1}.
\end{equation*}

Let $\bt$, $z_1,\dots,z_l$ be other variables.
We introduce difference operators
\begin{equation}\label{eq:79}
    \begin{split}
    & E_{n}[f] \defeq
    \sum_{\substack{I\subset \{1,\dots,N\}\\ \# I =
        n}} 
    f(w_I)
    \prod_{i\in I, j\notin I}
    \frac{
      w_i - w_j - \bt}
    {w_i-w_j}
    \prod_{i\in I} \sfu_i,
\\
    & 
        F_{n}[f] \defeq
    \sum_{\substack{I\subset \{1,\dots,N\}\\ \# I = n}}
    f(w_I - \hbar)
    \prod_{i\in I, j\notin I}
    \frac{
      w_i  - w_j + \bt}
    {w_i-w_j}
    \prod_{i\in I} \left(
    \prod_{k=1}^l (w_i - \hbar - z_k)
    \cdot\sfu_i^{-1}\right),
    \end{split}
\end{equation}
where $1\le n\le N$ and $f$ is a symmetric polynomial in $n$
variables, and $f(w_I)$, $f(w_I - \hbar)$ mean substitution of
$\left(w_i\right)_{i\in I}$, $\left(w_i-\hbar\right)_{i\in I}$ to $f$
respectively. These are elements in the localized ring
\begin{equation*}
  \tilde{\cAh} = \CC[\hbar,\bt,z_1,\dots,z_l]\langle w_i, \sfu_i^{\pm 1},
  (w_i - w_j)^{-1}\rangle/\text{relations}.
\end{equation*}
Then \cite[Th.~A.7]{2016arXiv160403625B} says the quantized Coulomb
branch $\cAh$ is the subalgebra of $\tilde{\cAh}$ generated by
operators $E_n[f]$, $F_n[f]$, and symmetric polynomials in
$(w_1,\dots,w_N)$. This result gives us an algebraic characterization
of $\cAh$. We will use it as a starting point, and will not use the
original definition of $\cAh$ in
\cite{2015arXiv150303676N,2016arXiv160103586B}.

If $f\equiv 1$, $E_n[1]$ is a rational version of the $n$th Macdonald
operator, once we understand $\sfu_i$ as the $\hbar$-difference
operator $f(w_1,\dots,w_N) \mapsto
f(w_1,\dots,w_i+\hbar,\dots,w_N)$.
The same is true for $F_n[1]$ with $l=0$. These observation will give
us a link between $\cAh$ and the graded Cherednik algebra with $l=0$,
and lead a proof of \thmref{thm:main1} with $l=0$. (See
\secref{sec:no_frame}.) The proof of \thmref{thm:main1} with $l>0$ is
given in the same way, by relating $E_n[f]$, $F_n[f]$ with the
cyclotomic rational Cherednik algebra.

\begin{Theorem}\label{thm:embedding}
  There is a faithful embedding of $\SDAHA^{\mathrm{cyc}}_{N,l}$ to
  the ring $\tilde\cAh$ of localized difference operators such that
  \begin{gather*}
    e_{\Gamma_N} \left(\sum_{i=1}^N (l^{-1}\xi_i\eta_i
      + \bt\sum_{j<i} s_{ji})^n \right) e_{\Gamma_N}
    = \sum_{i=1}^N w_i^n \quad (n > 0),
    \begin{NB}
      \text{Do you have a better relation ?}
    \end{NB}%
    \\
      e_{\Gamma_N} \left(\sum_{i=1}^N \xi_i^l\right) e_{\Gamma_N} = E_1[1], \qquad
      e_{\Gamma_N} \left(\sum_{i=1}^N (l^{-1}\eta_i)^l\right) e_{\Gamma_N} = F_1[1],
  \end{gather*}
  where $e_{\Gamma_N}$ is the idempotent for the group
  $\Gamma_N = \mathfrak S_N \ltimes (\ZZ/l\ZZ)^N$, and $\xi_i$,
  $\eta_i$ are generators of $\DAHA^{\mathrm{cyc}}_{N,l}$. \textup(See
  \secref{sec:cycl-rati-cher} for the presentation of
  $\DAHA^{\mathrm{cyc}}_{N,l}$\textup).
\end{Theorem}

This embedding is given by the composition of an embedding
$e_\Gamma \DAHA^{\mathrm{cyc}}_{N,l} e_\Gamma\to
\DAHA^{\mathrm{gr}}_N[c_1,\dots,c_{l-1}]$
in \cite{MR2318640} (also \cite{MR2181264} for $l=1$) and the faithful
embedding of $\DAHA^{\mathrm{gr}}_N$ by rational Demazure-Lusztig
operators. Here $e_\Gamma$ is the idempotent for the subgroup
$(\ZZ/l\ZZ)^N$.

For Yangian $Y(\hat{\gl}(1))$, we define a homomorphism to $\tilde\cAh$
by setting images of generators $D_{0,m}$, $e_n$, $f_n$ as explicit
operators. This is an analog of representations of Yangian of finite
type by difference operators \cite{GKLO,kwy}.

The paper is organized as follows.  We give a proof of
\thmref{thm:main1} with $l=0$ in \secref{sec:no_frame}, assuming some
results on the graded Cherednik algebra.  In \secref{sec:Poisson} we
compute some Poisson brackets in order to reduce generators of the
quantized Coulomb branch.  In \secref{sec:DAHA} we recall the
definitions and some properties of the graded and rational Cherednik
algebras.  It includes two embeddings of the graded Cherednik algebra,
one is into the ring of differential operators via Dunkl operators and
the other is into the ring of difference operators via rational
Demazure-Lusztig operators.  In \secref{sec:cycl-rati-cher} we recall
the definition of the cyclotomic rational Cherednik algebra and an
embedding of its partially symmetrized subalgebra into the graded
Cherednik algebra due to Oblomkov.  Then some calculations yield
\thmref{thm:embedding}, and hence \thmref{thm:main1} is proved for
general $l$.  We relate the quantized Coulomb branch to the affine
Yangian in \secref{sec:aff_Yangian}, \secref{sec:app},
\secref{sec:app2}.

\subsection*{Acknowledgments}

This research was originally started together with A.~Braverman and
M.~Finkelberg as continuation of the appendix in
\cite{2016arXiv160403625B}. P.~Etingof continuously helped us from the
beginning. It became soon clear that we only need to calculate the
image of certain operators in the partially symmetrized
$\DAHA^{\mathrm{cyc}}_{N,l}$ under the embedding to
$\DAHA^{\mathrm{gr}}_N$ given in \cite{MR2318640}. The authors of this
paper took this computational approach, while the other three consider
another route through non-spherical Cherednik algebras. Hence the
papers are written separately as this paper and \cite{BEF}. The
authors thank them for sharing many insights. We also thank I.~Losev
for discussing another derivation of \thmref{thm:main1} from
\cite{2016arXiv160500592L}.

H.N.\ thanks A.~Kirillov for teaching him basics on Cherednik
algebras, T.~Suzuki for telling him the reference \cite{MR2318640},
and A.~Tsymbaliuk for discussion on affine Yangian.

This work was started when H.N.\ was staying at the Simons Center for
Geometry and Physics, and was completed when he was visiting at the
Aspen Center for Physics, which is supported by National Science
Foundation grant PHY-1066293. H.N.\ thanks both institutes for
hospitality and nice research atmosphere.

The research of H.N.\ is supported by JSPS Kakenhi Grant Numbers
24224001, 
25220701, 
16H06335. 
The research of R.K.\ is supported by JSPS Kakenhi Grant Number
25220701.

\section{The case with no framing}\label{sec:no_frame}

Let us start with the case $l=0$ to illustrate our strategy. We assume
some results on the graded Cherednik algebra. The proof of the case
$l > 0$ will be given after preparing necessary results for cyclotomic
rational Cherednik algebras in \secref{sec:DAHA}.

Consider the commutator of the class $E_{n}[f]$ and the $k$th power
sum $p_k(\vec{w}) = \sum_{i=1}^N w_i^k$. It is
\begin{equation*}
    \left[ p_k(\vec{w}), E_{n}[f]\right]
    = - 
      \sum_{\substack{I\subset \{1,\dots,N\}\\ \# I = n}}
      \left( p_k((w_i+ \hbar)_{i\in I})
        - p_k(w_I)\right)
      f(w_I)
      \prod_{i\in I, j\notin I}
      \frac{
        w_{i} - w_{j} - \bt}
      {w_{i}-w_{j}}
      \prod_{i\in I} \sfu_{i}, 
\end{equation*}
thanks to the commutation relation
\(
   [\prod_{i\in I} \sfu_i,w_j] = \hbar \prod_{i\in I} \sfu_j
\)
if $j\in I$ and $0$ otherwise.
Therefore we can get elements $E_n[f]$ for general $f$
inductively from $E_n[1]$ and symmetric functions in $w_i$ by taking
commutators divided by $\hbar$. The same is true for $F_n[f]$.
(In particular, $\cA$ is generated by $E_n[1]$, $F_n[1]$ and symmetric
functions in $\vec{w}$ \emph{as a Poisson algebra}.)
%


\begin{proof}[Proof of \thmref{thm:main1} with $l=0$]
    Consider the embedding of $\SDAHA_N^{\operatorname{gr}}$ to the ring of rational
    difference operators on $\ft = \CC^N$, obtained as the
    trigonometric degeneration of the usual embedding of the spherical
    part of the double affine Hecke algebra $\DAHA_N$. Symmetric functions in $w_i$ are considered as
    functions on $\ft$. (The detail will be reviewed in \subsecref{subsec:DL} for a reader who is unfamiliar with Cherednik algebras.)

    Since $E_n[1]$, $F_n[1]$ are trigonometric degeneration of Macdonald
    operators, they are contained in $\SDAHA_N^{\operatorname{gr}}$.
    (In the notation in \secref{sec:DAHA}, $E_n[1]$ and $F_n[1]$ correspond
    to $S(e_n(X))S$ and $S(e_n(X^{-1}))S$ respectively, where $S$ is
    the symmetrizer and $e_n$ is the $n$th elementary symmetric polynomial. cf.\ Example~\ref{ex:gl2}.)
    \begin{NB}
        The assertion for $E_n$ is well-known. For $F_n$, it is
        probably possible to compute directly $S(e_n(X^{-1}))S$, but I
        consider as follows. Note $F_n$ is diagonalized by Macdonald
        polynomials as $P_\lambda(q,t) =
        P_\lambda(q^{-1},t^{-1})$. Also the eigenvalues are
        eigenvalues of $E_n$ with $q\mapsto q^{-1}$, $t\mapsto
        t^{-1}$. This means the assertion.
    \begin{NB2}
    For $F_n$, observe that $F_n$ is equal to $E_n$ if we replace
    $w_r$ by $-w_r$. In other words, $F_n = \omega E_n \omega$, where
    $(\omega f)(\vec{w}) = f(-\vec{w})$.
        Or, it is also true that $F_n = E_n|_{\substack{\hbar\mapsto
            -\hbar\\ \bt\mapsto -\bt}}$.
    \end{NB2}%
    \end{NB}%

    By the explanation preceding the proof, all operators
    $E_n[f]$, $F_n[f]$ in \eqref{eq:79} are also contained in $\SDAHA_N^{\operatorname{gr}}$.
    Therefore, by \cite[Th.~A.7]{2016arXiv160403625B}, the image of $\cAh$ in $\tilde{\cAh}$ is contained
    in $\SDAHA_N^{\operatorname{gr}}$. Thus we have an injective homomorphism
    $\cAh\to \SDAHA_N^{\operatorname{gr}}$. We know that both $\cAh$, $\SDAHA_N^{\operatorname{gr}}$
    degenerate to $\operatorname{Sym}^N (\CC\times\CC^\times)$ at
    $\hbar=0=\bt$, we are done.
\end{proof}

\begin{NB}
This is commented out, as we have already explained in Introduction.

We will apply the same argument to the case $l>0$ after we construct
an embedding of the cyclotomic rational Cherednik algebra to the ring of
difference operators such that its image contains \eqref{eq:79} with
$f=1$ (and symmetric functions).
 \end{NB}

\section{Poisson brackets}\label{sec:Poisson}

Let us continue computation of Poisson brackets in order to reduce the number of Poisson generators further.


\begin{NB}
Set $\bt = 0$, $z_k=0$ in \eqref{eq:79}. We have
$F_1 = \sum_{r=1}^a (w_r - \hbar)^l \sfu_r^{-1}$. We have
\begin{equation*}
    \begin{split}
        & [F_1[ w - \hbar], F_1] 
        = \left[ \sum_{r=1}^a (w_r - \hbar)^{l+1} \sfu_r^{-1},
          \sum_{s=1}^a (w_s - \hbar)^l \sfu_s^{-1}\right] \\
        =\; & 
        \sum_{r=1}^a \left[ w_r - \hbar, (w_r - \hbar)^l
          \sfu_r^{-1}\right] (w_r - \hbar)^l \sfu_r^{-1}
        = \hbar
        \sum_{r=1}^a \left\{ (w_r - \hbar)^l \sfu_r^{-1} \right\}^2.
    \end{split}
\end{equation*}
Therefore
\begin{equation*}
    \left\{ \sum_{r=1}^a w_r y_r, \sum_{s=1}^a y_s \right\}
    = \sum_{r=1}^a y_r^2.
\end{equation*}
More generally
\begin{equation*}
    \begin{split}
        & \left[ \sum_{r=1}^a (w_r - \hbar)^{l+1} \sfu_r^{-1},
          \sum_{s=1}^a \left\{(w_s - \hbar)^l \sfu_s^{-1}\right\}^p
        \right] \\
        =\; & 
        \sum_{r=1}^a \left[ w_r - \hbar, \left\{ (w_r - \hbar)^l
          \sfu_r^{-1} \right\}^p \right] (w_r - \hbar)^l \sfu_r^{-1}
        = p \hbar
        \sum_{r=1}^a \left\{ (w_r - \hbar)^l \sfu_r^{-1} \right\}^{p+1}.
    \end{split}
\end{equation*}
Therefore
\begin{equation*}
    \left\{ \sum_{r=1}^a w_r y_r, \sum_{s=1}^a y_s^p \right\}
    = \sum_{r=1}^a y_r^{p+1}.
\end{equation*}
We can produce power sums from the first one $\sum_r y_r$ and $\sum_r
w_r y_r$. On the other hand, the latter can be produced from the first
as
\begin{equation*}
    \left[ \sum_r (w_r - \hbar)^l \sfu_r^{-1}, \sum_s w_s(w_s - \hbar)\right]
    = -2\hbar \sum_r (w_r - \hbar)^{l+1} \sfu_r^{-1},
\end{equation*}
and hence
\begin{equation*}
    \left\{
      \sum_{r=1}^a y_r, \sum_{s=1}^a w_r^2
      \right\} = -2 \sum_{r=1}^a w_r y_r.
\end{equation*}
 \end{NB}%

Recall the Coulomb branch is the $N$th symmetric power $\operatorname{Sym}^N \mathcal S_l$ of the surface $\mathcal S_l = \{ xy = z^l \}$ in $\CC^3$. Let us introduce $x_i$, $y_i$, $z_i$ ($1\le i\le N$) for functions on $(\mathcal S_l)^N$. The Poisson brackets
are given by
\begin{equation*}
    \{ x_i, y_j \} = \delta_{ij} w_i^{l-1}, \qquad
    \{ w_i, x_j \} = -\delta_{ij} x_i, \qquad 
    \{ w_i, y_j \} = \delta_{ij} y_j.
\end{equation*}
\begin{NB}
Recall $x = \sfu$, $y = w^l \sfu^{-1}$. Therefore $[w,x] = -\hbar x$, $[w, y] = \hbar y$.
\end{NB}%
Hence
\begin{equation*}
    \left\{
      \sum_{i=1}^N w_i^2, \sum_{j=1}^N y_j
      \right\} = 2 \sum_{i=1}^N w_i y_i, \qquad
      \left\{ \sum_{i=1}^N w_i y_i, \sum_{j=1}^N y_j^n \right\}
    = n\sum_{i=1}^N y_i^{n+1}
\end{equation*}
for $n\in\ZZ_{>0}$. Note that $F_n[1]$ is specialized to the $n$th elementary symmetric polynomial in $y_i$ at $\bt = \hbar = z_k = 0$. Therefore the above implies that we can obtain elements $F_n[1]$ inductively from $F_1[1]$ and symmetric polynomials in $w_i$ by taking commutators divided by $\hbar$. The same is true for $E_n[1]$. Therefore to obtain an isomorphism $\cAh\cong \SDAHA_{N,l}^{\operatorname{cyc}}$, it is enough to construct an algebra embedding of $\SDAHA_{N,l}^{\operatorname{cyc}}$ to the ring $\tilde\cAh$ of rational difference operators on $\operatorname{Lie}T$ such that $E_1[1]$, $F_1[1]$ are contained in $\SDAHA_{N,l}^{\operatorname{cyc}}$.


%% file: Cherednik.tex
\section{Cherednik algebras}\label{sec:DAHA}

We recall definitions of various versions of Cherednik algebras and their faithful
representations by differential (Dunkl) difference operators. Our
basic reference is \cite{MR1441642}. All the results are due to Cherednik.

\subsection{Definitions}

The \emph{graded Cherednik algebra} (alias the \emph{trigonometric
  double affine Hecke algebra}) $\DAHA^{\gr}_N$ for $\gl_N$ is the
$\CC[\hbar,\bt]$-algebra generated by $\pi^{\pm 1}$, $s_0$, \dots,
$s_{N-1}$, $w_1$, \dots, $w_N$ with the relations
\begin{subequations}
\begin{gather}
    [w_i, w_j] = 0 \quad (i,j=1,\dots,N),
\\
\begin{minipage}[t]{.8\linewidth}
    \begin{center}
        $\pi$, $s_i$ ($i=0,\dots,N-1$) satisfy the relation of the
        extended affine Weyl group of $\gl_N$, e.g., $\pi s_i =
        s_{i+1} \pi$ ($i=0,\dots,N-1\bmod N$), etc,
    \end{center}
\end{minipage}
\\
    \pi w_i = w_{i+1} \pi \quad (i=1,\dots,N-1), \qquad
    \pi w_N = (w_1 + \hbar) \pi,
\\
\label{eq:88}
    s_i w_i = w_{i+1} s_i - \bt, \quad
    s_i w_{i+1} = w_i s_i + \bt, \quad 
    s_i w_j = w_j s_i \quad (i=1,\dots, N-1, j\neq i,i+1),
\\
    s_0 w_1 = (-\hbar + w_N) s_0 + \bt, \quad
    s_0 w_N = (\hbar+ w_1) s_0 - \bt, \quad
    s_0 w_i = w_i s_0 \quad (i\neq 1,N).
\end{gather}
\end{subequations}
This is the presentation obtained from one for the original Cherednik
algebra (see e.g., \cite[Def.~4.1]{MR1441642}) by setting $X_i =
\exp(\boldsymbol\beta w_i)$, $q^2 = \exp(\boldsymbol\beta \hbar)$,
$t^2 = \exp(-\boldsymbol\beta\bt)$ and taking the limit
$\boldsymbol\beta\to 0$.
\begin{NB}
    See 2016-02-20 Demazure Lusztig embedding.xoj.
\end{NB}%
We will not use the original Cherednik algebra considered in
\cite{MR1441642}, in particular we will use the notation $X_i$ for a
different object below.

\begin{NB}
    Remark that we took the convention $t^2 =
    \exp(\boldsymbol\beta\bt)$ for a while.
\end{NB}%

The graded Cherednik algebra $\DAHA^{\gr}_N$ has another presentation
with generators $s_1$, \dots, $s_{N-1}$, $X_i^{\pm 1}$, $w_i$
($i=1,\dots,N$) and relations
\begin{subequations}
\begin{gather}
    [w_i, w_j] = 0 = [X_i, X_j] \quad (i,j=1,\dots,N),
    \\
    \text{$s_i$ ($i=1,\dots,N-1$) are the standard generators of the
      symmetric group $\mathfrak S_N$},
    \\
    \text{the same as \eqref{eq:88}},
    \\
    s X_i^{\pm 1}= X_{s(i)}^{\pm 1}s \quad (s\in \mathfrak S_N),
    \\
    [w_i, X_j] =
    \begin{cases}
        -\bt X_j s_{ji} & \text{if $i > j$},\\
        -\bt X_i s_{ij} & \text{if $i < j$},\\
        -\hbar X_i + \bt \sum_{k<i} X_k s_{ki} 
        + \bt \sum_{k > i} X_i s_{ik} & \text{if $i=j$},
    \end{cases}
\end{gather}
\end{subequations}
where $s_{ij}$ is the transposition $(ij)$. The isomorphism of two
presentations is given by setting $X_1 = \pi s_{N-1} \dots s_2 s_1$,
$X_2 = s_1 X_1 s_1 = s_1\pi s_{N-1} \dots s_3 s_2$, etc. (See \cite[Section 1]{MR1653020} for detail.)
\begin{NB}
    See `2016-02-18 trigonometric DAHA.xoj' for the proof.
\end{NB}%
The inverse is given by 
$\pi = X_1 s_1 s_2 \dots s_{N-1}$,
$s_0 = \pi^{-1} s_1 \pi$.
This presentation matches with one in \cite[\S1.1]{MR3150250} by
setting $\bt = \kappa$, $\hbar = -1$.

\begin{NB}
    The following might be useful: $s_{ik} w_j = w_j s_{ik} - \bt s_{ik}
    s_{jk} + \bt s_{jk} s_{ik}$ for $i < j < k$.
\end{NB}%

The \emph{rational Cherednik algebra} $\DAHA_N^{\operatorname{rat}}$
for $\gl_N$ is the quotient of the algebra $\CC[\hbar,\bt]\langle
x_1,\dots,x_N, \linebreak[2] y_1,\dots, y_N\rangle \rtimes\mathfrak
S_N$ by the relations
\begin{subequations}
\begin{gather}
    [x_i, x_j] = 0 = [y_i, y_j] \quad (i,j=1,\dots N),
\\
    [y_i, x_j] =
    \begin{cases}
        -\hbar + \bt \sum_{k\neq i} s_{ik} & \text{if $i=j$}, \\
        - \bt s_{ij} & \text{if $i\neq j$}.
    \end{cases}
\end{gather}
\end{subequations}

Suzuki \cite{MR2181264} introduced an embedding $\iota\colon
\DAHA_N^{\operatorname{rat}}\to \DAHA^{\gr}_N$ given by
\begin{equation}
    \label{eq:89}
\begin{gathered}
    \iota(w) = w \quad (w\in\mathfrak S_N),\\
    \iota(x_i) = X_i \quad (i=1,\dots, N),\\
    \iota(y_i) = X_i^{-1}\left( w_i - \bt \sum_{j < i} s_{ji} \right)
    \quad (i=1,\dots, N).
\end{gathered}
\end{equation}
It will be clear that this is an algebra embedding thanks to
trigonometric Dunkl operators recalled in the next subsection.

\begin{NB}
\begin{Example}
    Suppose $N=2$. We have $y_1 = X_1^{-1} w_1 $, $y_2 =
    X_2^{-1} (w_2 - \bt s_1)$. Then
    \begin{equation*}
        \begin{split}
            y_2 y_1
            & = X_2^{-1} (w_2 - \bt s_1) X_1^{-1} w_1
            = X_2^{-1}(X_1^{-1} w_2 + [w_2, X_1^{-1}]) w_1
            - \bt X_2^{-1} s_1 X_1^{-1} w_1
            \\
            & =
            X_2^{-1} X_1^{-1} w_2 w_1
            + X_2^{-1}(- X_1^{-1} [w_2, X_1]X_1^{-1}) w_1
            - \bt X_2^{-2} s_1 w_1
            \\
            & =
            X_2^{-1} X_1^{-1} w_2 w_1
            + \bt X_2^{-1} X_1^{-1} X_1 s_1 X_1^{-1} w_1
            - \bt X_2^{-2} s_1 w_1
            = X_2^{-1} X_1^{-1} w_1 w_2,
            \\
            y_1 y_2 &=
            X_1^{-1} w_1 X_2^{-1}(w_2 - \bt s_1)
            = X_1^{-1}(X_2^{-1} w_1 + [w_1, X_2^{-1}]) 
            (w_2 - \bt s_1)
            \\
            &=
            X_1^{-1} X_2^{-1} w_1(w_2 - \bt s_1)
            - X_1^{-1} X_2^{-1} [w_1, X_2] X_2^{-1}(w_2 - \bt s_1)
            \\
            &=
            X_1^{-1} X_2^{-1} w_1(w_2 - \bt s_1)
            + \bt X_1^{-1} X_2^{-1} X_1 s_1 X_2^{-1} (w_2 - \bt s_1)
            \\
            &=
            X_1^{-1} X_2^{-1} w_1(w_2 - \bt s_1)
            + \bt X_1^{-1} X_2^{-1} (s_1 w_2 - \bt)
            = X_1^{-1} X_2^{-1} w_1 w_2.
        \end{split}
    \end{equation*}
    Therefore $[y_1, y_2] = 0$. Furthermore
    \begin{equation*}
        \begin{split}
        y_2 X_2 &= X_2^{-1}(w_2 - \bt s_1) X_2
        = X_2^{-1}(-\hbar X_2 + \bt X_1 s_1 + X_2 w_2)
        - \bt X_2^{-1} s_1 X_2
        \\
        &= -\hbar + w_2,
        \\
        X_2 y_2 &= X_2 X_2^{-1}(w_2 - \bt s_1) = w_2 - \bt s_1,
        \\
        y_2 X_1 &= X_2^{-1}(w_2 - \bt s_1) X_1
        =  X_2^{-1}(- \bt X_1 s_1 + X_1 w_2)
        - \bt X_2^{-1} s_1 X_1
        \\
        &=
        - \bt X_2^{-1} X_1 s_1 + X_2^{-1} X_1 w_2 - \bt s_1
        \\
        X_1 y_2 &=  X_1 X_2^{-1}(w_2 - \bt s_1)
        = X_2^{-1} X_1 w_2 - \bt X_2^{-1} X_1 s_1,
        \\
        y_1 X_2 &= X_1^{-1} w_1 X_2 = X_1^{-1} (X_2 w_1 - \bt X_1 s_1)
        = X_1^{-1} X_2 w_1 - \bt s_1,
        \\
        X_2 y_1 &= X_2 X_1^{-1} w_1,
        \\
        y_1 X_1 &= X_1^{-1} w_1 X_1 = w_1 + X_1^{-1}(-\hbar X_1 + \bt X_1 s_1)
        = w_1 - \hbar + \bt s_1,
        \\
        X_1 y_1 &= w_1.
        \end{split}
    \end{equation*}
Therefore the defining relations follow.
\end{Example}

\end{NB}%

\subsection{Dunkl operators}

Consider the (Laurent) polynomial ring
$\CC[\hbar,\bt,\linebreak[2]X_1^{\pm 1},\dots,X_N^{\pm 1}]$. We define a
representation of $\DAHA^{\gr}_N$ as follows: $X_i^{\pm 1}$, $s_i$ act in
the standard way and $w_i$ acts by the trigonometric Dunkl operator
\begin{equation*}
    w_i\mapsto
    -\hbar X_i\frac{\partial}{\partial X_i}
    + \bt \sum_{k\neq i} \frac{X_i}{X_i - X_k}(1 - s_{ik})
    + \bt \sum_{k < i} s_{ik}.
\end{equation*}
See \cite[Th.~3.7]{MR1805058} for the proof. It is not difficult to check the defining
relations directly.
\begin{NB}
    See `2016-02-19 Dunkl embedding.xoj' for the proof.
\end{NB}%

The corresponding representation of the rational Cherednik algebra
$\DAHA_N^{\operatorname{rat}}$ is given by the rational Dunkl operator
\begin{equation*}
    y_i \mapsto
    -\hbar \frac{\partial}{\partial x_i}
    + \bt \sum_{k\neq i} \frac1{x_i - x_k}(1 - s_{ik}).
\end{equation*}
See \cite[Prop.~4.5]{MR1881922} for the proof. It is even simpler to check the defining
relations than the trigonometric case.
\begin{NB}
    For example, for $i\neq j$,
    \begin{equation*}
        [y_i, x_j] = \bt \frac1{x_i - x_j}(1 - s_{ij}) (x_j)
        - \bt x_j \frac{1}{x_i - x_j}(1 - s_{ij})
        = -\bt s_{ij},
    \end{equation*}
    and
    \begin{equation*}
        [y_i, x_i] = -\hbar + \bt\sum_{k\neq i} \frac1{x_i - x_k} (1-s_{ik})(x_i)
        - \bt \sum_{k\neq i} x_i \frac1{x_i - x_k}(1 - s_{ik})
        = -\hbar + \bt \sum_{k\neq i} s_{ik}.
    \end{equation*}
    The commutation relation $[y_i, y_j] = 0$ is checked as in
    \cite[Th.~6.5]{Etingof-CM}: we check $[[y_i, y_j],x_k] = [[y_i,
    x_k],y_j] -[[y_j,x_k],y_i] = 0$ by computation. Then the assertion
    follows from the obvious claim $[y_i, y_j] 1 = 0$.
\end{NB}%
It is known that this is a faithful representation. (See \cite[Prop.~4.5]{MR1881922}.) Now
trigonometric and rational Dunkl operators are compatible, hence
\eqref{eq:89} indeed gives an embedding of algebras.

\subsection{Rational Demazure-Lusztig operators}\label{subsec:DL}

Let us consider the polynomial ring
$\CC[\hbar,\bt,w_1,\dots,w_N]$. Let $s_i^w$ denote the ordinary simple
reflection on $\CC[w_1,\dots,w_N]$ for $i\neq 0$, and $s_0^w(w_1) =
w_N - \hbar$, $s_0^w(w_N) = w_1 + \hbar$, $s_0^w w_i = w_i$ ($i\neq 1,N$).
We extend them linearly in $\hbar$, $\bt$. We define a representation
of $\DAHA_N^{\gr}$ as follows: $w_i$ acts by multiplication, $\pi$ is
as above, and
\begin{equation*}
        s_i \mapsto s_i^w + \frac{\bt}{w_i - w_{i+1}}(s_i^w - 1)
        \quad (i\neq 0),
        \qquad
        s_0 \mapsto s_0^w - \frac{\bt}{\hbar + w_1 - w_N}(s_0^w - 1).
\end{equation*}
Note that $\pi s_{N-1}^w \dots s_2^w s_1^w$ is the difference
operator $\sfu_1(w_i) = w_i + \hbar \delta_{i1}$.

It is known that this is a faithful representation \cite[Prop.~1.6.3 (a)]{MR2133033}.
We can prove it by an argument using the notion of leading term, which will be introduced in \secref{sec:cycl-rati-cher}.
The argument is similar to one for the case of double affine Hecke algebra (See \cite[Th.~5.7 and Cor.~5.8]{MR1441642}).

Let us also consider the restriction $\Res$ to the space of symmetric
polynomials in $\vec{w}$. (See \cite[(4.5)]{MR1441642}.) The spherical
subalgebra $\SDAHA_N^{\gr}$ preserves the space of symmetric
polynomials, and $\Res$ gives a faithful representation of
$\SDAHA_N^{\gr}$.

\begin{Example}\label{ex:gl2}
Consider the $N=2$ case. 
We have
\begin{equation*}
    \begin{split}
        X_1 & = \pi s_1
        \begin{NB}
        = \pi s_1^w  + \pi \frac{\bt}{w_1 -
          w_2}(s_1^w- 1)            
        \end{NB}%
        = \left(1 + \frac{\bt}{w_2 - w_1 - \hbar}\right)\sfu_1 
        - \frac{\bt}{w_2 - w_1 - \hbar} \pi,
        \\
        X_2 & = s_1 X_1 s_1
        \begin{NB}
            = s_1\pi
        \end{NB}%
        = \left(1 + \frac{\bt}{w_1 - w_2}\right) \sfu_2
        - \frac{\bt}{w_1 - w_2}\pi.
    \end{split}
\end{equation*}
\begin{NB}
    Note $\pi s_1^w = \sfu_1$, $s_1^w\pi = \sfu_2$.
\end{NB}%
Note $\Res\pi = \sfu_1$ as $\pi = \sfu_1 s_1^w$. Therefore
\begin{equation*}
    \Res X_1 + \Res X_2
    = \left(1 - \frac{\bt}{w_1 - w_2}\right)\sfu_1
      + \left(1 + \frac{\bt}{w_1 - w_2}\right)\sfu_2.
\end{equation*}
This is nothing but $E_1[1]$ in \eqref{eq:79}.

We also have
\begin{equation*}
    \begin{split}
        X_1^{-1} &= s_1\pi^{-1} 
        = \left(1 + \frac{\bt}{w_1 - w_2}\right) \sfu_1^{-1}
        - \frac{\bt}{w_1 - w_2} \pi^{-1},
        \\
        X_2^{-1} &= \pi^{-1} s_1
        = \left(1 + \frac{\bt}{w_2 - \hbar - w_1}\right)\sfu_2^{-1}
        - \frac{\bt}{w_2 - \hbar - w_1} \pi^{-1}.
    \end{split}
\end{equation*}
Since $\Res\pi^{-1} = \sfu_2^{-1}$, we get
\begin{equation*}
    \Res X_1^{-1} + \Res X_2^{-1} = 
    \left(1 + \frac{\bt}{w_1 - w_2}\right)\sfu_1^{-1}
      + \left(1 - \frac{\bt}{w_1 - w_2}\right)\sfu_2^{-1}.
\end{equation*}
This is nothing but $F_1[1]$ in \eqref{eq:79} with $l=0$. By the discussion in \secref{sec:Poisson}, we thus proved that $\cAh\cong\SDAHA_{N}^{\operatorname{gr}}$ for $N=2$, $l=0$. This proof is nothing but the detail of one in \secref{sec:no_frame}, as we have just computed Macdonald operators explicitly.

We also have
\begin{equation*}
    \begin{split}
        X_1^{-1}w_1 &=
        \begin{NB}
            {\scriptscriptstyle
        \left(1 + \frac{\bt}{w_1 - w_2}\right) \sfu_1^{-1} w_1
        - \frac{\bt}{w_1 - w_2} \pi^{-1} w_1
        =}
        \end{NB}%
        \left(1 + \frac{\bt}{w_1 - w_2}\right)
        (w_1 - \hbar) \sfu_1^{-1}
        - \frac{\bt}{w_1 - w_2} (w_2 - \hbar)\pi^{-1},
        \\
        X_2^{-1}w_2 &= 
        \begin{NB}
            {\scriptscriptstyle
        \left(1 + \frac{\bt}{w_2 - w_1 - \hbar}\right)\sfu_2^{-1} w_2
        - \frac{\bt}{w_2 - w_1 - \hbar} \pi^{-1} w_2 = }
        \end{NB}%
        \left(1 + \frac{\bt}{w_2 - w_1 - \hbar}\right)(w_2 - \hbar)\sfu_2^{-1}
        - \frac{\bt}{w_2 - w_1 - \hbar} w_1 \pi^{-1}.
    \end{split}
\end{equation*}
Hence
\begin{equation*}
    \Res X_1^{-1} w_1 + \Res X_2^{-1} w_2
    = \left(1 + \frac{\bt}{w_1 - w_2}\right)(w_1 - \hbar)\sfu_1^{-1}
    + \left(1 - \frac{\bt}{w_1 - w_2}\right)(w_2 - \hbar)\sfu_2^{-1}
    + \bt \sfu_2^{-1}.
\end{equation*}
Then $X_2^{-1} s_1 = \pi^{-1}$, hence
\begin{equation*}
    \Res X_1^{-1} w_1 + \Res X_2^{-1} (w_2 - \bt s_1) 
    = \left(1 + \frac{\bt}{w_1 - w_2}\right)(w_1 - \hbar)\sfu_1^{-1}
    + \left(1 - \frac{\bt}{w_1 - w_2}\right)(w_2 - \hbar)\sfu_2^{-1}.
\end{equation*}
This is nothing but $F_1[1]$ in \eqref{eq:79} with $l=1$, $z_k =
0$. Under Suzuki's embedding \eqref{eq:89}, we have $y_1 = X_1^{-1}
w_1$, $y_2 = X_2^{-1}(w_2 - \bt s_1)$. Therefore $\Res y_1 + \Res y_2
= F_1[1]$.

By the discussion in \secref{sec:Poisson}, we thus proved that
$\cAh\cong\SDAHA_{N}^{\operatorname{rat}}$ for $N=2$, $l=1$.
\end{Example}


%% file: cyclotomic.tex

\section{Cyclotomic rational Cherednik algebras}\label{sec:cycl-rati-cher}

Let $\Gamma_N = \mathfrak{S}_N \ltimes (\ZZ / l\ZZ)^N$ be the wreath product of the symmetric group and the cyclic group of order $l$.
Denote a fixed generator of the $i$th factor of $(\ZZ / l\ZZ)^N$ by $\alpha_i$.
The group $\Gamma_N$ acts on $\CC\langle\xi_1,\dots,\xi_N, \linebreak[2] \eta_1,\dots, \eta_N\rangle$ by
\begin{gather*}
    \alpha_i(\xi_i) = \ve \xi_i,\ \alpha_i(\xi_j) = \xi_j,\ \alpha_i(\eta_i) = \ve^{-1} \eta_i,\ \alpha_i(\eta_j) = \eta_j \quad (i \neq j),
\end{gather*}
with the obvious $\mathfrak{S}_N$-action.
Here $\ve$ denotes a primitive $l$th root of unity.

The \emph{cyclotomic rational Cherednik algebra} $\DAHA_{N,l}^{\operatorname{cyc}}$ for $\gl_N$ is the quotient of the algebra $\CC[\hbar,\bt, c_1, \dots, c_{l-1}]\langle\xi_1,\dots,\xi_N, \linebreak[2] \eta_1,\dots, \eta_N\rangle \rtimes \Gamma_N$ by the relations
\begin{subequations}
\begin{gather}
    [\xi_i, \xi_j] = 0 = [\eta_i, \eta_j] \quad (i,j=1,\dots, N),
\\
    [\eta_i, \xi_j] =
    \begin{cases}
        -\hbar + \bt \sum_{k\neq i} \sum_{m=0}^{l-1} s_{ik} \alpha_i^m \alpha_k^{-m} + \sum_{m=1}^{l-1} c_m \alpha_i^m & \text{if $i=j$}, \\
        - \bt \sum_{m=0}^{l-1} s_{ij} \ve^m \alpha_i^m \alpha_j^{-m} & \text{if $i\neq j$}.
    \end{cases}
\end{gather}
\end{subequations}
Let $e_{\Gamma_N}$ be the idempotent for the group $\Gamma_N = \mathfrak{S}_N \ltimes (\ZZ / l\ZZ)^N$.
The spherical part of the cyclotomic rational Cherednik algebra $\DAHA_{N,l}^{\operatorname{cyc}}$ is defined as
\[
	\SDAHA_{N,l}^{\operatorname{cyc}} = e_{\Gamma_N} \DAHA_{N,l}^{\operatorname{cyc}} e_{\Gamma_N}.
\]

Let $e_{\Gamma}$ be the idempotent for the group $(\ZZ / l \ZZ)^N$, that is, 
\begin{equation*}
	e_{\Gamma} = \dfrac{1}{l^N} \sum_{g \in (\ZZ / l \ZZ)^N} g = \dfrac{1}{l^N} \prod_{i = 1}^N \left( \sum_{m_i=0}^{l-1} \alpha_i^{m_i} \right).
\end{equation*}
Oblomkov \cite{MR2318640} introduced an embedding $\DAHA_N^{\operatorname{rat}} \to \left( e_{\Gamma} \DAHA^{\operatorname{cyc}}_{N,l} e_{\Gamma} \right)_{(\xi_i^{l})}$ given by
\begin{gather*}
    x_i \mapsto e_{\Gamma} \xi_i^l e_{\Gamma} \quad (i=1,\dots, N),\\
    y_i \mapsto l^{-1} e_{\Gamma} \xi_i^{1-l} \eta_i e_{\Gamma} \quad (i=1,\dots, N).
\end{gather*}
This induces an embedding $\iota \colon e_{\Gamma} \DAHA^{\operatorname{cyc}}_{N,l} e_{\Gamma} \to \DAHA^{\gr}_N [c_1,\ldots,c_{l-1}]$ such that
\begin{gather*}
    \iota(e_{\Gamma} \xi_i^l e_{\Gamma}) = X_i \quad (i=1,\dots, N),\\
    \iota(e_{\Gamma} \xi_i \eta_i e_{\Gamma}) = l \left( w_i - \bt \sum_{j < i} s_{ji} \right)
    \quad (i=1,\dots, N).
\end{gather*}
When $l=1$, this coincides with Suzuki's embedding \eqref{eq:89}.
We restrict $\iota$ to the spherical part and obtain the embedding $\iota \colon \SDAHA^{\operatorname{cyc}}_{N,l} \to \SDAHA^{\gr}_N[c_1,\ldots,c_{l-1}]$.

Fix $i=1, \dots, N$ and define for each $k=0,1, \dots, l-1$
\begin{equation*}
	e_{\Gamma}[k] = \dfrac{1}{l^N} \left( \sum_{m_i=0}^{l-1} (\ve^{-k} \alpha_i)^{m_i} \right) \prod_{j \neq i} \left( \sum_{m_j=0}^{l-1} \alpha_j^{m_j} \right).
\end{equation*}
We have
\begin{equation*}
	\alpha_i^{m} e_{\Gamma}[k] = \ve^{km} e_{\Gamma}[k], \quad 
	\alpha_j^{m} e_{\Gamma}[k] = e_{\Gamma}[k] \quad (j \neq i),
\end{equation*}
and hence
\begin{equation*}
	e_{\Gamma}[k] e_{\Gamma} = e_{\Gamma} e_{\Gamma}[k] = 0 \quad (k \neq 0), \quad
	e_{\Gamma}[k]^2 = e_{\Gamma}[k].
\end{equation*}
We also have
\begin{equation*}
	\xi_i e_{\Gamma}[k] = e_{\Gamma}[k+1] \xi_i, \quad
	e_{\Gamma}[k] \eta_i = \eta_i e_{\Gamma}[k+1].
\end{equation*}
These imply
\begin{equation*}
	\iota(e_{\Gamma} \xi_i^k e_{\Gamma}) = \iota(e_{\Gamma} \eta_i^k e_{\Gamma}) = 0
\end{equation*}
for $k=1, \dots, l-1$.

\begin{Lemma}\label{lem:image}
For $k=1, \dots, l$, we have
\begin{multline*}
    		\iota(e_{\Gamma} \eta_i^k \xi_i^k e_{\Gamma}) = ( l w_i - \hbar + \sum_{m=1}^{l-1} c_m + l \bt \sum_{i<j} s_{ij} ) ( l w_i - 2\hbar + \sum_{m=1}^{l-1} (1 + \ve^m)c_m + l \bt \sum_{i<j} s_{ij} ) \\
		\cdots ( l w_i - k\hbar + \sum_{m=1}^{l-1} (1 + \ve^m + \dots + \ve^{(k-1)m})c_m + l \bt \sum_{i<j} s_{ij} ).
\end{multline*}
\end{Lemma}

\begin{proof}
For $k=1$, we have
\begin{equation*}
	\begin{split}
    		\iota(e_{\Gamma} \eta_i \xi_i e_{\Gamma}) &= \iota(e_{\Gamma} ( \xi_i \eta_i - [\xi_i, \eta_i] ) e_{\Gamma}) \\
		&= l \left( w_i - \bt \sum_{j < i} s_{ji} \right) + \left( - \hbar + l \bt \sum_{j \neq i} s_{ij} + \sum_{m=1}^{l-1} c_m \right) \\
		&= l w_i - \hbar + \sum_{m=1}^{l-1} c_m + l \bt \sum_{i<j} s_{ij},
	\end{split}
\end{equation*}
hence the assertion holds. 
For $k \geq 1$, we have
\begin{equation*}
	\begin{split}
    		(e_{\Gamma} \eta_i^k \xi_i^k e_{\Gamma}) (e_{\Gamma} \eta_i \xi_i e_{\Gamma}) &= e_{\Gamma} (\eta_i^k \xi_i^k)(\eta_i \xi_i) e_{\Gamma} \\
		&= e_{\Gamma} \eta_i^{k+1} \xi_i^{k+1} e_{\Gamma} + e_{\Gamma} \eta_i^k [\xi_i^k, \eta_i] \xi_i e_{\Gamma}.
	\end{split}
\end{equation*}
We calculate the second term as
\begin{equation*}
	\begin{split}
		e_{\Gamma} \eta_i^k [\xi_i^k, \eta_i] \xi_i e_{\Gamma} &= e_{\Gamma} \eta_i^k ( \sum_{p=1}^{k} \xi_i^{k-p}  [\xi_i, \eta_i] \xi_i^{p-1} ) \xi_i e_{\Gamma} \\
		&= \sum_{p=1}^{k} e_{\Gamma} \eta_i^k \xi_i^{k-p} [\xi_i, \eta_i] e_{\Gamma}[p]  \xi_i^{p} e_{\Gamma} \\
		&= \sum_{p=1}^{k} (\hbar - \sum_{m=1}^{l-1} \varepsilon^{pm} c_m ) e_{\Gamma} \eta_i^k \xi_i^{k} e_{\Gamma}.
	\end{split}
\end{equation*}
Therefore
\begin{equation*}
	\begin{split}
		& \iota(e_{\Gamma} \eta_i^{k+1} \xi_i^{k+1} e_{\Gamma}) = \iota(e_{\Gamma} \eta_i^k \xi_i^k e_{\Gamma}) \iota(e_{\Gamma} \eta_i \xi_i e_{\Gamma}) - \iota(e_{\Gamma} \eta_i^k [\xi_i^k, \eta_i] \xi_i e_{\Gamma}) \\
		=\; & \iota (e_{\Gamma} \eta_i^k \xi_i^k e_{\Gamma}) \left( ( l w_i - \hbar + \sum_{m=1}^{l-1} c_m + l \bt \sum_{i < j} s_{ij} ) - ( k\hbar - \sum_{m=1}^{l-1} (\ve^m + \cdots + \varepsilon^{km} ) c_m ) \right) \\
		=\; & \iota (e_{\Gamma} \eta_i^k \xi_i^k e_{\Gamma}) \left( l w_i - (k+1)\hbar + \sum_{m=1}^{l-1} ( 1 + \varepsilon^m + \cdots + \varepsilon^{km} ) c_m + l \bt \sum_{i < j} s_{ij} \right)
	\end{split}
\end{equation*}
and the assertion follows by induction on $k$.
\end{proof}

By \lemref{lem:image} and $\iota (e_{\Gamma} \xi_i^{l} e_{\Gamma}) = X_i$, we obtain the following.

\begin{Proposition}\label{prop:parameters}
We have
\begin{equation*}
	\iota (e_{\Gamma} (l^{-1} \eta_i)^{l} e_{\Gamma}) = ( w_i - \hbar - z_1 + \bt \sum_{i < j} s_{ij} ) \cdots ( w_i - \hbar - z_l + \bt \sum_{i < j} s_{ij} ) X_i^{-1}
\end{equation*}
where
\begin{equation*}
	z_k = - l^{-1} \Big( (l-k) \hbar + \sum_{m=1}^{l-1} ( 1 + \varepsilon^m + \cdots + \varepsilon^{(k-1)m} ) c_m \Big).
\end{equation*}
\end{Proposition}

Remark that $z_l = 0$.

\begin{Lemma}\label{lem:exchange}
If $i<j$ then the following identities hold in $\DAHA^{\gr}_N$:
\begin{gather}
	s_{ij} w_i = w_j s_{ij} - \bt( 1 + \sum_{k=i+1}^{j-1} s_{ik} s_{kj} ), \label{eq:exchange1} \\
	s_{ij} ( w_i + \bt\sum_{i<k} s_{ik} ) = ( w_j + \bt\sum_{j<k} s_{jk} )s_{ij}. \label{eq:exchange2}
\end{gather}
In particular we have
\begin{equation}\label{eq:exchange3}
	s_{i_{p-1},i_p} s_{i_{p-2},i_{p-1}} \dots s_{i_1,i_2} s_{i_0,i_1} ( w_{i_0} + \bt\sum_{i_0 < k} s_{i_0 k} ) = ( w_{i_p} + \bt\sum_{i_p < k} s_{i_p k} ) s_{i_{p-1},i_p} s_{i_{p-2},i_{p-1}} \dots s_{i_1,i_2} s_{i_0,i_1}
\end{equation}
for $i_0 < i_1 < \dots < i_p$.
\end{Lemma}

\begin{proof}
We prove (\ref{eq:exchange1}) by induction on $j$.
The case $j=i+1$ is nothing but the relation \eqref{eq:88} in $\DAHA^{\gr}_N$.
Assume it holds for $j$.
Then
\begin{equation*}
	\begin{split}
		s_{i,j+1} w_i &= s_j s_{ij} s_j w_i = s_j ( s_{ij} w_i ) s_j \\
		&= s_j ( w_j s_{ij} - \bt ( 1 + \sum_{k=i+1}^{j-1} s_{ik} s_{kj} ) ) s_j \\
		&= ( w_{j+1} s_j - \bt ) s_{ij} s_j - \bt ( 1 + \sum_{k=i+1}^{j-1} s_j s_{ik} s_{kj} s_j ) \\
		&= w_{j+1} s_{i,j+1} - \bt s_{ij} s_j - \bt ( 1 + \sum_{k=i+1}^{j-1} s_{ik} s_{k,j+1} ) \\
		&= w_{j+1} s_{i,j+1} - \bt ( 1 + \sum_{k=i+1}^{j} s_{ik} s_{k,j+1} ).
	\end{split}
\end{equation*}
Then we use (\ref{eq:exchange1}) to prove (\ref{eq:exchange2}):
\begin{equation*}
	\begin{split}
		s_{ij} ( w_i + \bt \sum_{i<k} s_{ik} ) &= w_j s_{ij} - \bt ( 1 + \sum_{k=i+1}^{j-1} s_{ik} s_{kj} ) + \bt \sum_{i < k} s_{ij} s_{ik} \\
&= w_j s_{ij} - \bt ( 1 + \sum_{k=i+1}^{j-1} s_{ik} s_{kj} ) + \bt ( \sum_{k=i+1}^{j-1} s_{ij} s_{ik} + 1 + \sum_{j < k} s_{ij} s_{ik} )\\
&= ( w_j + \bt \sum_{j<k} s_{jk} )s_{ij}. \qedhere
	\end{split}
\end{equation*}
\end{proof}

\begin{Lemma}\label{lem:reorder}
We have
\begin{multline*}
	( w_i - \hbar - z_1 + \bt \sum_{i < j} s_{ij} ) \cdots ( w_i - \hbar - z_l + \bt \sum_{i < j} s_{ij} )\\
	= \sum_{p=0}^{l} \bt^p \sum_{0 = k_0 < k_1 < \dots < k_p \leq l} \ \sum_{i = i_0 < i_1 < \dots < i_p} \prod_{k=k_0+1}^{k_1 - 1} ( w_{i_0} - \hbar - z_k) \prod_{k=k_1 + 1}^{k_2 - 1} ( w_{i_1} - \hbar - z_k) \\
		\dots \prod_{k=k_p + 1}^{l} ( w_{i_p} - \hbar - z_k) s_{i_{p-1},i_p} s_{i_{p-2},i_{p-1}} \dots s_{i_1,i_2} s_{i_0,i_1}.
\end{multline*}
\begin{NB}
  The sum of the right hand side is slightly changed, as the case
  $p=0$ was not clear. It was originally
\(
   \sum_{1\le k_1 < \dots < k_p \le l}
\)
and
\(
 \prod_{k=1}^{k_1 - 1}
\).
See also {\bf NB} below.
\end{NB}%
\end{Lemma}

\begin{proof}
We set $A_k = w_i - \hbar - z_k + \bt \sum_{i < j} s_{ij}$ and
\begin{multline*}
		B_s = \sum_{p=0}^{s} \bt^p \sum_{0 = k_0 < k_1 < \dots < k_p \leq s} \ \sum_{i = i_0 < i_1 < \dots < i_p} \prod_{k=k_0+1}^{k_1 - 1} ( w_{i_0} - \hbar - z_k) \prod_{k=k_1 + 1}^{k_2 - 1} ( w_{i_1} - \hbar - z_k) \\
		\dots \prod_{k=k_p + 1}^{s} ( w_{i_p} - \hbar - z_k) s_{i_{p-1},i_p} s_{i_{p-2},i_{p-1}} \dots s_{i_1,i_2} s_{i_0,i_1}.
\end{multline*}
We prove $A_1 \cdots A_s = B_s$ by induction on $s$.
\begin{NB}
  The definition of $B_s$ is slightly changed. The sum includes
  $k_0=0$. If $p=0$, we have a single term $k_0 = 0$, and the product
  in the second line is $\prod_{k=1}^s (w_{i_0} - \hbar -
  z_k)$. Otherwise the case $p=0$ is not clear.
\end{NB}%
For the case $s=1$, $A_1 = B_1$ follows by definition.
By the induction hypothesis and (\ref{eq:exchange3}) in Lemma~\ref{lem:exchange}, we have
{\allowdisplaybreaks
\begin{equation*}
	\begin{split}
		& A_1 \cdots A_{s+1} \\
		=\; & \Big\{ \sum_{p=0}^{s} \bt^p \sum_{0 = k_0 < k_1 < \dots < k_p \leq s} \ \sum_{i = i_0 < i_1 < \dots < i_p} \prod_{k=k_0+1}^{k_1 - 1} ( w_{i_0} - \hbar - z_k) \prod_{k=k_1 + 1}^{k_2 - 1} ( w_{i_1} - \hbar - z_k) \\
		& \qquad \dots \prod_{k=k_p + 1}^{s} ( w_{i_p} - \hbar - z_k) s_{i_{p-1},i_p} s_{i_{p-2},i_{p-1}} \dots s_{i_1,i_2} s_{i_0,i_1} \Big\} (w_i - \hbar - z_{s+1} + \bt \sum_{i < j} s_{ij}) \\
		=\; & \sum_{p=0}^{s} \bt^p \sum_{0 = k_0 < k_1 < \dots < k_p \leq s} \ \sum_{i = i_0 < i_1 < \dots < i_p} \prod_{k=k_0+1}^{k_1 - 1} ( w_{i_0} - \hbar - z_k) \prod_{k=k_1 + 1}^{k_2 - 1} ( w_{i_1} - \hbar - z_k) \\
		& \qquad \dots \prod_{k=k_p + 1}^{s} ( w_{i_p} - \hbar - z_k) (w_{i_p} - \hbar - z_{s+1} + \bt \sum_{i_p < j} s_{i_p j}) s_{i_{p-1},i_p} s_{i_{p-2},i_{p-1}} \dots s_{i_1,i_2} s_{i_0,i_1} \\
		=\; & B_{s+1}' + B_{s+1}'',
	\end{split}
\end{equation*}
where}
\begin{multline*}
	        B_{s+1}' = \sum_{p=0}^{s} \bt^p \sum_{0 = k_0 < k_1 < \dots < k_p \leq s} \ \sum_{i = i_0 < i_1 < \dots < i_p} \prod_{k=k_0+1}^{k_1 - 1} ( w_{i_0} - \hbar - z_k) \prod_{k=k_1 + 1}^{k_2 - 1} ( w_{i_1} - \hbar - z_k) \\
		\dots \prod_{k=k_p + 1}^{s+1} ( w_{i_p} - \hbar - z_k) s_{i_{p-1},i_p} s_{i_{p-2},i_{p-1}} \dots s_{i_1,i_2} s_{i_0,i_1}
\end{multline*}
and
\begin{multline*}
	        B_{s+1}'' = \sum_{p=0}^{s} \bt^{p+1} \sum_{0 = k_0 < k_1 < \dots < k_p \leq s} \ \sum_{i = i_0 < i_1 < \dots < i_{p+1}} \prod_{k=k_0+1}^{k_1 - 1} ( w_{i_0} - \hbar - z_k) \prod_{k=k_1 + 1}^{k_2 - 1} ( w_{i_1} - \hbar - z_k) \\
		\dots \prod_{k=k_p + 1}^{s} ( w_{i_p} - \hbar - z_k) s_{i_p, i_{p+1}} s_{i_{p-1},i_p} \dots s_{i_1,i_2} s_{i_0,i_1}.
\end{multline*}
The term $B_{s+1}'$ is equal to the contribution of the terms satisfying $k_p < s+1$ to $B_{s+1}$.
The term $B_{s+1}''$ after replacing $p+1$ by $p$ is equal to the contribution of the terms satisfying $k_p = s+1$ to $B_{s+1}$.
Hence the assertion follows. 
\end{proof}

By Lemma~\ref{lem:reorder}, we obtain
\begin{multline}\label{eq:Y_i}
		\iota(e_{\Gamma} (l^{-1}\eta_i)^l e_{\Gamma}) \\
		= \sum_{p=0}^{l} \bt^p \sum_{0=k_0 < k_1 < \dots < k_p \leq l} \ \sum_{i = i_0 < i_1 < \dots < i_p} \prod_{k=k_0+1}^{k_1 - 1} ( w_{i_0} - \hbar - z_k) \prod_{k=k_1 + 1}^{k_2 - 1} ( w_{i_1} - \hbar - z_k) \\
		\dots \prod_{k=k_p + 1}^{l} ( w_{i_p} - \hbar - z_k) X_{i_p}^{-1} s_{i_{p-1},i_p} s_{i_{p-2},i_{p-1}} \dots s_{i_1,i_2} s_{i_0,i_1}.
\end{multline}
Put $Y_i = \iota(e_{\Gamma} (l^{-1}\eta_i)^l e_{\Gamma})$.

Recall the representation of the graded Cherednik algebra $\DAHA^{\gr}_N$ on the polynomial ring $\CC[\hbar,\bt,w_1,\ldots,w_N]$ given in \subsecref{subsec:DL}.
We extend the scalar to $\CC[\hbar,\bt,c_1,\ldots,c_{l-1}]$ and consider the action of $e_{\Gamma} \DAHA^{\operatorname{cyc}}_{N,l} e_{\Gamma}$ on $\CC[\hbar,\bt,c_1,\ldots,c_{l-1}][w_1,\ldots,w_N]$ via the embedding $\iota \colon e_{\Gamma} \DAHA^{\operatorname{cyc}}_{N,l} e_{\Gamma} \to \DAHA^{\gr}_N [c_1, \ldots, c_{l-1}]$.

We introduce the notion of leading term of operators acting on
$\CC[\hbar,\bt,c_1,\ldots,c_{l-1}][w_1,\ldots,w_N]$ as in
\cite[Def.~5.1]{MR1441642}.
Let $P = \bigoplus_{i=1}^N \ZZ \ve_i$ be the weight lattice of $\gl_N$.
Fix positive roots $R^+ = \{ \ve_i - \ve_j \mid i < j\}$ and denote by $P^+$ the set of dominant integral weights.
The Weyl group is denoted by $W = \mathfrak{S}_N$.
For $\lambda, \mu \in P$, we define $\lambda \leq \mu$ if $\mu - \lambda$ is a sum of positive roots with coefficients in $\ZZ_{\geq 0}$.
Let us define another partial order $\lessdot$ on $P$.
Given $\lambda \in P$, we denote by $\lambda^+$ the unique element in $P^+ \cap W \lambda$.
For $\lambda, \mu \in P$, we define $\lambda \lessdot \mu$ if
\begin{gather*}
	\lambda^+ < \mu^+ \text{ or }\\
	\lambda^+ = \mu^+ \text{ and } \lambda > \mu.
\end{gather*}

\begin{Example}
We have $\ve_1 > \ve_2 > \cdots > \ve_{N-1} > \ve_N$ and they are all in the same $W$-orbit.
Hence we have $\ve_1 \lessdot \ve_2 \lessdot \cdots \lessdot \ve_{N-1} \lessdot \ve_N$.
Moreover we see that $\{ \lambda \in P \mid \lambda \lessdot \ve_i \} = \{ \ve_j \mid j < i\}$ since $\ve_1$ is a minuscule dominant weight and $\ve_1, \ldots, \ve_N$ exhaust its $W$-orbit.
Similarly we see that $\{ \lambda \in P \mid \lambda \lessdot - \ve_i \} = \{ - \ve_j \mid j > i\}$.
\end{Example}
We denote by $\sfu^{\lambda} = \sfu_1^{\lambda_1} \cdots \sfu_N^{\lambda}$ for $\lambda = \sum_{i=1}^N \lambda_i \ve_i \in P$.

\begin{Definition}
Let $T$ be an operator on $\CC[\hbar,\bt,c_1,\ldots,c_{l-1}][w_1,\ldots,w_N]$ of the form
\[
	T = \sum_{\lambda \in P, w \in W} g_{\lambda,w} \sfu^{\lambda} w
\]
for some $g_{\lambda,w} = g_{\lambda,w}(w_1, \ldots, w_N) \in \CC[\hbar,\bt,c_1,\ldots,c_{l-1}][w_1,\ldots,w_N][(w_i - w_j)^{-1}]$.
If it can be written as
\[
	T = \sum_{w \in W} g_{\lambda_0,w} \sfu^{\lambda_0} w + \sum_{\lambda \lessdot \lambda_0, w' \in W} g_{\lambda_0,w'} \sfu^{\lambda} w'
\]
for some $\lambda_0$ and at least one of $g_{\lambda_0,w} \neq 0$, we say that $\sum_{w \in W} g_{\lambda_0,w} \sfu^{\lambda_0} w$ is the leading term of $T$.
\end{Definition}

\begin{Example}
In Example \ref{ex:gl2}, we have
\begin{gather*}
\Res X_1 = \sfu_1, \quad \Res X_2 = \left(1 + \frac{\bt}{w_1 - w_2}\right) \sfu_2
        - \frac{\bt}{w_1 - w_2} \sfu_1, \\
\Res X_1^{-1} = \left(1 + \frac{\bt}{w_1 - w_2}\right) \sfu_1^{-1}
        - \frac{\bt}{w_1 - w_2} \sfu_2^{-1}, \quad \Res X_2^{-1} = \sfu_2^{-1}.
\end{gather*}
Therefore the leading terms are
\[
	\sfu_1, \quad \left(1 + \frac{\bt}{w_1 - w_2}\right) \sfu_2, \quad \left(1 + \frac{\bt}{w_1 - w_2}\right) \sfu_1^{-1}, \quad \sfu_2^{-1},
\]
respectively.
We also have
\[
	\Res y_1 = \left(1 + \frac{\bt}{w_1 - w_2}\right)
        (w_1 - \hbar) \sfu_1^{-1}
        - \frac{\bt}{w_1 - w_2} (w_2 - \hbar)\sfu_2^{-1}, \quad
	\Res y_2 = (w_2 - \hbar) \sfu_2^{-1}.
\]
Therefore the leading terms are
\[
	\left(1 + \frac{\bt}{w_1 - w_2}\right)(w_1 - \hbar) \sfu_1^{-1}, \quad
	(w_2 - \hbar) \sfu_2^{-1}
\]
respectively.
\end{Example}

\begin{Proposition}\label{prop:leading_term}
  \textup{(1)}
We have
\begin{equation*}
	\Res X_i = \sum_{j \leq i} g_j \sfu_j
\end{equation*}
for some rational functions $g_j=g_j(w_1, \ldots, w_N)$, hence the leading term of $\Res X_i$ is $g_i \sfu_i$.
Moreover the leading term of $\Res X_N$ is
\begin{equation*}
	\prod_{j \neq N} \dfrac{w_N - w_j - \bt}{w_N - w_j} \sfu_N.
\end{equation*}

\textup{(2)}
We have
\begin{equation*}
	\Res X_i^{-1} = \sum_{j \geq i} g_j' \sfu_j^{-1}
\end{equation*}
for some rational functions $g_j'=g_j'(w_1, \ldots, w_N)$, hence the leading term of $\Res X_i^{-1}$ is $g_i' \sfu_i^{-1}$.
Moreover the leading term of $\Res X_1^{-1}$ is
\begin{equation*}
	\prod_{j \neq 1} \dfrac{w_1 - w_j + \bt}{w_1 - w_j} \sfu_1^{-1}.
\end{equation*}

\textup{(3)}
We have
\begin{equation*}
	\Res Y_i = \sum_{j \geq i} g_j'' \sfu_j^{-1}
\end{equation*}
for some rational functions $g_j''=g_j''(w_1, \ldots, w_N)$, hence the leading term of $\Res Y_i$ is $g_i'' \sfu_i^{-1}$.
Moreover the leading term of $\Res Y_1$ is
\begin{equation*}
	\prod_{k=1}^l (w_1 - \hbar - z_k) \prod_{j \neq 1} \dfrac{w_1 - w_j + \bt}{w_1 - w_j} \sfu_1^{-1}.
\end{equation*}
\end{Proposition}

\begin{proof}
We follow \cite[Lecture 5]{MR1441642} and modify arguments for our degenerate setting.
For an affine root $\alpha = \ve_i - \ve_j + k \delta$, we set
\[
	G(\alpha) = 1 + \dfrac{\bt}{w_i - w_j - k \hbar} (s^w_{\alpha} - 1).
\]
Take $\lambda = \sum_{i=1}^N \lambda_i \ve_i \in P$ and define $X^{\lambda} = X_1^{\lambda_1} \cdots X_N^{\lambda_N}$.
Let $t_{\lambda}$ be the corresponding translation in the extended affine Weyl group.
Let $t_{\lambda} = \pi^m s_{i_r} \cdots s_{i_1}$ be a reduced expression and put $\beta_1 = \alpha_{i_1}, \beta_2 = s_{i_1} \alpha_{i_2}, \ldots, \beta_r = s_{i_1} \cdots s_{i_{r-1}} \alpha_{i_r}$.
Then $X^{\lambda}$ acts on $\CC[\hbar,\bt,c_1,\ldots,c_{l-1}][w_1,\ldots,w_N]$ as an operator $\sfu^{\lambda} G(\beta_r) \cdots G(\beta_1)$.
The assertions (1) and (2) follow from a similar argument as in \cite[Th.~5.6 and Ex.~5.4]{MR1441642}.
Note that $X_N$ and $X_1^{-1}$ correspond to anti-dominant weights $\ve_N$ and $-\ve_1$ respectively.
Hence their leading terms can be calculated explicitly as in \cite[Ex.~5.4]{MR1441642}.
\begin{NB}
Let $\lambda \in P$ be an anti-dominant weight.
Then $\widehat{R}_{t_{\lambda}} = \widehat{R}^+ \cap t_{\lambda}^{-1} \widehat{R}^- = \{ \beta_1, \ldots, \beta_r \}$ only contains roots of the form $\alpha + k\delta$ with $\alpha \in R$ and $k > 0$.
Then the leading term of $\sfu^{\lambda} G(\beta_r) \cdots G(\beta_1)$ is given by
\begin{equation*}
	\sfu^{\lambda} \prod_{\alpha \in \widehat{R}_{t_{\lambda}}} \left( 1 - \dfrac{\bt}{\alpha^{\vee}} \right)
	= \prod_{\alpha \in t_{\lambda} \widehat{R}_{t_{\lambda}}} \left( 1 - \dfrac{\bt}{\alpha^{\vee}} \right) \sfu^{\lambda}.
\end{equation*}
We easily see that $t_{\ve_N} \widehat{R}_{t_{\ve_N}} = \{ \ve_N - \ve_j \mid 1 \leq j \leq N-1\}$ and $t_{-\ve_1} \widehat{R}_{t_{-\ve_1}} = \{ \ve_j - \ve_1 \mid 2 \leq j \leq N\}$.
\end{NB}

We can slightly generalize the assertion (2) (and (1), clearly):
given $w \in W$ we have
\begin{equation*}
	\Res (X_i^{-1}w) = \sum_{j \geq i} g_j' \sfu_j^{-1}
\end{equation*}
for some rational functions $g_j'=g_j'(w_1, \ldots, w_N)$.
Then by \eqref{eq:Y_i}, we see that the term $p=0$ only contributes to the leading term of $\Res Y_i$.
This implies the assertion (3).
\end{proof}

By Proposition~\ref{prop:leading_term} (i), we see that the term containing $\sfu_N$ in $\sum_{i=1}^N \Res X_i$ only comes from $\Res X_N$.
Hence the term is
\[
	\prod_{j \neq N} \dfrac{w_N - w_j - \bt}{w_N - w_j} \sfu_N.
\]
Since $\sum_{i=1}^N \Res X_i$ is $W$-invariant, we conclude that
\[
	\sum_{i=1}^N \Res X_i = \sum_{i=1}^N \prod_{j \neq i} \dfrac{w_i - w_j - \bt}{w_i - w_j} \sfu_i.
\]
This coincides with $E_1[1]$ in \eqref{eq:79}.
By the same argument we conclude that
\[
	\sum_{i=1}^N \Res X_i^{-1} = \sum_{i=1}^N \prod_{j \neq i} \dfrac{w_i - w_j + \bt}{w_i - w_j} \sfu_i^{-1}
\]
and
\[
	\sum_{i=1}^N \Res Y_i = \sum_{i=1}^N \prod_{j \neq i} \dfrac{w_i - w_j + \bt}{w_i - w_j} \prod_{k=1}^l (w_i - \hbar - z_k) \sfu_i^{-1}.
\]
These coincide with $F_1[1]$ for $l=0$ and $l \geq 1$ in \eqref{eq:79} respectively.
Thus we obtain a complete proof of Theorem~\ref{thm:main1} for general $l$.

\begin{NB}
Therefore the leading term of $\Res Y_1$ is given by
\begin{equation*}
	\prod_{k=1}^l (w_1 - \hbar - z_k) \prod_{j \neq 1} \dfrac{w_1 - w_j + \bt}{w_1 - w_j} \sfu_1^{-1}
\end{equation*}
and we conclude that
\begin{equation*}
	\sum_{i=1}^N \Res Y_i = \sum_{i=1}^N \prod_{k=1}^l (w_i - \hbar - z_k) \prod_{j \neq i} \dfrac{w_i - w_j + \bt}{w_i - w_j} \sfu_i^{-1}
\end{equation*}
by the same argument as in the case of $\DAHA^{\gr}_N$.
\end{NB}

\begin{NB}
Consider the case $N=2,\ l=3$ for example.
We have
\begin{align*}
	Y_1 &= \iota(e_{\Gamma} (l^{-1}\eta_1)^l e_{\Gamma}) \\
	&= (w_1 -\hbar - z_1) (w_1 -\hbar + z_2) (w_1 -\hbar - z_3) X_1^{-1} \\
	& \quad + \bt \Big( (w_1 -\hbar - z_1) (w_1 -\hbar - z_2) + (w_1 -\hbar - z_1) (w_2 -\hbar - z_3) + (w_2 -\hbar - z_2) (w_2 -\hbar - z_3) \Big) X_2^{-1} s_1, \\
	&\\
	Y_2 &= \iota(e_{\Gamma} (l^{-1}\eta_2)^l e_{\Gamma}) \\
	&= (w_2 -\hbar - z_1) (w_2 -\hbar - z_2) (w_2 -\hbar - z_3) X_2^{-1}.
\end{align*}
Recall that
\begin{align*}
	\Res X_1^{-1} &= ( 1 + \dfrac{\bt}{w_1 - w_2}) \sfu_1^{-1} - \dfrac{\bt}{w_1 - w_2} \sfu_2^{-1}, \\
	&\\
	\Res X_2^{-1} &= \sfu_2^{-1}.
\end{align*}
Hence we obtain the following:
\begin{align*}
	\Res Y_1 &= (w_1 -\hbar - z_1) (w_1 -\hbar - z_2) (w_1 -\hbar - z_3) ( 1 + \dfrac{\bt}{w_1 - w_2}) \sfu_1^{-1} \\
	& \quad + \bt \Big( (w_1 -\hbar - z_1) (w_1 -\hbar - z_2) (w_1 -\hbar - z_3) \dfrac{-1}{w_1 - w_2} \\
	& \quad + (w_1 -\hbar - z_1) (w_1 -\hbar - z_2) + (w_1 -\hbar - z_1) (w_2 -\hbar - z_3) + (w_2 -\hbar - z_2) (w_2 -\hbar - z_3) \Big) \sfu_2^{-1} \\
	&= (w_1 -\hbar - z_1) (w_1 -\hbar - z_2) (w_1 -\hbar - z_3) ( 1 + \dfrac{\bt}{w_1 - w_2}) \sfu_1^{-1} \\
	& \quad - \dfrac{\bt}{w_1 - w_2} (w_2 -\hbar - z_1) (w_2 -\hbar - z_2) (w_2 -\hbar - z_3) \sfu_2^{-1},
	&\\
	\Res Y_2 &= (w_2 -\hbar - z_1) (w_2 -\hbar - z_2) (w_2 -\hbar - z_3) \sfu_2^{-1}, \\
	&\\
	\Res Y_1 + \Res Y_2 &= (w_1 -\hbar - z_1) (w_1 -\hbar - z_2) (w_1 -\hbar - z_3) ( 1 + \dfrac{\bt}{w_1 - w_2} ) \sfu_1^{-1} \\
	& \quad + (w_2 -\hbar - z_1) (w_2 -\hbar - z_2) (w_2 -\hbar - z_3) ( 1 - \dfrac{\bt}{w_1 - w_2} ) \sfu_2^{-1}.
\end{align*}

When $l=1$ and $N$ is general, we can calculate $\sum_{i_1 > \cdots > i_n} \Res\iota(y_{i_1} \cdots y_{i_n})$ directly.

\begin{Lemma}
Let $i_1 > \cdots > i_n$.
Then we have
\begin{equation*}
	\iota(y_{i_1} \cdots y_{i_n}) = \sum_{r=0}^n \bt^{n-r} \sum_{\substack{I \subset [n] \\ \# I = r }} \Big( \prod_{k \in I} ( w_{i_k} - \hbar ) \Big) \Big( \prod_{k \in I} X_{i_k}^{-1} \Big) \sum_{(j_k) \in \mathcal{J}^{(n)}_I} \Big( \prod_{k \in [n] \setminus I} X_{j_k}^{-1} s_{i_k, j_k} \Big)
\end{equation*}
where
\[
	\mathcal{J}^{(n)}_I = \{ (j_k)_{k \in [n] \setminus I} \mid i_k < j_k \text{ and } j_k \neq i_1, \ldots, i_{k-1} \}
\]
and we take the product in $\prod_{k \in [n] \setminus I} X_{j_k}^{-1} s_{i_k, j_k}$ from small $k$ to large.
\end{Lemma}

\begin{proof}
When a subset $I \subset [n]$ is fixed, we denote the all elements of $I$ by
\[
	q_1 < \dots < q_r
\]
and denote the all elements of $[n] \setminus I$ by
\[
	p_1 < \dots < p_{n-r}.
\]
We set $\chi_{k} = X_{j_k}^{-1} s_{i_k, j_k}$
We prove the assertion by induction on $n$.
The case $n=1$ follows from the definition of $\iota$.
By the induction hypothesis we have
\begin{equation*}
	\begin{split}
		y_{i_1} \cdots y_{i_{n+1}} &= y_{i_1} \cdots y_{i_n} \Big( (w_{i_{n+1}} - \hbar) X_{i_{n+1}}^{-1} + \bt \sum_{i_{n+1} < j_{n+1}} \chi_{n+1} \Big) \\
		&= A + B,
	\end{split}
\end{equation*}
where
\begin{equation*}
	A=\sum_{r=0}^n \sum_{\substack{I \subset [n] \\ \# I = r }} \sum_{(j_k) \in \mathcal{J}^{(n)}_I} \bt^{n-r} \Big( \prod_{k \in I} ( w_{i_k} - \hbar ) \Big) (w_{i_{n+1}} - \hbar) \Big( \prod_{k \in I} X_{i_k}^{-1} \Big) X_{i_{n+1}}^{-1} \Big( \prod_{k \in [n] \setminus I} \chi_k \Big) \\
\end{equation*}
and
\begin{equation*}
	\begin{split}
		B=& \sum_{r=0}^n \sum_{\substack{I \subset [n] \\ \# I = r }} \sum_{(j_k) \in \mathcal{J}^{(n)}_I} \bt^{n+1-r} \prod_{k \in I} ( w_{i_k} - \hbar ) \\
		& \quad \times \Big\{ - \Big( \prod_{k \in I} X_{i_k}^{-1} \Big) \sum_{\alpha = 1}^{n-r} \chi_{p_1} \cdots \chi_{p_{\alpha-1}} ( X_{j_{p_{\alpha}}}^{-1} s_{i_{n+1}, j_{p_{\alpha}}} ) s_{i_{p_{\alpha}},j_{p_{\alpha}}} \chi_{p_{\alpha+1}} \cdots \chi_{p_{n-r}} X_{i_{n+1}}^{-1} \\
		& \qquad - \sum_{\beta = 1}^r X_{i_{q_1}}^{-1} \cdots X_{i_{q_{\beta-1}}}^{-1} ( X_{i_{q_{\beta}}}^{-1} s_{i_{n+1}, i_{q_{\beta}}} ) X_{i_{q_{\beta+1}}}^{-1} \cdots X_{i_{q_n}}^{-1} \Big( \prod_{k \in [n] \setminus I} \chi_{k} \Big) X_{i_{n+1}}^{-1} \\
		& \qquad + \Big( \prod_{k \in I} X_{i_k}^{-1} \Big) \Big( \prod_{k \in [n] \setminus I } \chi_k \Big) \sum_{i_{n+1} < j_{n+1}} \chi_{n+1} \Big\}.
	\end{split}
\end{equation*}
By replacing $I \cup \{n+1\}$ by $I$ in $A$, we have
\begin{equation*}
	A=\sum_{r=1}^{n+1} \sum_{\substack{I \subset [n+1] \\ \# I = r \\ n+1 \in I}} \sum_{(j_k) \in \mathcal{J}^{(n+1)}_I} \bt^{n+1-r} \Big( \prod_{k \in I} ( w_{i_k} - \hbar ) \Big) \Big( \prod_{k \in I} X_{i_k}^{-1} \Big) \Big( \prod_{k \in [n+1] \setminus I} \chi_k \Big).
\end{equation*}
We calculate $\{\}$ in $B$.
The first term: we have
\begin{equation}\label{eq:first}
	\begin{split}
		& - \Big( \prod_{k \in I} X_{i_k}^{-1} \Big) \sum_{\alpha = 1}^{n-r} \chi_{p_1} \cdots \chi_{p_{\alpha-1}} ( X_{j_{p_{\alpha}}}^{-1} s_{i_{n+1}, j_{p_{\alpha}}} ) s_{i_{p_{\alpha}},j_{p_{\alpha}}} \chi_{p_{\alpha+1}} \cdots \chi_{p_{n-r}} X_{i_{n+1}}^{-1} \\
		=\; & - \Big( \prod_{k \in I} X_{i_k}^{-1} \Big) \sum_{\alpha = 1}^{n-r}  \chi_{p_1} \cdots \chi_{p_{\alpha-1}} X_{j_{p_{\alpha}}}^{-1} ( s_{i_{n+1}, j_{p_{\alpha}}} s_{i_{p_{\alpha}},j_{p_{\alpha}}} ) \chi_{p_{\alpha+1}} \cdots \chi_{p_{n-r}} X_{i_{n+1}}^{-1} \\
		=\; & - \Big( \prod_{k \in I} X_{i_k}^{-1} \Big) \sum_{\alpha = 1}^{n-r}  \chi_{p_1} \cdots \chi_{p_{\alpha-1}} X_{j_{p_{\alpha}}}^{-1} ( s_{i_{p_{\alpha}},j_{p_{\alpha}}} s_{i_{n+1}, i_{p_{\alpha}}} ) \chi_{p_{\alpha+1}} \cdots \chi_{p_{n-r}} X_{i_{n+1}}^{-1} \\
		=\; & - \Big( \prod_{k \in I} X_{i_k}^{-1} \Big) \sum_{\alpha = 1}^{n-r}  \chi_{p_1} \cdots \chi_{p_{\alpha-1}} \chi_{p_{\alpha}} s_{i_{n+1}, i_{p_{\alpha}}} \chi_{p_{\alpha+1}} \cdots \chi_{p_{n-r}} X_{i_{n+1}}^{-1}.
	\end{split}
\end{equation}
By the definition of the set $\mathcal{J}_I^{(n)}$, we have $i_{p_{\alpha}} \neq j_{p_{\alpha+1}}, \ldots, j_{p_{n-r}}$ and this implies that $s_{i_{n+1}, i_{p_{\alpha}}}$ commutes with $\chi_{p_{\alpha+1}}, \ldots, \chi_{p_{n-r}}$.
Thus \eqref{eq:first} is
\begin{equation*}
		- \Big( \prod_{k \in I} X_{i_k}^{-1} \Big) \Big( \prod_{k \in [n] \setminus I} \chi_k \Big) \sum_{\alpha = 1}^{n-r} s_{i_{n+1}, i_{p_{\alpha}}} X_{i_{n+1}}^{-1} 
		=- \Big( \prod_{k \in I} X_{i_k}^{-1} \Big) \Big( \prod_{k \in [n] \setminus I} \chi_k \Big) \sum_{\alpha = 1}^{n-r} X_{i_{p_{\alpha}}}^{-1} s_{i_{n+1}, i_{p_{\alpha}}}.
\end{equation*}
The second term: similarly
\begin{equation*}
	\begin{split}
		&- \sum_{\beta = 1}^r X_{i_{q_1}}^{-1} \cdots X_{i_{q_{\beta-1}}}^{-1} ( X_{i_{q_{\beta}}}^{-1} s_{i_{n+1}, i_{q_{\beta}}} ) X_{i_{q_{\beta+1}}}^{-1} \cdots X_{i_{q_n}}^{-1} \Big( \prod_{k \in [n] \setminus I} \chi_{k} \Big) X_{i_{n+1}}^{-1} \\
		=\; & - \Big( \prod_{k \in I} X_{i_k}^{-1} \Big) \Big( \prod_{k \in [n] \setminus I} \chi_{k} \Big) \sum_{\beta = 1}^r s_{i_{n+1}, i_{q_{\beta}}} X_{i_{n+1}}^{-1} \\
		=\; & - \Big( \prod_{k \in I} X_{i_k}^{-1} \Big) \Big( \prod_{k \in [n] \setminus I} \chi_{k} \Big) \sum_{\beta = 1}^r X_{i_{q_{\beta}}}^{-1} s_{i_{n+1}, i_{q_{\beta}}}.
	\end{split}
\end{equation*}
Hence $\{\}$ in $B$ is calculated as
\begin{equation*}
	\begin{split}
		&- \Big( \prod_{k \in I} X_{i_k}^{-1} \Big) \Big( \prod_{k \in [n] \setminus I} \chi_k \Big) \sum_{\alpha = 1}^{n-r} X_{i_{p_{\alpha}}}^{-1} s_{i_{n+1}, i_{p_{\alpha}}} 
		- \Big( \prod_{k \in I} X_{i_k}^{-1} \Big) \Big( \prod_{k \in [n] \setminus I} \chi_{k} \Big) \sum_{\beta = 1}^r X_{i_{q_{\beta}}}^{-1} s_{i_{n+1}, i_{q_{\beta}}} \\
		& \quad + \Big( \prod_{k \in I} X_{i_k}^{-1} \Big) \Big( \prod_{k \in [n] \setminus I } \chi_k \Big) \sum_{i_{n+1} < j_{n+1}} \chi_{n+1} \\
		=\; & \sum_{\substack{i_{n+1} < j_{n+1}\\ j_{n+1} \neq i_1, \ldots, i_n}} \Big( \prod_{k \in I} X_{i_k}^{-1} \Big) \Big( \prod_{k \in [n] \setminus I } \chi_k \Big) \chi_{n+1}
	\end{split}
\end{equation*}
and we obtain
\begin{equation*}
	\begin{split}
		B&= \sum_{r=0}^n \sum_{\substack{I \subset [n] \\ \# I = r }} \sum_{(j_k) \in \mathcal{J}^{(n)}_I} \bt^{n+1-r} \prod_{k \in I} ( w_{i_k} - \hbar ) \sum_{\substack{i_{n+1} < j_{n+1}\\ j_{n+1} \neq i_1, \ldots, i_n}} \Big( \prod_{k \in I} X_{i_k}^{-1} \Big) \Big( \prod_{k \in [n] \setminus I } \chi_k \Big) \chi_{n+1}\\
		&= \sum_{r=0}^n \sum_{\substack{I \subset [n+1] \\ \# I = r\\ n+1 \notin I}} \sum_{(j_k) \in \mathcal{J}^{(n+1)}_I} \bt^{n+1-r} \Big( \prod_{k \in I} (w_{i_k} - \hbar) \Big) \Big(\prod_{k \in I} X_{i_k}^{-1}\Big) \Big( \prod_{k \in [n+1] \setminus I } \chi_k \Big).
	\end{split}
\end{equation*}
The assertion is proved.
\end{proof} 

We reorder $\prod_{k \in [n] \setminus I} X_{j_k}^{-1} s_{i_k, j_k} = X_{j_{p_1}}^{-1} s_{i_{p_1}, j_{p_1}} \cdots X_{j_{p_{n-r}}}^{-1} s_{i_{p_{n-r}}, j_{p_{n-r}}}$.
For each $\alpha = 1, \ldots, n-r$, let $\alpha' \in \{ 1, \ldots, \alpha - 1 \}$ be the maximum number such that $j_{p_{\alpha}} = j_{p_{\alpha'}}$ if it exists.
Set
\[
l_{\alpha} = \begin{cases}
i_{p_{\alpha'}} \quad \text{if $\alpha'$ exists}, \\
j_{p_{\alpha}} \quad \text{if $j_{p_\alpha} \neq j_{p_\alpha - 1}, \ldots, j_{p_1}$}.
\end{cases}
\]
It is easy to see that $i_{q_1}, \ldots, i_{q_r}, l_1, \ldots, l_{n-r}$ are all distinct.
We have
\[
\prod_{k \in [n] \setminus I} X_{j_k}^{-1} s_{i_k, j_k} = X_{l_1}^{-1} \cdots X_{l_{n-r}}^{-1} s_{i_{p_1},j_{p_1}} \cdots s_{i_{p_{n-r}},j_{p_{n-r}}}.
\]
and hence
\begin{equation*}
	\begin{split}
		& \Res\iota(y_{i_1} \cdots y_{i_n}) \\
		=\; & \sum_{r=0}^n \bt^{n-r} \sum_{\substack{I \subset [n] \\ \# I = r }} \Big( \prod_{k \in I} ( w_{i_k} - \hbar ) \Big) \sum_{(j_k) \in \mathcal{J}^{(n)}_I} \Res ( X_{i_{q_1}}^{-1} \cdots X_{i_{q_r}}^{-1} X_{l_1}^{-1} \cdots X_{l_{n-r}}^{-1} s_{i_{p_1},j_{p_1}} \cdots s_{i_{p_{n-r}},j_{p_{n-r}}} ).
	\end{split}
\end{equation*}
The term $r=n$ only contributes to the leading term.
Therefore the leading term of $\Res\iota(y_{1} \cdots y_{n})$ is given by
\begin{equation*}
	\prod_{k=1}^n (w_{k} - \hbar) \prod_{k=1}^n \prod_{j \neq 1, \ldots, n} \dfrac{w_{k} - w_j + \bt}{w_{k} - w_j} \sfu_{1}^{-1} \cdots \sfu_{n}^{-1}
\end{equation*}
and we conclude that
\begin{equation*}
	\sum_{i_1 > \cdots > i_n} \Res\iota(y_{i_1} \cdots y_{i_n}) = \sum_{\substack{I \subset [N]\\ \# I = n}} \prod_{i \in I} \prod_{j \notin I} \dfrac{w_i - w_j + \bt}{w_i - w_j} \prod_{i \in I} (w_i - \hbar) \sfu_i^{-1}
\end{equation*}
by the same argument as in the case of $\DAHA^{\gr}_N$.

\end{NB}


%% file: Yangian.tex
\section{Affine Yangian of \texorpdfstring{$\gl(1)$}{gl(1)}}\label{sec:aff_Yangian}

\subsection{Presentation}

We use the presentation of the affine Yangian $Y(\widehat{\gl}(1))$ in
\cite{MR3077678}, given based on \cite{MR3150250}. See also
\cite{2014arXiv1404.5240T}.
\begin{NB}
    Let $\ve_1$, $\ve_2$ be the equivariant parameters for the usual
    $(\CC^\times)^2$ action on $\CC^2$. Then here are convention in
    \cite{MR3150250}:
    \begin{equation*}
        \kappa = -\frac{\ve_2}{\ve_1}, \quad
        \xi = 1 - \kappa = \frac{\ve_1+\ve_2}{\ve_1}, \quad
        \kappa(\kappa-1) = \frac{\ve_2(\ve_1+\ve_2)}{\ve_1^2}.
    \end{equation*}
\end{NB}%

Let us first prepare some functions. Let $\hbar$, $\bt$ be indeterminate
as before. We set
\begin{gather*}
    G_0(x) = -\log x, \quad G_n(x) = \frac{x^{-n}-1}n  \quad (n\ge 1),
\\
    \varphi_n(x) = \sum_{q=-\hbar,\hbar+\bt,-\bt}
    x^n (G_n(1 - qx) - G_n(1 + qx)),
\\
    \phi_n(x) = x^n G_n(1 - (\hbar+\bt) x).
\end{gather*}

\begin{NB}
    Recall that $\bt = \kappa$, $\hbar = -1$. Therefore $\xi = 1 - \kappa = 1 -
    \bt$. But this should be understood as $-\hbar - \bt$.
\end{NB}

The affine Yangian $Y(\widehat{\gl}(1))$ of $\gl(1)$ is a
$\CC[\omega,\hbar,\bt]$-algebra generated by $D_{0,m}$ ($m\ge 1$),
$e_n$, $f_n$ ($n\ge 0$)
\begin{NB}
The affine Yangian $Y(\widehat{\gl}(1))$ of $\gl(1)$ is generated by
$D_{0,m}$ ($m\ge 1$), $e_n$, $f_n$ ($n\ge 0$) together with a family
of central elements $\bc = (\bc_0,\bc_1,\bc_2,\dots)$ and $\hbar, \bt$
\end{NB}%
with relations
\begin{subequations}\label{eq:1}
\begin{gather}
    [D_{0,m}, D_{0,n}] = 0 \quad (m,n\ge 1), \label{eq:2}\\
    [D_{0,m}, e_n] = - \hbar e_{m+n-1} \quad (m\ge 1, n\ge 0),\label{eq:3}\\
    [D_{0,m}, f_n] = \hbar f_{m+n-1} \quad (m\ge 1, n\ge 0),\label{eq:4}\\
    3[e_2, e_1] - [e_3, e_0] + 
    (\hbar^2 + \bt(\hbar + \bt))[e_1, e_0]
    + \hbar \bt (\hbar + \bt) e_0^2 = 0, \label{eq:e} \\
    3[f_2, f_1] - [f_3, f_0] + 
    (\hbar^2 + \bt(\hbar + \bt))[f_1, f_0]
    - \hbar \bt (\hbar + \bt) f_0^2 = 0, \label{eq:f} \\
    [e_0, [e_0, e_1]] = 0 = [f_0, [f_0, f_1]], \label{eq:tri} \\
    [e_m, f_n] = \hbar \EE_{m+n} \quad (m,n\ge 0), \label{eq:ef}
\end{gather}
where elements $\EE_{m+n}$ are determined through the formula
\begin{equation}\label{eq:DE}
    1 - \bt (\hbar + \bt) \sum_{n\ge 0} \EE_n x^{n+1} 
    = \frac{(1-(\hbar+\bt)x)(1+\omega \bt x)}{1 - (\hbar+(1-\omega)\bt)x}
    \exp\left(-\sum_{n\ge 0} \dfrac{D_{0,n+1}}{\hbar} \varphi_n(x)\right).
\end{equation}
\end{subequations}
Note that the right hand side is $1$ at $\bt = 0$ or
$\hbar+\bt=0$. ($\varphi_n$ vanishes if $\hbar =0$, $\bt=0$ or
$\hbar+\bt=0$.) Therefore $\EE_n$ is a polynomial in $D_{0,m}$
($m\ge 0$) with coefficients in $\CC[\omega,\hbar,\bt]$. For example,
\begin{equation}\label{eq:6}
\begin{gathered}
  h_0 = 0, \quad h_1 = \omega, \quad h_2 = 2D_{0,1} + \omega(\hbar+(1-\omega)\bt),
  \\
  h_{l+1} = l(l+1) D_{0,l} + (\text{lower order term}),
\end{gathered}
\end{equation}
where the lower order term means a polynomial in $\hbar$, $\bt$, $\omega$ and $D_{0,m}$ with $m < l$.
\begin{NB}
\begin{equation*}
    1 - \bt (\hbar + \bt) \sum_{n\ge 0} \EE_n x^{n+1} 
    = \exp\left(\sum_{n\ge 0} (-1)^{n+1} \bc_n \phi_n(x)\right)
    \exp\left(-\sum_{n\ge 0} \dfrac{D_{0,n+1}}{\hbar} \varphi_n(x)\right).
\end{equation*}
\end{NB}%

If we set $\hbar = -1$, $\bt = \kappa$, $\xi = 1 - \kappa = 1 - \bt$,
$D_{1,n} = e_n$, $D_{-1,n} = f_n$, these are the defining relations in
\cite{MR3077678} with $c_0 = 0$, $c_n = -\bt^n \omega^n$ ($n > 0$) as
$\exp(\sum_{n\ge 0} (-1)^{n+1} \bc_n \phi_n(x)) = (1-(\hbar+\bt)x)(1 +
\omega\bt x)/(1 - (\hbar + (1-\omega)\bt)x)$.\footnote{The sign of
  $\hbar\bt(\hbar+\bt)e_0^2$, $\hbar\bt(\hbar+\bt)f_0^2$ are opposite,
  and $\bt$ is missing in the definition of $\EE_n$. We believe that
  they are typo.}

\begin{Remark}
  \begin{NB}
    See `2016-06-23 higher exchange relations and Serre like
    relations.pdf' for detail.
  \end{NB}%
Applying $[D_{0,n+1}, [D_{0,m+1}, \bullet]] + \hbar [D_{0,m+n+1}, \bullet]$ to 
\eqref{eq:f}, we get
\begin{multline*}
  3 [f_{m+2}, f_{n+1}] - 3[f_{m+1}, f_{n+2}] - [f_{m+3}, f_n] + [f_m, f_{n+3}]
\\
  + (\hbar^2 + \bt(\hbar+\bt))([f_{m+1}, f_n] - [f_m, f_{n+1}])
  - \hbar\bt(\hbar + \bt) (f_m f_n + f_n f_m) = 0. \tag{\ref{eq:f}'}
\end{multline*}
\eqref{eq:f} is the special case $m=n=0$.
Similarly \eqref{eq:tri} implies
\begin{multline*}
    [f_i, [f_j, f_{k+1}]] + [f_i, [f_k, f_{j+1}]] + [f_j, [f_i, f_{k+1}]]
\\
    + [f_j, [f_k, f_{i+1}]] + [f_k, [f_i, f_{j+1}]] + [f_k, [f_j, f_{i+1}]] = 0.
\tag{\ref{eq:tri}'}
\end{multline*}
We have the corresponding formula for $e_n$. See \cite{2014arXiv1404.5240T}.
\end{Remark}

\begin{Remark}\label{rem:Tsy_rel}
  \begin{NB}
    See `2016-07-27-Note-15-20.xoj' for detail.
  \end{NB}%
  The relation \eqref{eq:3} is replaced by
  \begin{equation}\label{eq:3'}
  \begin{gathered}
    \begin{aligned}[t]
      & [\EE_{m+3}, e_n] - 3[\EE_{m+2},e_{n+1}] + 3[\EE_{m+1}, e_{n+2}] -
      [\EE_m, e_{n+3}]
      \\
      & \quad - (\hbar^2 + \bt(\hbar + \bt))( [\EE_{m+1},e_n] - [\EE_m,
      e_{n+1}]) - \hbar \bt (\hbar+\bt) (\EE_m e_n + e_n \EE_m) = 0,
    \end{aligned}
    \\
    [\EE_0, e_n] = 0 = [\EE_1, e_n], \quad
    [\EE_2, e_n] = -2\hbar e_n
  \end{gathered}
    \tag{\ref{eq:3}'}
  \end{equation}
  for $m$, $n\ge 0$ in \cite{2014arXiv1404.5240T}. One can check two
  relations are equivalent as follows. First observe that
  (\ref{eq:3'}) is equivalent to
  \begin{equation*}
    h(x) e(y)
    = \left(
      e(y) h(x) \frac{
        (1 - (y^{-1} + \hbar) x)(1 - (y^{-1}+\bt)x)(1 - (y^{-1}-\hbar-\bt)x)
        }{
        (1 - (y^{-1} - \hbar) x)(1 - (y^{-1}-\bt)x)(1 - (y^{-1}+\hbar+\bt)x)
        }
      \right)_+,
  \end{equation*}
  where
  \begin{equation*}
    h(x) = 1 - \bt (\hbar+\bt) \sum_{n=0}^\infty \EE_n x^{n+1},\quad
    e(y) = \sum_{n=0}^\infty e_n y^{n+1}
  \end{equation*}
  and $(\ )_+$ denotes the part with positive powers in $y$. Here
  \(
  \frac
  {
        (1 - (y^{-1} + \hbar) x)(1 - (y^{-1}+\bt)x)(1 - (y^{-1}-\hbar-\bt)x)
        }
        {
        (1 - (y^{-1} - \hbar) x)(1 - (y^{-1}-\bt)x)(1 - (y^{-1}+\hbar+\bt)x)
        }
  \)
  is regarded as an element in $\CC[\hbar,\bt,y^{-1}][[x]]$.
  Then this
  is equivalent to
  \begin{equation*}
    \left[ \log h(x), e(y) \right]
    = \left(e(y) \log
      \frac{
        (1 - (y^{-1} + \hbar) x)(1 - (y^{-1}+\bt)x)(1 - (y^{-1}-\hbar-\bt)x)
        }{
        (1 - (y^{-1} - \hbar) x)(1 - (y^{-1}-\bt)x)(1 - (y^{-1}+\hbar+\bt)x)
        }
      \right)_+.
  \end{equation*}
  Now we observe that
  \begin{equation*}
    \log
    \frac{
      (1 - (y^{-1} + \hbar) x)(1 - (y^{-1}+\bt)x)(1 - (y^{-1}-\hbar-\bt)x)
    }{
      (1 - (y^{-1} - \hbar) x)(1 - (y^{-1}-\bt)x)(1 - (y^{-1}+\hbar+\bt)x)
    }
    = \sum_{n=0}^\infty y^{-n} \varphi_n(x).
  \end{equation*}
  Similarly we have the corresponding relation (\ref{eq:4}') on $\EE_m$
  and $f_n$ equivalent to \eqref{eq:4}. We also have the obvious relation
  \begin{equation*}\label{eq:22}
    [\EE_m, \EE_n] = 0 \quad (m,n\ge 0) \tag{\ref{eq:2}'}
  \end{equation*}
  instead of \eqref{eq:2}. Thus $Y(\widehat{\gl}(1))$ is generated by
  $\EE_{n+2}$, $e_n$, $f_n$ ($n\ge 0$) with relations (\ref{eq:22},
  \ref{eq:3'}, \ref{eq:4}', \ref{eq:e}, \ref{eq:f}, \ref{eq:tri},
  \ref{eq:ef}) and $h_0=0$, $h_1=\omega$. This is the presentation in
  \cite{2014arXiv1404.5240T}.
  \begin{NB}
  In particular, it is independent of
  $\omega$: When we change $\omega$, elements $\EE_n$, $e_n$, $f_n$ are fixed,
  but $D_{0,m}$ is changed according to \eqref{eq:DE}.    
  \end{NB}%
  \begin{NB}
    In particular, it is independent of $\bc$: When we change $\bc$,
    elements $\EE_n$, $e_n$, $f_n$ are fixed, but $D_{0,m}$ is changed
    according to \eqref{eq:DE}.
  \end{NB}%
\end{Remark}

\subsection{From Yangian to difference operators}

Let $(B_n(x))_{n\ge 1}$ be the Bernoulli polynomials:
\begin{enumerate}
      \item $B_n(x+1) - B_n(x) = nx^{n-1}$,
    \item $\int_0^1 B_n(x)dx = 0$.
\end{enumerate}
\begin{NB}
    E.g., $B_1(x)=x-1/2$, $B_2(x) = x^2 - x + 1/6$, $B_3(x) = x^3 - 3/2 x^2 + 1/2 x$, ...
\end{NB}%
We also set $B_0(x)\equiv 1$.
\begin{NB}
    In other words, $\int_x^{x+1} B_n(u) du = x^n$.
\end{NB}%
We then set $\bar B_n(w) = (-\hbar)^n B_n(-{w}/\hbar)/n$ so that $\bar
B_n(w - \hbar) - \bar B_n(w) = - \hbar w^{n-1}$.
\begin{NB}
    HN corrected the definition on June 17.
\end{NB}%

\begin{Theorem}[\protect{\cite{MR3150250}}]\label{thm:from-yang-diff}
    Let $\cAh$ be the quantized Coulomb branch for the quiver gauge
    theory for the Jordan quiver with $\dim V=N$, $\dim W = 0$.
    Then we have a surjective homomorphism of algebras $\Phi\colon
    Y(\widehat{\gl}(1)) \to \cAh$ given by
    \begin{gather*}
        D_{0,m} \mapsto \sum_{i=1}^N 
        \bar B_m(w_i - (N-1)\bt) - \bar B_m(-(i-1)\bt) \quad (m\ge 1),\\
        e_n \mapsto E_1[(w+\hbar-(N-1)\bt)^n], \qquad
        f_{n} \mapsto F_1[(w+\hbar-(N-1)\bt)^n] \quad (n\ge 0),
    \end{gather*}
    where $\omega = N$.
    \begin{NB}
      the parameter is $\bc_0 = 0$, $\bc_n = - \bt^n N^n$.
    \end{NB}%
\end{Theorem}

This result is not new, as $Y(\widehat{\gl}(1))$ is defined as the
limit of $\SDAHA^{\operatorname{gr}}_N$ as $N\to\infty$ in
\cite{MR3150250}. We give a self-contained proof here as the
presentation \eqref{eq:1} was obtained afterwards in \cite{MR3077678}.

\subsection{Proof}\label{subsec:proof}

We check (\ref{eq:2}, \ref{eq:3}, \ref{eq:4}) and (\ref{eq:ef}, \ref{eq:DE}) in this
subsection. A proof of (\ref{eq:e}, \ref{eq:f}, \ref{eq:tri}) will be given in
\secref{sec:app}.

A proof of (\ref{eq:3}', \ref{eq:4}') will be given in
\secref{sec:app2}.  Hence a reader who prefers the presentation in
\remref{rem:Tsy_rel} should read this subsection until \lemref{lem:EF}
for (\ref{eq:2}, \ref{eq:ef}), and then \secref{sec:app}, \secref{sec:app2}.

It is obvious that \eqref{eq:2} is satisfied.

Let us set $\bar w_i = w_i - (N-1)\bt$, $\bar w = w - (N-1)\bt$.

It is clear that we have
\begin{Lemma}\label{lem:BE}
Let $f = f(w)$ be a polynomial in one variable $w$. Then we have
\begin{equation*}
    \left[ \sum_{i=1}^N \bar B_n(\bar w_i), E_1[f] \right] = 
    - \hbar E_1[f (\bar w+\hbar)^{n-1}],
\quad
    \left[ \sum_{i=1}^N \bar B_n(\bar w_i), F_1[f]\right] =
    \hbar F_1[f (\bar w+\hbar)^{n-1}].
\end{equation*}
\end{Lemma}

\begin{NB}
    \begin{proof}
        \begin{equation*}
            \begin{split}
                & \left[\sum_{t=1}^a \bar B_n(w_t-(N-1)\bt), F_{1}[f]\right] =
                \left[\sum_{t=1}^a \bar B_n(w_t-(N-1)\bt), \sum_{r=1}^a f(w_r -
                  \hbar) \prod_{s\neq r}
                  \frac{
                    w_{r} - w_{s} + \bt} {w_{r}-w_{s}} 
                  \sfu_{r}^{-1}\right],
                \\
                =\; & \sum_{r=1}^a f(w_r - \hbar) \prod_{s\neq r}
                \frac{
                  w_{r} - w_{s} + \bt} {w_{r}-w_{s}} (\bar B_n(w_r-(N-1)\bt) -
                \bar B_n(w_r-(N-1)\bt-\hbar)) 
                \sfu_{r}^{-1}
                \\
                =\; & \hbar \sum_{r=1}^a f(w_r - \hbar) (w_r -(N-1)\bt
                )^{n-1} \prod_{s\neq r}
                \frac{
                  w_{r} - w_{s} + \bt} {w_{r}-w_{s}} 
                \sfu_{r}^{-1} = \hbar F_1[f(w-(N-1)\bt + \hbar)^{n-1}].
            \end{split}
        \end{equation*}
        Recall the variable for $f$ is shifted for $F_1[f]$.
        For the first, we use $\bar B_n(w_r) - \bar B_n(w_r+\hbar) =
        -\hbar (w_r+\hbar)^{n-1}$.
    \end{proof}
\end{NB}%

This checks (\ref{eq:3}, \ref{eq:4}).

\begin{NB}
    This relation should correspond to
    \begin{equation*}
        D^{(a)}_{1,n} = [D^{(a)}_{0,n+1}, D^{(a)}_{1,0}]
    \end{equation*}
    in \cite[(1.32)]{MR3150250}.

    However, in \cite[Lem.~1.9]{MR3150250}, a symmetric polynomial
    $B^{(a)}_l\in \CC[\kappa,y_1,\dots,y_a]^{\mathfrak S_a}$ such that
    \begin{equation*}
        B^{(a)}_n(\la_1 - \kappa, \la_2 - 2\kappa, \dots,
        \la_a - a\kappa) = \sum_{i=1}^a
        \sum_{j=1}^{\la_i}\left((j-1) - \kappa(i-1)\right)^n
    \end{equation*}
    for any partition $\la$ with $l(\la)\le a$ is used. It is unique,
    and its existence was proved in an indirect way. Using Bernoulli
    polynomials, we see that
    \begin{equation*}
        B^{(a)}_n(y_1,\dots,y_a) = \frac1{n+1}
        \sum_{i=1}^a \left(B_{n+1}(y_i+\kappa) - B_{n+1}(-\kappa(i-1))\right).
    \end{equation*}
    In fact, as $B_{n+1}(y_i - \kappa (i-1)+ 1) - B_{n+1}(y_i - \kappa(i-1))
    = (n+1)(y_i - \kappa(i-1))^n$, we have
    \begin{equation*}
        \sum_{j=1}^{\la_i} ((j-1) - \kappa(i-1))^n
        = \frac1{n+1}B_{n+1}(\la_i - (i-1)\kappa) -
        \frac1{n+1}B_{n+1}(-\kappa(i-1)).
    \end{equation*}

    Then $D^{(a)}_{0,n}$ is defined as $S B^{(a)}_{n-1}(y_1 -
    a\kappa,y_2 - a\kappa,\dots, y_a - a\kappa)S$. We have
    \begin{equation*}
        D^{(a)}_{0,n} = \frac1{n} S\sum_{i=1}^a
        \left( B_n( y_i - (a-1)\kappa) - B_n(-\kappa(i-1))\right) S.
    \end{equation*}

    Therefore we should probably need to use this instead of $\bar
    B_n(w_r)$....
\end{NB}%

\begin{Remark}
  In \cite{MR3150250} a polynomial denoted by $B^{(N)}_m$ is
  introduced, and $\Phi(D_{0,m})$ is introduced as
  $S B^{(n)}_m(w_1 - N\bt, \dots, w_N - N\bt) S$, where $S$ is the
  complete idempotent. See \cite[Lem.~1.9 and (1.29)]{MR3150250}. One
  can directly check that $B^{(N)}_m$ is given by the Bernoulli
  polynomial, and our definition of $\Phi(D_{0,m})$ coincides with
  \cite{MR3150250}. Also $\Phi(D_{\pm 1,0})$ is defined as
  $\sum_i \operatorname{Res} X_i^{\pm 1}$. See
  \cite[(1.32)]{MR3150250}. Therefore it coincides with our
  $E_1[1] = \Phi(e_0)$, $F_1[1] = \Phi(f_0)$. Together with the
  relations (\ref{eq:3}, \ref{eq:4}), we see that our homomorphism is
  exactly the same as one in \cite{MR3150250}.
\end{Remark}

\begin{NB}
Let us compare operators $F_{1}$ in \eqref{eq:79} for $l > 0$ and
$l=0$. In order to distinguish them, let us denote them by
$F_{1}^{(l)}$, $F_{1}^{(0)}$ respectively. We introduce
\begin{equation*}
    \cB_l^{\vec z}(\vec{w}) \defeq \sum_{r=1}^a \sum_{n=0}^{l} (-1)^n
     e_n(\vec{z}+\hbar) \bar B_{l+1-n}(w_r),
\end{equation*}
where $e_n(\vec{z}+\hbar)$ is the $n$th elementary symmetric
function in variables $z_k+\hbar$ ($k=1,\dots,l$).
From
\(
   \sum_n (-1)^n e_n(\vec{z}+\hbar) w^{l-n} = 
   \prod_k (w - z_k-\hbar),
\)
we get

\begin{Lemma}
\begin{equation*}
    \left[\cB_l^{\vec z}(\vec w), F_{1}^{(0)}\right] 
    = - \hbar F_1^{(0)}[\prod_k (w - z_k-\hbar)] =  - \hbar F_{1}^{(l)}.
\end{equation*}
\end{Lemma}


\end{NB}%

\begin{NB}
\begin{Lemma}
    $[E_1, F_1] = 0$ for $l=0$.
\end{Lemma}

\begin{proof}
Let us start with
\begin{equation*}
    [E_1, F_1] = \sum_{r,t=1}^a \left[
      \prod_{s\neq r} \frac{w_r - w_s - \bt}{w_r - w_s} \sfu_r,
      \prod_{u\neq t} \frac{w_t - w_u + \bt}{w_t - w_u} \sfu_t^{-1}
    \right].
\end{equation*}
The term for $r\neq t$ is
\begin{multline*}
    \prod_{s\neq r}\frac{w_r - w_s - \bt}{w_r - w_s}
    \times \prod_{u\neq r,t} \frac{w_t - w_u + \bt}{w_t - w_u}
    \times \frac{w_t - w_r - \hbar + \bt}{w_t - w_r - \hbar}
    \sfu_r \sfu_t^{-1}
\\
   - \prod_{u\neq t}\frac{w_t - w_u + \bt}{w_t - w_u}
   \times\prod_{s\neq r,t} \frac{w_r - w_s - \bt}{w_r - w_s}
   \times \frac{w_r - w_t + \hbar - \bt}{w_r - w_t + \hbar} \sfu_t^{-1} \sfu_r.
\end{multline*}
Since $\sfu_r \sfu_t^{-1} = \sfu_t^{-1}\sfu_r$, we get
\begin{multline*}
    \prod_{s\neq r,t}\frac{w_r - w_s - \bt}{w_r - w_s}
    \times \prod_{u\neq r,t} \frac{w_t - w_u + \bt}{w_t - w_u}
\\
    \times\left(
      \frac{w_r - w_t - \bt}{w_r - w_t}
      \frac{w_t - w_r - \hbar + \bt}{w_t - w_r - \hbar}
      - \frac{w_t - w_r + \bt}{w_t - w_r}
      \frac{w_r - w_t + \hbar - \bt}{w_r - w_t + \hbar}
    \right)\sfu_r \sfu_t^{-1} = 0.
\end{multline*}

Next consider the sum over $r=t$. We have
\begin{equation*}
    \sum_{r=1}^a \left(\prod_{s\neq r} \frac{w_r - w_s - \bt}{w_r - w_s}
    \frac{w_r + \hbar - w_s + \bt}{w_r + \hbar - w_s}
    - \prod_{s\neq r} \frac{w_r - w_s + \bt}{w_r - w_s}
    \frac{w_r - \hbar - w_s - \bt}{w_r - \hbar - w_s}
    \right).
\end{equation*}
The second product is equal to the first one if we swap $r$ and
$s$. Therefore this is equal to $0$ if we sum over $r$.
\end{proof}
\end{NB}%

In order to check \eqref{eq:ef} we start with the following:
\begin{Lemma}\label{lem:EF}
    For $m,n\ge 0$
    \begin{equation*}
        \left[E_1[(\bar w+\hbar)^m], F_1[(\bar w+\hbar)^n]\right]
        = -
        \frac{\hbar}{\bt(\hbar+\bt)}
        [x^{m+n+1}]
        \left[
          \prod_{i=1}^N \frac{\left(1 - (\bar w_i - \bt)x\right)
            \left(1 - (\bar w_i + \hbar + \bt)x\right)}
          {(1 - \bar w_i x)\left(1 - (\bar w_i + \hbar)x\right)}
        \right],
    \end{equation*}
    where $[x^{m+n+1}]$ denotes the coefficient of $x^{m+n+1}$.
\end{Lemma}

\begin{NB}
\begin{Lemma}
For $m,n\ge 0$
    \begin{multline*}
        [E_1[w^m], F_1[w^n]] \\
        = -
        \frac{\hbar}{\bt(\hbar+\bt)}
        [x^{m+n+l+1}]
        \left[\prod_{k=1}^l (1 - x(z_k+\hbar))
        \prod_{r=1}^a \frac{\left(1 - (w_r + \bt)x\right)
      \left(1 - (w_r - \hbar - \bt)x\right)}
      {(1 - w_rx)\left(1 - (w_r - \hbar)x\right)}-1\right],
    \end{multline*}
    where $[x^{m+n+l+1}]$ denotes the coefficient of $x^{m+n+l+1}$.
\end{Lemma}

\begin{Remark}
    As we only take coefficients of powers of $x$ greater than $l$, we
    can add the right hand side by a polynomial of degree $l$.
\end{Remark}
\end{NB}

\begin{proof}
\begin{NB}
We can assume $l=0$ without loss of generality.
\begin{NB2}
    See 2016-02-09-Note-00-34.xoj.
\end{NB2}%
\end{NB}%
The left hand side is
\begin{equation*}
    \sum_{i,j=1}^N \left[
      (\bar w_i+\hbar)^m \prod_{s\neq i} \frac{\bar w_i - \bar w_s - \bt}
      {\bar w_i - \bar w_s} \sfu_i,
      \bar w_j^n 
      \prod_{u\neq j} \frac{\bar w_j - \bar w_u + \bt}{\bar w_j - \bar w_u}
      \sfu_j^{-1}
    \right].
\end{equation*}
(Note that $w_i - w_s = \bar w_i - \bar w_s$.)
It is easy to check that terms with $i\neq j$ vanish.
\begin{NB}
The term for $i\neq j$ is
\begin{multline*}
    (\bar w_i+\hbar)^m \bar w_j^n 
    \times \prod_{s\neq i}\frac{w_i - w_s - \bt}{w_i - w_s}
    \times \prod_{u\neq i,j} \frac{w_j - w_u + \bt}{w_j - w_u}
    \times \frac{w_j - w_i - \hbar + \bt}{w_j - w_i - \hbar}
    \sfu_i \sfu_j^{-1}
\\
   - (\bar w_i+\hbar)^m \bar w_j^n 
   \times \prod_{u\neq j}\frac{w_j - w_u + \bt}{w_j - w_u}
   \times\prod_{s\neq i,j} \frac{w_i - w_s - \bt}{w_i - w_s}
   \times \frac{w_i - w_j + \hbar - \bt}{w_i - w_j + \hbar} \sfu_j^{-1} \sfu_i.
\end{multline*}
Since $\sfu_i \sfu_j^{-1} = \sfu_j^{-1}\sfu_i$, we get
\begin{multline*}
    (\bar w_i+\hbar)^m \bar w_j^n 
    \times \prod_{s\neq i,j}\frac{w_i - w_s - \bt}{w_i - w_s}
    \times \prod_{u\neq i,j} \frac{w_j - w_u + \bt}{w_j - w_u}
\\
    \times\left(
      \frac{w_i - w_j - \bt}{w_i - w_j}
      \frac{w_j - w_i - \hbar + \bt}{w_j - w_i - \hbar}
      - \frac{w_j - w_i + \bt}{w_j - w_i}
      \frac{w_i - w_j + \hbar - \bt}{w_i - w_j + \hbar}
    \right)\sfu_i \sfu_j^{-1} = 0.
\end{multline*}
\end{NB}%

Next consider the sum over $i=j$. We have
\begin{multline*}
    \sum_{i=1}^N \Bigg(
    (\bar w_i+\hbar)^{m+n} 
    \prod_{s\neq i} \frac{\bar w_i - \bar w_s - \bt}{\bar w_i - \bar w_s}
    \frac{\bar w_i + \hbar - \bar w_s + \bt}{\bar w_i + \hbar - \bar w_s}
\\
    - 
    \bar w_i^{m+n} 
    \prod_{s\neq i} \frac{\bar w_i - \bar w_s + \bt}{\bar w_i - \bar w_s}
    \frac{\bar w_i - \hbar - \bar w_s - \bt}{\bar w_i - \hbar - \bar w_s}
    \Bigg).
\end{multline*}
This expression was appeared in \cite[Cor.~B.6]{MR3150250}. Let us use
the same technique to compute this: Consider 
\begin{multline*}
    \sum_{i=1}^N \Bigg(
    \frac{x\bt}{1-(\bar w_i+\hbar)x} 
    \prod_{s\neq i} \frac{\bar w_i - \bar w_s - \bt}{\bar w_i - \bar w_s}
    \frac{\bar w_i + \hbar - \bar w_s + \bt}{\bar w_i + \hbar - \bar w_s}
\\
    - 
    \frac{x\bt}{1 - \bar w_i x} 
    \prod_{s\neq i} \frac{\bar w_i - \bar w_s + \bt}{\bar w_i - \bar w_s}
    \frac{\bar w_i - \hbar - \bar w_s - \bt}{\bar w_i - \hbar - \bar w_s}
    \Bigg).
\end{multline*}
This is a rational function in $x$, vanishes at $x=0$, regular at
$x=\infty$, and with at most simple poles. Compare it with
\begin{equation*}
    - \frac{\hbar}{\hbar + \bt}\left(
    \prod_{i=1}^N \frac{\left(1 - (\bar w_i - \bt)x\right)
      \left(1 - (\bar w_i + \hbar + \bt)x\right)}
      {(1 - \bar w_i x)\left(1 - (\bar w_i + \hbar)x\right)}-1\right),
\end{equation*}
which has the same properties and the equal residues. Therefore two
are equal thanks to the maximal principle.
\begin{NB}
    Let us check the equality for $N=1$. We have
    \begin{equation*}
        \frac{x\bt}{1 - (w_i+\hbar) x} - \frac{x\bt}{1 - w_i x}
        = \frac{x^2\bt\hbar}{(1 - w_i x)(1 - (w_i + \hbar)x)}.
    \end{equation*}
    On the other hand,
    \begin{multline*}
        \frac{\left(1 - (w_i - \bt)x\right)
          \left(1 - (w_i + \hbar + \bt)x\right)}
        {(1 - w_ix)\left(1 - (w_i + \hbar)x\right)}-1
        = \frac{x^2 \left\{(w_i - \bt)(w_i + \hbar + \bt) - w_i(w_i + \hbar) 
          \right\}}
        {(1 - w_ix)\left(1 - (w_i + \hbar)x\right)}
\\
        = - \frac{x^2 \bt (\hbar + \bt)}
        {(1 - w_ix)\left(1 - (w_i + \hbar)x\right)}.
    \end{multline*}
\end{NB}%
Since we are taking coefficients of $x^{m+n+1}$, we can ignore the
constant term $-1$.
\end{proof}

\begin{NB}
Note
\begin{equation*}
    \begin{split}
    & \prod_{k=1}^l (1 - x (z_k+\hbar)) 
    = (1 -(\hbar+\bt) x)^l 
    \exp\left(\sum_{k=1}^l \log\left(1 - \frac{x(z_k - \bt)}{1 - (\hbar + \bt)x}
      \right)\right)
\\
   =\; & (1 - (\hbar+\bt)x)^l 
   \exp\left(-\sum_{n=1}^\infty \frac1n 
     \sum_{k=1}^l (z_k - \bt)^n \left(\frac{x}{1-(\hbar+\bt)x}\right)^n
     \right).
    \end{split}
\end{equation*}
\end{NB}%

\begin{NB}
Note
\begin{equation*}
    \begin{split}
    & \prod_{k=1}^l (1 - x (z_k+\hbar)) 
    = (1 - x\bt)^l 
    \exp\left(\sum_{k=1}^l \log\left(1 - \frac{x(z_k+\hbar - \bt)}{1 - x\bt}
      \right)\right)
\\
   =\; & (1 - x\bt)^l 
   \exp\left(-\sum_{n=1}^\infty \frac1n 
     \sum_{k=1}^l (z_k+\hbar - \bt)^n \left(\frac{x}{1-x\bt}\right)^n
     \right).
    \end{split}
\end{equation*}
\end{NB}%

\begin{NB}
   \cite[(1.70)]{MR3150250} considered
   \begin{equation*}
       \exp\left(\sum_{n\ge 0} (-1)^{n+1} \mathbf c_n \phi_n(x)\right)
       = (1 + \xi x)^{\mathbf c_0}
       \exp\left[-\sum_{n\ge 1} \frac{(-1)^n}n \mathbf c_n 
         \left\{\left(\frac{x}{1+\xi x}\right)^n - x^n\right\}\right].
   \end{equation*}
   \begin{NB2}
    $\phi_n(x) = x^n G_n(1+\xi x) = (x^n/(1+\xi x)^n - x^n)/n$.
   \end{NB2}
   Since $-x^n$ is missing here, I do not see how to write this term
   in $\phi_n(x)$.
\end{NB}%

Finally let us rewrite
\begin{equation}\label{eq:86}
    \begin{split}
        & \prod_{i=1}^N \frac{\left(1 - (\bar w_i - \bt)x\right)
      \left(1 - (\bar w_i + \hbar + \bt)x\right)}
      {(1 - \bar w_ix)\left(1 - (\bar w_i + \hbar)x\right)}
      \\
      =\; &
      \exp\left[
        \sum_{n=1}^\infty \sum_{i=1}^N \left(
          \bar w_i^n - (\bar w_i - \bt)^n + (\bar w_i + \hbar)^n
          - (\bar w_i + \hbar + \bt)^n
        \right) \frac{x^n}n \right]
    \end{split}
\end{equation}
in terms of the normalized Bernoulli polynomials $\bar B_n(\bar w_i)$. 

We use the following formula for Bernoulli polynomials:
\begin{equation*}
        w^n = \frac1{n+1}\sum_{k=0}^n \binom{n+1}k B_k(w),
\qquad
        B_k(w+v) = \sum_{i=0}^k \binom{k}i B_i(w) v^{k-i}.
\end{equation*}
A direct calculation shows
\begin{multline*}
\sum_{n=1}^\infty (\bar w_i - \bt)^n \frac{x^n}n = 
	1 + \frac{1 + \bt x}{\hbar x} \log\left( 1 + \bt x\right)
	- \frac{1+(\hbar+\bt)x}{\hbar x} \log\left(1 + (\bt+\hbar)x\right) \\
	+ \frac{\bar B_1(\bar w_i)}\hbar\log\left( \frac{1 + x(\hbar +
        \bt)}{1+x\bt}\right) 
	- \sum_{k=2}^\infty \frac{\bar B_k(\bar w_i)}{(k-1)\hbar} \left(\left(\frac{x}{1 + x(\hbar +
	\bt)}\right)^{k-1} - \left(\frac{x}{1+x\bt}\right)^{k-1} \right).
%
%
\end{multline*}
\begin{NB}
\begin{NB2}
\begin{equation*}
    \begin{split}
        & \sum_{n=1}^\infty w_r^n \frac{z^n}n
        = \sum_{n=1}^\infty \frac{(\hbar z)^n}{n(n+1)}
        \sum_{k=0}^{n} \binom{n+1}{k} B_k(\frac{w_r}\hbar)
        = \sum_{n=1}^\infty \frac{(\hbar z)^n}{n(n+1)} - \sum_{k=1}^\infty
        \sum_{n=k}^\infty \frac{\hbar^{n-k} z^n}{n(n+1)} \binom{n+1}{k} 
        k \bar B_k(w_r)
        \\
        =\; & \sum_{n=1}^\infty \frac{(\hbar z)^n}{n(n+1)} 
        (1 - (n+1) \frac{\bar B_1(w_r)}\hbar) 
        - \sum_{k=2}^\infty \frac{\bar B_k(w_r)}{k-1}
        \sum_{n=k}^\infty \binom{n-1}{k-2} z^n \hbar^{n-k}
        \\
        =\; & \sum_{n=1}^\infty \frac{(\hbar z)^n}{n(n+1)} 
        - \sum_{n=1}^\infty \frac{(\hbar z)^n}n \frac{\bar B_1(w_r)}\hbar
        - \sum_{k=2}^\infty \frac{\bar B_k(w_r)}{k-1}
        \sum_{n=1}^\infty \binom{n+k-2}{n} z^{n+k-1} \hbar^{n-1}
        \\
        =\; & \sum_{n=1}^\infty \frac{(\hbar z)^n}{n(n+1)} 
        + \log(1 - \hbar z)\frac{\bar B_1(w_r)}\hbar
        - \sum_{k=2}^\infty \frac{z^{k-1}\bar B_k(w_r)}{(k-1)\hbar}
        \left( \frac1{(1-z\hbar)^{k-1}} - 1 \right)
        \\
        =\; & 1 + \frac{1 - \hbar z}{\hbar z}\log(1 - \hbar z)
        + \log(1 - \hbar z)\frac{\bar B_1(w_r)}\hbar
        - \sum_{k=2}^\infty \frac{z^{k-1}\bar B_k(w_r)}{(k-1)\hbar}
        \left( \frac1{(1-z\hbar)^{k-1}} - 1 \right),
    \end{split}
\end{equation*}
\end{NB2}%
\begin{equation*}
    \begin{split}
        & \sum_{n=1}^\infty w_r^n \frac{z^n}n
        = \sum_{n=1}^\infty \frac{(-\hbar z)^n}{n(n+1)}
        \sum_{k=0}^{n} \binom{n+1}{k} B_k(-\frac{w_r}\hbar)
        = \sum_{n=1}^\infty \frac{(-\hbar z)^n}{n(n+1)} + \sum_{k=1}^\infty
        \sum_{n=k}^\infty \frac{(-\hbar)^{n-k} z^n}{n(n+1)} \binom{n+1}{k} 
        k \bar B_k(w_r)
        \\
        =\; & \sum_{n=1}^\infty \frac{(-\hbar z)^n}{n(n+1)} 
        (1 + (n+1) \frac{\bar B_1(w_r)}{-\hbar}) 
        + \sum_{k=2}^\infty \frac{\bar B_k(w_r)}{k-1}
        \sum_{n=k}^\infty \binom{n-1}{k-2} z^n (-\hbar)^{n-k}
        \\
        =\; & \sum_{n=1}^\infty \frac{(-\hbar z)^n}{n(n+1)} 
        + \sum_{n=1}^\infty \frac{(-\hbar z)^n}n \frac{\bar B_1(w_r)}{-\hbar}
        + \sum_{k=2}^\infty \frac{\bar B_k(w_r)}{k-1}
        \sum_{n=1}^\infty \binom{n+k-2}{n} z^{n+k-1} (-\hbar)^{n-1}
        \\
        =\; & \sum_{n=1}^\infty \frac{(-\hbar z)^n}{n(n+1)} 
        - \log(1 + \hbar z)\frac{\bar B_1(w_r)}{-\hbar}
        + \sum_{k=2}^\infty \frac{z^{k-1}\bar B_k(w_r)}{-(k-1)\hbar}
        \left( \frac1{(1+z\hbar)^{k-1}} - 1 \right)
        \\
        =\; & 1 - \frac{1 + \hbar z}{\hbar z}\log(1 + \hbar z)
        + \log(1 + \hbar z)\frac{\bar B_1(w_r)}\hbar
        - \sum_{k=2}^\infty \frac{z^{k-1}\bar B_k(w_r)}{(k-1)\hbar}
        \left( \frac1{(1+z\hbar)^{k-1}} - 1 \right),
    \end{split}
\end{equation*}
where we have used
\(
    \sum_{n=1}^\infty \nicefrac{x^n}{n(n+1)}
    = 1 + \frac{1-x}{x} \log(1 - x)
\)
at the last equality.

For $(w_r-\bt)^n$, we use
\begin{equation*}
    B_k(x+y) = \sum_{i=0}^k \binom{k}i B_i(x) y^{k-i}.
\end{equation*}
Hence
\begin{equation*}
    \bar B_k(x+y) = (-\hbar)^k \frac{B_k(-\nicefrac{x+y}\hbar)}k
    = \frac{(-\hbar)^k}k \sum_{i=0}^k \binom{k}i B_i(-\nicefrac{x}\hbar)
    \left(\frac{y}{-\hbar}\right)^{k-i}
    = \frac{y^k}k
    + \sum_{i=1}^k \binom{k}i \frac{i}k \bar B_i({x}) y^{k-i}.
\end{equation*}
\begin{NB2}
Therefore
\begin{equation}\label{eq:87-NB2}
    \begin{split}
        & 
        1 + \frac{1-\hbar z}{\hbar z} \log(1 - \hbar z)
        - \sum_{n=1}^\infty (w_r + \bt)^n \frac{z^n}n
        \\
        = \; & - \log(1 - \hbar z)\frac{\bar B_1(w_r
          + \bt)}\hbar + \sum_{k=2}^\infty \frac{z^{k-1} \bar B_k(w_r
          + \bt)}{(k-1)\hbar} 
        \left( \frac1{(1-z\hbar)^{k-1}} - 1 \right)
        \\
        =\; & - \log(1 - \hbar z)\frac{\bar B_1(w_r)
          + \bt}\hbar 
        + \sum_{k=2}^\infty \frac{z^{k-1}}{(k-1)\hbar} 
        \left( - \frac{\bt^k}k + \sum_{i=1}^k \binom{k}i \frac{i}k
          \bar B_i(w_r) \bt^{k-i}\right)
        \left( \frac1{(1-z\hbar)^{k-1}} - 1 \right)
        \\
        =\; & - \log(1 - \hbar z)\frac{\bar B_1(w_r)
          + \bt}\hbar 
        \begin{aligned}[t]
        & - \sum_{k=2}^\infty \frac{z^{k-1}\bt^k}{k(k-1)\hbar}
        \left( \frac1{(1-z\hbar)^{k-1}} - 1 \right)
        \\
        & \qquad + \sum_{k=2}^\infty \frac{1}{(k-1)\hbar}
        \sum_{i=1}^k \binom{k-1}{i-1}
          \bar B_i(w_r) \bt^{k-i}
        \left( \left(\frac{z}{1-z\hbar}\right)^{k-1}
          - z^{k-1} \right).
        \end{aligned}
   \end{split}
\end{equation}
\end{NB2}
Therefore
\begin{equation}\label{eq:87}
    \begin{split}
	& 
        - 1 + \frac{1+\hbar z}{\hbar z} \log(1 + \hbar z)
        + \sum_{n=1}^\infty (w_r - \bt)^n \frac{z^n}n
        \\
        = \; & \log(1 + \hbar z)\frac{\bar B_1(w_r
          - \bt)}\hbar - \sum_{k=2}^\infty \frac{z^{k-1} \bar B_k(w_r
          - \bt)}{(k-1)\hbar} 
        \left( \frac1{(1+z\hbar)^{k-1}} - 1 \right)
        \\
        =\; & \log(1 + \hbar z)\frac{\bar B_1(w_r)
          - \bt}\hbar 
        - \sum_{k=2}^\infty \frac{z^{k-1}}{(k-1)\hbar} 
        \left( \frac{(-\bt)^k}k + \sum_{i=1}^k \binom{k}i \frac{i}k
          \bar B_i(w_r) (-\bt)^{k-i}\right)
        \left( \frac1{(1+z\hbar)^{k-1}} - 1 \right)
        \\
        =\; & \log(1 + \hbar z)\frac{\bar B_1(w_r)
          - \bt}\hbar 
        \begin{aligned}[t]
        & - \sum_{k=2}^\infty \frac{z^{k-1}(-\bt)^k}{k(k-1)\hbar}
        \left( \frac1{(1+z\hbar)^{k-1}} - 1 \right)
        \\
        & \qquad - \sum_{k=2}^\infty \frac{1}{(k-1)\hbar}
        \sum_{i=1}^k \binom{k-1}{i-1}
          \bar B_i(w_r) (-\bt)^{k-i}
        \left( \left(\frac{z}{1+z\hbar}\right)^{k-1}
          - z^{k-1} \right).
        \end{aligned}
   \end{split}
\end{equation}
\begin{NB2}
Note
\begin{equation*}
    \begin{split}
        & - \sum_{k=2}^\infty \frac{z^{k-1}\bt^k}{k(k-1)\hbar} \left(
          \frac1{(1-z\hbar)^{k-1}} - 1 \right) = - \sum_{n=1}^\infty
        \frac{z^{n}\bt^{n+1}}{n(n+1)\hbar} \left(
          \frac1{(1-z\hbar)^{n}} - 1 \right)
        \\
        = \; & - \frac{\bt}{\hbar}\left\{ \frac{1-z\hbar}{z\bt}
          \left(1 - \frac{z\bt}{1 - z\hbar}\right) \log\left(1 -
            \frac{z\bt}{1 - z\hbar}\right) - \frac{1 - z\bt}{z\bt}
          \log\left( 1 - z\bt\right) \right\}
        \\
        =\; &- \frac{1-z(\hbar+\bt)}{z\hbar} 
        \log\left(\frac{1 - z(\bt+\hbar)}{1 - z\hbar}\right) 
        + \frac{1 - z\bt}{z\hbar} \log\left( 1 - z\bt\right)
        \\
        =\; &\frac{1-z(\hbar+\bt)}{z\hbar} 
          \log({1 - z\hbar})
          - \frac{1-z(\hbar+\bt)}{z\hbar} \log({1 - z(\bt+\hbar)}) 
          + \frac{1 - z\bt}{z\hbar} \log\left( 1 - z\bt\right)
    \end{split}
\end{equation*}
\end{NB2}
Note
\begin{equation*}
    \begin{split}
        & \sum_{k=2}^\infty \frac{z^{k-1}(-\bt)^k}{k(k-1)\hbar} \left(
          \frac1{(1+z\hbar)^{k-1}} - 1 \right) = \sum_{n=1}^\infty
        \frac{z^{n}(-\bt)^{n+1}}{n(n+1)\hbar} \left(
          \frac1{(1+z\hbar)^{n}} - 1 \right)
        \\
        = \; & - \frac{\bt}{\hbar}\left\{ \frac{1+z\hbar}{-z\bt}
          \left(1 + \frac{z\bt}{1 + z\hbar}\right) \log\left(1 +
            \frac{z\bt}{1 + z\hbar}\right) + \frac{1 + z\bt}{z\bt}
          \log\left( 1 + z\bt\right) \right\}
        \\
        =\; & \frac{1+z(\hbar+\bt)}{z\hbar} 
        \log\left(\frac{1 + z(\bt+\hbar)}{1 + z\hbar}\right) 
        - \frac{1 + z\bt}{z\hbar} \log\left( 1 + z\bt\right)
        \\
        =\; & - \frac{1+z(\hbar+\bt)}{z\hbar} 
          \log({1 + z\hbar})
          + \frac{1+z(\hbar+\bt)}{z\hbar} \log({1 + z(\bt+\hbar)}) 
          - \frac{1 + z\bt}{z\hbar} \log\left( 1 + z\bt\right)
    \end{split}
\end{equation*}
\begin{NB2}
also
\begin{equation*}
    \begin{split}
        & \sum_{k=2}^\infty \frac{1}{(k-1)\hbar} \sum_{i=1}^k
        \binom{k-1}{i-1} \bar B_i(w_r) \bt^{k-i} \left(
          \left(\frac{z}{1-z\hbar}\right)^{k-1} - z^{k-1} \right)
        \\
        =\; & \frac{\bar B_1(w_r)}\hbar
        \sum_{k=2}^\infty \frac{\bt^{k-1}}{(k-1)} \left(
          \left(\frac{z}{1-z\hbar}\right)^{k-1} - z^{k-1} \right)
        + \sum_{i=2}^\infty \bar B_i(w_r)  \sum_{k=i}^\infty
        \frac{\bt^{k-i}}{(k-1)\hbar} \binom{k-1}{i-1} \left(
          \left(\frac{z}{1-z\hbar}\right)^{k-1} - z^{k-1} \right)
        \\
        =\; & - \frac{\bar B_1(w_r)}\hbar\left\{ \log\left(1 -
            \frac{z\bt}{1-z\hbar}\right) - \log(1- z\bt) \right\}
      +
        \sum_{i=2}^\infty \frac{\bar B_i(w_r)}{(i-1)\hbar}
        \sum_{k=i}^\infty \binom{k-2}{i-2} \bt^{k-i} \left(
          \left(\frac{z}{1-z\hbar}\right)^{k-1} - z^{k-1} \right)
        \\
        =\; & - \frac{\bar B_1(w_r)}\hbar\log\left(
            \frac{1 - z(\hbar + \bt)}{(1-z\hbar)(1-z\bt)}\right)
      +
        \sum_{i=2}^\infty \frac{\bar B_i(w_r)}{(i-1)\hbar}
        \sum_{k=0}^\infty \binom{k+i-2}{k} \bt^k \left(
          \left(\frac{z}{1-z\hbar}\right)^{k+i-1} - z^{k+i-1} \right)
        \\
        =\; & - \frac{\bar B_1(w_r)}\hbar\log\left(
            \frac{1 - z(\hbar + \bt)}{(1-z\hbar)(1-z\bt)}\right)
      +
        \sum_{i=2}^\infty \frac{\bar B_i(w_r)}{(i-1)\hbar}
        \left(
          \left(\frac{z}{1-z\hbar}\right)^{i-1}
          \left(1 - \frac{z\bt}{1-z\hbar}\right)^{1-i}
          - z^{i-1}(1 - z\bt)^{1-i}
          \right)
        \\
        =\; & - \frac{\bar B_1(w_r)}\hbar\log\left(
            \frac{1 - z(\hbar + \bt)}{(1-z\hbar)(1-z\bt)}\right)
      +
        \sum_{k=2}^\infty \frac{\bar B_k(w_r)}{(k-1)\hbar}
        \left(\left(\frac{z}{1 - z(\hbar + \bt)}\right)^{k-1}
          - \left(\frac{z}{1-z\bt}\right)^{k-1}
          \right).
    \end{split}
\end{equation*}
$-\nicefrac{\bar B_1(w_r)}\hbar\log(1 - z\hbar)$ cancels with one in
\eqref{eq:87-NB2}.
\end{NB2}
also
\begin{equation*}
    \begin{split}
        & \sum_{k=2}^\infty \frac{1}{(k-1)\hbar} \sum_{i=1}^k
        \binom{k-1}{i-1} \bar B_i(w_r) (-\bt)^{k-i} \left(
          \left(\frac{z}{1+z\hbar}\right)^{k-1} - z^{k-1} \right)
        \\
        =\; & \frac{\bar B_1(w_r)}\hbar
        \sum_{k=2}^\infty \frac{(-\bt)^{k-1}}{(k-1)} \left(
          \left(\frac{z}{1+z\hbar}\right)^{k-1} - z^{k-1} \right)
        + \sum_{i=2}^\infty \bar B_i(w_r)  \sum_{k=i}^\infty
        \frac{(-\bt)^{k-i}}{(k-1)\hbar} \binom{k-1}{i-1} \left(
          \left(\frac{z}{1+z\hbar}\right)^{k-1} - z^{k-1} \right)
        \\
        =\; & - \frac{\bar B_1(w_r)}\hbar\left\{ \log\left(1 +
            \frac{z\bt}{1+z\hbar}\right) - \log(1 + z\bt) \right\}
      +
        \sum_{i=2}^\infty \frac{\bar B_i(w_r)}{(i-1)\hbar}
        \sum_{k=i}^\infty \binom{k-2}{i-2} (-\bt)^{k-i} \left(
          \left(\frac{z}{1+z\hbar}\right)^{k-1} - z^{k-1} \right)
        \\
        =\; & - \frac{\bar B_1(w_r)}\hbar\log\left(
            \frac{1 + z(\hbar + \bt)}{(1+z\hbar)(1+z\bt)}\right)
      +
        \sum_{i=2}^\infty \frac{\bar B_i(w_r)}{(i-1)\hbar}
        \sum_{k=0}^\infty \binom{k+i-2}{k} (-\bt)^k \left(
          \left(\frac{z}{1+z\hbar}\right)^{k+i-1} - z^{k+i-1} \right)
        \\
        =\; & - \frac{\bar B_1(w_r)}\hbar\log\left(
            \frac{1 + z(\hbar + \bt)}{(1+z\hbar)(1+z\bt)}\right)
      +
        \sum_{i=2}^\infty \frac{\bar B_i(w_r)}{(i-1)\hbar}
        \left(
          \left(\frac{z}{1+z\hbar}\right)^{i-1}
          \left(1 + \frac{z\bt}{1+z\hbar}\right)^{1-i}
          - z^{i-1}(1 + z\bt)^{1-i}
          \right)
        \\
        =\; & - \frac{\bar B_1(w_r)}\hbar\log\left(
            \frac{1 + z(\hbar + \bt)}{(1+z\hbar)(1+z\bt)}\right)
      +
        \sum_{k=2}^\infty \frac{\bar B_k(w_r)}{(k-1)\hbar}
        \left(\left(\frac{z}{1 + z(\hbar + \bt)}\right)^{k-1}
          - \left(\frac{z}{1+z\bt}\right)^{k-1}
          \right).
    \end{split}
\end{equation*}
$-\nicefrac{\bar B_1(w_r)}\hbar\log(1 - z\hbar)$ cancels with one in
\eqref{eq:87}. Hence \eqref{eq:87} is equal to
\begin{multline*}
    \frac{1+z\hbar}{z\hbar} \log(1 + \hbar z)
    - \frac{1 + z(\hbar+\bt)}{z\hbar} \log(1 + z(\bt + \hbar))
    + \frac{1 + z\bt}{z\hbar} \log(1 + z\bt)
\\
    + \frac{\bar B_1(w_r)}{\hbar} \log\left(
      \frac{1 + z(\hbar + \bt)}{1 + z\bt}
      \right)
      - \sum_{k=2}^\infty \frac{\bar B_k(w_r)}{(k-1)\hbar}
        \left(\left(\frac{z}{1 + z(\hbar + \bt)}\right)^{k-1}
          - \left(\frac{z}{1+z\bt}\right)^{k-1}
          \right).
\end{multline*}
\end{NB}%
Taking difference with the same expression with $\bar w_i =
-(i-1)\bt$, we get
\begin{multline}\label{eq:5}
    \sum_{n=1}^\infty (\bar w_i-\bt)^n \frac{x^n}n = 
    - \log(1 + i x \bt)
    + \frac{\bar B_1(\bar w_i) - \bar B_1(-(i-1)\bt)}\hbar 
    \log\left( \frac{1+x(\hbar + \bt)}{1+x\bt}\right) 
    \\
    - \sum_{k=2}^\infty \frac{\bar
      B_k(\bar w_i) - \bar B_k(-(i-1)\bt)}{(k-1)\hbar} 
    \left(\left(\frac{x}{1 + x(\hbar +
          \bt)}\right)^{k-1} - \left(\frac{x}{1+x\bt}\right)^{k-1}
    \right).
\end{multline}

Similarly we have
\begin{equation*}
    \begin{split}
        & \sum_{n=1}^\infty \bar w_i^n \frac{x^n}n =
         - \log(1 + x(i-1)\bt)
    + \frac{\bar B_1(\bar w_i) - \bar B_1(-(i-1)\bt)}\hbar 
    \log\left( 1 + x\hbar\right) 
    \\ & \quad
    - \sum_{k=2}^\infty \frac{\bar
      B_k(\bar w_i) - \bar B_k(-(i-1)\bt)}{(k-1)\hbar} 
    \left(\left(\frac{x}{1 + x\hbar}\right)^{k-1} - x^{k-1}
    \right),
    \\
    & \sum_{n=1}^\infty (\bar w_i + \hbar)^n \frac{x^n}n =
        - \log(1 + x((i-1)\bt - \hbar))
    + \frac{\bar B_1(\bar w_i) - \bar B_1(-(i-1)\bt)}\hbar 
    \log\left( \frac1{1 - x\hbar}\right) 
    \\ & \quad
    - \sum_{k=2}^\infty \frac{\bar
      B_k(\bar w_i) - \bar B_k(-(i-1)\bt)}{(k-1)\hbar} 
    \left( x^{k-1} - \left(\frac{x}{1 - x\hbar}\right)^{k-1}\right),
    \\
    & \sum_{n=1}^\infty (\bar w_i + \hbar + \bt)^n \frac{x^n}n =
        - \log(1 + x((i-2)\bt - \hbar))
    + \frac{\bar B_1(\bar w_i) - \bar B_1(-(i-1)\bt)}\hbar 
    \log\left( \frac{1-x\bt}{1 - x(\hbar+\bt)}\right) 
    \\ & \quad
    - \sum_{k=2}^\infty \frac{\bar
      B_k(\bar w_i) - \bar B_k(-(i-1)\bt)}{(k-1)\hbar} 
    \left( \left(\frac{x}{1 - x\bt}\right)^{k-1} 
      - \left(\frac{x}{1 - x(\hbar+\bt)}\right)^{k-1}\right).
    \end{split}
\end{equation*}
The $\log$ terms give us
\begin{multline*}
    \sum_{i=1}^N \left(
    \log(1 + x i\bt) - \log(1 + x(i-1) \bt)
    - \log(1 + x((i-1)\bt-\hbar)) 
    + \log(1 + x((i-2)\bt-\hbar)) 
    \right)
\\
    = \log(1 + Nx \bt) - \log(1 + x((N-1)\bt-\hbar))
    + \log(1 - x(\bt+\hbar)).
\end{multline*}

\begin{NB}
(The following is wrong, 2016/6/20)
\end{NB}

\begin{NB}
    $K(\kappa,\omega,x)$ in \cite[(1.39)]{MR3150250} with $\xi =
    -(\hbar+\bt)$, $\xi + \kappa\omega = -\hbar + (N-1)\bt$,
    $\kappa\omega = N \bt$ is
   \begin{equation*}
       \frac{(1+\xi x)(1 + \kappa\omega x)}{1 + (\xi+\kappa\omega)x} =
       \frac{(1 - (\hbar+\bt)x)(1 + N\bt x)}{(1 - (\hbar + (1-N)\bt)x}
   \end{equation*}
\begin{NB2}
   This seems to be compatible with $K(\kappa,\omega,s)$ in
   \cite[(1.39)]{MR3150250} by $\xi = -\bt$, $\xi + \kappa\omega = \hbar$,
   $\kappa\omega = \hbar + \bt$. Note
   \begin{equation*}
       \frac{(1 - \bt x)(1 + (\hbar + \bt) x)}{1 + \hbar x}
       = (1 - \bt x)
       \exp\left(\sum_{n\ge 0} (-1)^n (\hbar + \bt)^n \phi_n(x)\right)
       = \exp\left(\sum_{n\ge 0}(-1)^{n+1}\left(\delta_{n,0} 
           - (\hbar + \bt)^n\right) \phi_n(x)\right)
   \end{equation*}
   for $\xi = -\bt$.
\end{NB2}%
\end{NB}%
The alternating sum for $\nicefrac{\bar B_1(\bar w_i) - \bar
  B_1(-(i-1)\bt)}\hbar$ is
\begin{multline*}
    \log (1 - x\hbar) - \log (1 - x(\hbar + \bt)) + \log(1 - x\bt)
\\
    - \log (1 + x\hbar) - \log (1 + x\bt) + \log(1 + x(\hbar + \bt)).
\end{multline*}
This is nothing but $\varphi_0(x)$.
\begin{NB}
    This is $\varphi_0(x)$ if $\hbar = -1$, $\kappa = \bt$, $\xi = 1 -
    \kappa = - (\hbar + \bt)$.
\begin{NB2}
    This is $-\varphi_0(x)$ if $\hbar = 1$, $\kappa = -\bt$, $\xi = 1 -
    \xi = \hbar + \bt$.
\end{NB2}%
\end{NB}%
Finally the alternating sum for $\nicefrac{\bar B_k(\bar w_i) - \bar
  B_k(-(i-1)\bt)}\hbar$ is
\begin{multline*}
    \frac1{k-1}\Biggl[
      -\left(\frac{x}{1 - \hbar x}\right)^{k-1} 
      + \left(\frac{x}{1 - (\hbar + \bt)x}\right)^{k-1} 
      - \left(\frac{x}{1-\bt x}\right)^{k-1}
\\
      + \left(\frac{x}{1+\hbar x}\right)^{k-1}
      + \left(\frac{x}{1+\bt x}\right)^{k-1}
      - \left(\frac{x}{1+(\hbar + \bt)x}\right)^{k-1}
      \Biggr].
\end{multline*}
This is $\varphi_{k-1}(x)$.
\begin{NB}
    This is $\varphi_{k-1}(x)$.
\end{NB}%
Therefore \eqref{eq:86} is equal to
\begin{equation*}
    \frac{(1 - (\hbar+\bt)x)(1 + N\bt x)}{1 - (\hbar + (1-N)\bt)x}
    \exp\left(
      - \sum_{k=1}^\infty \sum_{i=1}^N 
      \frac{\bar B_k(\bar w_i) - \bar B_k(-(i-1)\bt)}{\hbar} \varphi_{k-1}(x)
      \right).
\end{equation*}
This shows \eqref{eq:ef} with $\omega = N$.
\begin{NB}
$\bc_0 = 0$, $\bc_n = -\bt^n N^n$.
\end{NB}

The proof of the remaining relations is given in \secref{sec:app}.


%% file: auto.tex
\subsection{Automorphism}

\begin{NB}
Let us introduce a central element $\omega$ and assume $c_0 = 0$,
$c_n = - \bt^n \omega^n$ in this subsection.
\end{NB}%

Let $a$ be a complex number. We define an automorphism $\tau_a$ of
$Y(\widehat{\gl}(1))$ by
\begin{gather*}
  \tau_a(e_n) = \sum_{k=0}^n \binom{n}{k} a^{n-k} e_k, \qquad
  \tau_a(f_n) = \sum_{k=0}^n \binom{n}{k} a^{n-k} f_k, \\
  \tau_a(\EE_n) = \sum_{k=0}^n \binom{n}{k} a^{n-k} \EE_k.
\end{gather*}
\begin{NB}
  See `2016-06-22 automorphism of Yangian.pdf' for a motivation of
  this definition and checks of relations (\ref{eq:e}, \ref{eq:f},
  \ref{eq:tri}, \ref{eq:ef}). For example,
  \begin{equation*}
    \begin{split}
    & [\tau_a(e_m), \tau_a(f_n)] = \sum_{s=0}^m \sum_{t=0}^n 
    \binom{m}{s} \binom{n}{t} a^{m+n-s-t} [e_s,f_t] \\
    =\; & \sum_{u=0}^{m+n} \sum_{s+t=u} \binom{m}{s}\binom{n}{t}
    a^{m+n-u} \EE_u = \sum_{u=0}^{m+n} \binom{m+n}{u}
    a^{m+n-u} \EE_u = 
    \tau_a(\EE_{m+n}).
    \end{split}
  \end{equation*}
\end{NB}%
A direct computation shows that relations (\ref{eq:e}, \ref{eq:f}, \ref{eq:tri}, \ref{eq:ef}) are preserved.

It is also easy to check (\ref{eq:2}', \ref{eq:3}', \ref{eq:4}') in
\remref{rem:Tsy_rel}. Since those are equivalent to (\ref{eq:2},
\ref{eq:3}, \ref{eq:4}), the automorphism $\tau_a$ is well-defined.

\begin{NB}
We check that $\tau_a$ preserves the relation (\ref{eq:3}').
We have
\begin{equation*}
	[\tau_a(\EE_0), \tau_a(e_n)] = \sum_{k=0}^n \binom{n}{k} a^{n-k} [\EE_0, e_k] = 0,
\end{equation*}
\begin{equation*}
	[\tau_a(\EE_1), \tau_a(e_n)] = \sum_{k=0}^n \binom{n}{k} a^{n-k} [a\EE_0 + \EE_1, e_k] = 0,
\end{equation*}
\begin{equation*}
	[\tau_a(\EE_2), \tau_a(e_n)] = \sum_{k=0}^n \binom{n}{k} a^{n-k} [a^2\EE_0 + a\EE_1 + \EE_2, e_k] = -\hbar \tau_a(e_n).
\end{equation*}
We calculate $[\tau_a(\EE_{m+3}), \tau_a(e_n)]$ as
\begin{equation*}
	\begin{split}
		[\tau_a(\EE_{m+3}), \tau_a(e_n)] &= \sum_{k=0}^{m+3} \sum_{l=0}^{n} \binom{m+3}{k} \binom{n}{l} a^{m+n-k-l+3} [\EE_k, e_l] \\
		&= \sum_{k=0}^{m+3} \sum_{l=0}^{n} \left( \binom{m}{k} + 3\binom{m}{k-1} + 3\binom{m}{k-2} + \binom{m}{k-3} \right) \binom{n}{l} a^{m+n-k-l+3} [\EE_k, e_l].
	\end{split}
\end{equation*}
Similarly we obtain
\begin{equation*}
	[\tau_a(\EE_{m+2}), \tau_a(e_{n+1})] = \sum_{k=0}^{m+2} \sum_{l=0}^{n+1} \left( \binom{m}{k} + 2\binom{m}{k-1} + \binom{m}{k-2} \right) \left( \binom{n}{l} + \binom{n}{l-1} \right) a^{m+n-k-l+3} [\EE_k, e_l],
\end{equation*}
\begin{equation*}
	[\tau_a(\EE_{m+1}), \tau_a(e_{n+2})] = \sum_{k=0}^{m+1} \sum_{l=0}^{n+2} \left( \binom{m}{k} + \binom{m}{k-1} \right) \left( \binom{n}{l} + 2\binom{n}{l-1} + \binom{n}{l-2} \right) a^{m+n-k-l+3} [\EE_k, e_l],
\end{equation*}
\begin{equation*}
	[\tau_a(\EE_{m}), \tau_a(e_{n+3})] = \sum_{k=0}^{m} \sum_{l=0}^{n+3} \binom{m}{k} \left( \binom{n}{l} + 3\binom{n}{l-1} + 3\binom{n}{l-2} + \binom{n}{l-3} \right) a^{m+n-k-l+3} [\EE_k, e_l],
\end{equation*}
\begin{equation*}
	[\tau_a(\EE_{m+1}), \tau_a(e_{n})] = \sum_{k=0}^{m+1} \sum_{l=0}^{n} \left( \binom{m}{k} + \binom{m}{k-1} \right) \binom{n}{l} a^{m+n-k-l+1} [\EE_k, e_l],
\end{equation*}
\begin{equation*}
	[\tau_a(\EE_{m}), \tau_a(e_{n+1})] = \sum_{k=0}^{m} \sum_{l=0}^{n+1} \binom{m}{k} \left( \binom{n}{l} + \binom{n}{l-1} \right) a^{m+n-k-l+1} [\EE_k, e_l].
\end{equation*}
Therefore, applying $\tau_a$ to the left hand side of the first equality in (\ref{eq:3}'), we have
\begin{equation*}
	\begin{split}
		&\sum_{k=0}^m \sum_{l=0}^n \binom{m}{k} \binom{n}{l} a^{m+n-k-l} \Big( [\EE_{k+3}, e_l] - 3[\EE_{k+2},e_{l+1}] + 3[\EE_{k+1}, e_{l+2}] - [\EE_k, e_{l+3}]
      \\
      & \quad - (\hbar^2 + \bt(\hbar + \bt))( [\EE_{k+1},e_l] - [\EE_k,
      e_{l+1}]) - \hbar \bt (\hbar+\bt) (\EE_k e_l + e_l \EE_k) \Big) = 0.
	\end{split}
\end{equation*}
\end{NB}%

Let us give another proof of (\ref{eq:2}, \ref{eq:3}, \ref{eq:4}).
  
\begin{Lemma}
  We have
  \begin{equation*}
    \tau_a(D_{0,m}) \equiv \sum_{k=1}^m \binom{m-1}{k-1} a^{m-k} D_{0,k}
  \end{equation*}
  modulo a central element.
\end{Lemma}

\begin{NB}
  We have $\bt \EE_0 = c_0$,
  $\bt \EE_1 = -c_1 + \frac{c_0(c_0 - 1)}2 (\hbar+\bt)$. Therefore
  $\tau_a(c_0) = c_0$, $\tau_a(c_1) = c_1 - a c_0$. So the assumption
  $c_0 = 0$, $c_n = - \bt^n \omega^n$ seems unavoidable.
\end{NB}%

\begin{proof}
By the binomial theorem, we have
\begin{equation*}
  \tau_a\left(1 - \bt(\hbar+\bt) \sum_{n=0}^\infty \EE_n x^{n+1}\right)
  = 1 - \bt(\hbar+\bt) \sum_{n=0}^\infty \EE_n \left(\frac{x}{1-ax}\right)^{n+1}.
\end{equation*}
On the other hand
\begin{equation*}
   \exp\left(- \sum_{n=0}^\infty \sum_{k=1}^{n+1} \binom{n}{k-1}
     a^{n+1-k} \frac{D_{0,k}}{\hbar} \varphi_n(x)\right)
   = \exp\left(
    - \sum_{n=0}^\infty \frac{D_{0,n+1}}{\hbar} 
    \varphi_n\left(\frac{x}{1-ax}\right)
   \right)
\end{equation*}
follows from the identity
\begin{equation*}
  \sum_{n=k}^\infty \binom{n}{k} a^{n-k} \varphi_n(x) 
  = \varphi_k\left(\frac{x}{1-ax}\right) \qquad (n\ge 0).
\end{equation*}
Therefore 
\begin{equation*}
   C_m \defeq \tau_a(D_{0,m}) - \sum_{k=1}^m \binom{m-1}{k-1} a^{m-k} D_{0,k}
\end{equation*}
is given by
\begin{equation*}
    \frac{(1 - (\hbar+\bt)x)(1 + \omega\bt x)}{1 - (\hbar + (1-\omega)\bt)x}
   \exp\left(-\sum_{n=0}^\infty \frac{C_{n+1}}{\hbar} \varphi_n(x)\right)
   = 
    \frac{(1 - (a+\hbar+\bt)x)(1 - (a - \omega\bt) x)}
    {(1 - ax)(1 - (a + \hbar + (1-\omega)\bt)x)}.
\end{equation*}
\begin{NB}
  \begin{equation*}
    -\sum_{n=0}^\infty \frac{C_{n+1}}{\hbar} \varphi_n(x)
  = \log
    \frac{(1 - (a+\hbar+\bt)x)(1 - (a - \omega\bt) x)
      (1 - (\hbar + (1-\omega)\bt)x)}
    {(1 - ax)(1 - (a + \hbar + (1-\omega)\bt)x)
      (1 - (\hbar+\bt)x)(1 + \omega\bt x)}.
  \end{equation*}
\end{NB}%
Assuming $\omega = N$ is a positive integer, we substitute
$\bar w_i = a - (i-1)\bt$ to \eqref{eq:5} and the subsequent three
equations. We have
\begin{multline*}
  \exp\left(
  -\sum_{k=1}^\infty \sum_{i=1}^N
   \frac{\bar B_k(a - (i-1)\bt) - \bar B_k(-(i-1)\bt)}{\hbar}
   \varphi_{k-1}(x)
   \right)
\\
   = 
    \frac{(1 - (a+\hbar+\bt)x)(1 - (a - N\bt) x)
      (1 - (\hbar + (1-N)\bt)x)}
    {(1 - ax)(1 - (a + \hbar + (1-N)\bt)x)
      (1 - (\hbar+\bt)x)(1 + N\bt x)}.
\end{multline*}
Note that
\(
   \sum_{i=1}^N
   \bar B_k(a - (i-1)\bt) - \bar B_k(-(i-1)\bt)
\)
is a polynomial in $a$, $\hbar$, $\bt$ and $N$. Therefore $C_k$ is the
central element obtained by replacing $N$ by $\omega$.
\end{proof}

\begin{NB}
  $C_1 = a \omega$, $C_2 = -a\bt \omega(\omega - 1) + \omega(a^2 + \hbar a)$.
\end{NB}%

From the proof, we can remove the shift $-(N-1)\bt$ in
\thmref{thm:from-yang-diff}.
\begin{Proposition}
  Let $\cAh$ be the quantized Coulomb branch for $\dim V=N$,
  $\dim W = 0$. Then we have a surjective homomorphism of algebras
  $\Psi\colon Y(\widehat{\gl}(1))\to \cAh$ given by
  \begin{gather*}
    D_{0,m} \mapsto \sum_{i=1}^N
    \bar B_m(w_i) - \bar B_m(-(i-1)\bt) \quad (m\ge 1),\\
    e_n \mapsto E_1[(w+\hbar)^n], \qquad f_{n} \mapsto F_1[(w+\hbar)^n] 
    \quad (n\ge 0),
  \end{gather*}
  where $\omega = N$.
  \begin{NB}
    the parameter is $\bc_0 = 0$, $\bc_n = - \bt^n N^n$.    
  \end{NB}%
\end{Proposition}

In fact, looking at the proof in \subsecref{subsec:proof}, \secref{sec:app},
we find that the argument go through when we use $w_i$, $w$ instead of
$\bar w_i$, $\bar w$. It gives a direct proof without using the
automorphism $\tau_a$.

\subsection{Shifted Yangian}

\begin{NB}
  The formulation in this subsection is different from \cite{kwy} or
  \cite[Appendix]{2016arXiv160403625B} as $f_n^{(l)}$ is usually
  defined simply as $f_{n+l}$. Is it possible to change the image of
  $D_{0,m}$ so that we have a homomorphism from $Y_l [z_1,\dots,z_l]$,
  where $Y_l$ is the subalgebra generated by $D_{0,m}$, $e_n$,
  $f_{n+l}$ ?

  \begin{NB2}
  These are rectified at the end of July 2016.
  \end{NB2}
\end{NB}%

Now we consider the case $\dim W = l > 0$. Let us compare operators
$F_{1}$ in \eqref{eq:79} for $l > 0$ and $l=0$. In order to
distinguish them, let us denote them by $F_{1}^{(l)}$, $F_{1}^{(0)}$
respectively. They are related by
\begin{equation*}
   F_1^{(l)}[(w+\hbar)^n] = F_1^{(0)}[(w+\hbar)^n \prod_{k=1}^l (w - z_k)]
   = \sum_{i=0}^l (-1)^{l-i} e_{l-i}(\vec{z}+\hbar) F_1^{(0)}[(w+\hbar)^{i+n}],
\end{equation*}
where $e_{l-i}(\vec{z}+\hbar)$ is the $(l-i)$th elementary symmetric
function in variables $z_1+\hbar$, \dots, $z_l+\hbar$. Thus the
commutation relations on $F_1^{(l)}[(w+\hbar)^n]$ are deduced from
those on $F_1^{(0)}[(w+\hbar)^n]$.

\begin{NB}
Old version:
  
We consider the subalgebra of $Y(\widehat{\gl}(1))$ generated by
elements $D_{0,m}$ ($m\ge 1$), $e_n$ ($n\ge 0$) and
\begin{equation*}
   f_n^{(l)} \defeq
    \sum_{i=0}^l (-1)^{l-i} e_{l-i}(\vec{z}+\hbar) f_{n+i} \quad (n\ge 0).
\end{equation*}
Let us denote it by $Y_l(\vec{z})$. If all $z_k$ are $-\hbar$,
$f_n^{(l)}$ is $f_{n+l}$. We call $Y_l(\vec{z})$ the \emph{shifted
  Yangian} of $\widehat{\gl}(1)$.

We thus have
\begin{Theorem}\label{thm:shiftedYangian}
  Let $\cAh$ be the quantized Coulomb branch for $\dim V=N$,
  $\dim W = l$. Then we have a surjective homomorphism of algebras
  $\Psi\colon Y_l(\vec{z})\to \cAh$ given by
  \begin{gather*}
    D_{0,m} \mapsto \sum_{i=1}^N
    \bar B_m(w_i) - \bar B_m(-(i-1)\bt) \quad (m\ge 1),\\
    e_n \mapsto E_1[(w+\hbar)^n], \qquad 
    f_{n}^{(l)} \mapsto F_1^{(l)}[(w+\hbar)^n] 
    \quad (n\ge 0),
  \end{gather*}
  where $\bc_0 = 0$, $\bc_n = -\bt^n N^n$.
\end{Theorem}
\end{NB}%

\begin{NB}
tentative 2016/7/29

Let $Y_{l}$ be the subalgebra of $Y(\widehat{\gl}(1))$ generated by $h_n, e_n, f_{n+l}$ $(n \geq 0)$.
We can define the homomorphism from $Y_l$ to the quantized Coulomb branch $\cAh$ with $\dim W = l$ by
\begin{gather*}
    h(x) \mapsto \prod_{k=1}^l (1 - (z_k + \hbar)x) \prod_{i=1}^N \dfrac{(1 - (w_i - \bt)x)(1 - (w_i + \hbar + \bt)x)}{(1 - w_i x)(1 - (w_i + \hbar)x)}, \\
    e_n \mapsto E_1[(w+\hbar)^n], \qquad 
    f_{n+l} \mapsto F_1^{(l)}[(w+\hbar)^n] 
    \quad (n\ge 0).
\end{gather*}
Let $a_1, \ldots, a_N$ be arbitrary elements in $\CC[\hbar,\bt]$.
Then the argument in the proof of \thmref{thm:from-yang-diff} implies
\begin{equation*}
	\begin{split}
		& \prod_{i=1}^N \dfrac{(1 - (w_i - \bt)x)(1 - (w_i + \hbar + \bt)x)}{(1 - w_i x)(1 - (w_i + \hbar)x)} \\
		=\; & \left( \prod_{i=1}^N \dfrac{(1 - (a_i - \bt)x)(1 - (a_i + \hbar + \bt)x)}{(1 - a_i x)(1 - (a_i + \hbar)x)} \right) \exp \left( - \sum_{n \geq 0} \sum_{i=1}^N \dfrac{\bar B_{n+1}(w_i) - \bar B_{n+1}(a_i)}{\hbar} \varphi_{n}(x) \right).
	\end{split}
\end{equation*}
Then we can determine the parameter $\bc$ by
\begin{equation*}
	\exp\left(\sum_{n\ge 0} (-1)^{n+1} \bc_n \phi_n(x)\right) = \prod_{k=1}^l (1 - (z_k + \hbar)x) \prod_{i=1}^N \dfrac{(1 - (a_i - \bt)x)(1 - (a_i + \hbar + \bt)x)}{(1 - a_i x)(1 - (a_i + \hbar)x)}
\end{equation*}
so that the homomorphism defined above sends
\begin{equation*}
	D_{0,n+1} \mapsto \sum_{i=1}^N \bar B_{n+1}(w_i) - \bar B_{n+1}(a_i) \quad (n \geq 0).
\end{equation*}
\end{NB}%

\begin{NB}
  Another approach 2016-07-29 by HN
\end{NB}%

Let $Y_l(\vec{z})$ be a $\CC[\omega,\hbar,\bt]$-algebra generated by
$D_{0,m}$ ($m\ge 1$), $e_n$, $f_{n+l}$ ($n\ge 0$) with relations
\eqref{eq:1} where \eqref{eq:DE} is replaced by
\begin{multline}\label{eq:7}
      1 - \bt (\hbar + \bt) \sum_{n\ge 0} \EE_n x^{n+1} 
      \\
      = \prod_{k=1}^l (1 - (z_k + \hbar)x) \times
      \frac{(1-(\hbar+\bt)x)(1+\omega \bt x)}{1 - (\hbar+(1-\omega)\bt)x}
    \exp\left(-\sum_{n\ge 0} \dfrac{D_{0,n+1}}{\hbar} \varphi_n(x)\right). 
\end{multline}
The right hand side is \emph{not} $1$ at $\bt = 0$ nor $\hbar+\bt =
0$. Nevertheless we only need $h_n$ with $n\ge l$ in \eqref{eq:ef}, and they are well-defined as
$\prod_{k=1}^l (1 - (z_k + \hbar)x)$ is of degree $l$.
\begin{NB}
  Let us write this new $h_n$ by $\tilde h_n$, while $h_n$ is the original one. Then
  \begin{equation*}
    \begin{split}
    1 - \bt(\hbar+\bt) \sum_{n\ge 0} \tilde h_n x^{n+1} &=
    \prod_{k=1}^l (1 - (z_k + \hbar)x)
    \left(1 - \bt(\hbar+\bt) \sum_{n\ge 0} h_n x^{n+1}\right)
    \\
    &= - \left(\sum_{i=0}^l (-1)^i e_i(\vec{z}+\hbar) x^i \right)
        \left(\bt(\hbar+\bt) \sum_{n\ge -1} h_n x^{n+1}\right),
    \end{split}
  \end{equation*}
  where $h_{-1} = - 1/\bt(\hbar+\bt)$. Therefore
  \begin{equation*}
    \tilde h_n = h_n - e_1(\vec{z}+\hbar) h_{n-1} + \cdots + (-1)^l e_l(\vec{z}+\hbar) h_{n-l}
  \end{equation*}
  if $n \ge l$. If $n < l$, we need to use $h_{-1}$, hence has a problem.
\end{NB}%

\begin{Theorem}\label{thm:shiftedYangian}
  Let $\cAh$ be the quantized Coulomb branch for $\dim V=N$,
  $\dim W = l$. Then we have a surjective homomorphism of algebras
  $\Psi\colon Y_l(\vec{z})\to \cAh$ given by
  \begin{gather*}
    D_{0,m} \mapsto \sum_{i=1}^N
    \bar B_m(w_i) - \bar B_m(-(i-1)\bt) \quad (m\ge 1),\\
    e_n \mapsto E_1[(w+\hbar)^n], \qquad 
    f_{n+l} \mapsto F_1^{(l)}[(w+\hbar)^n] 
    \quad (n\ge 0),
  \end{gather*}
  where $\omega = N$.
\end{Theorem}

\begin{Remark}
  Let us switch to the presentation in \remref{rem:Tsy_rel}. Let $Y_l$
  be the subalgebra of $Y(\hat{\gl}(1))$ generated by $h_n$, $e_n$,
  $f_{n+l}$. Then we can `formally' define a homomorphism $Y_l\to \cAh$ by
  \begin{gather*}
    h(x) \mapsto \prod_{k=1}^l (1 - (z_k + \hbar)x) \prod_{i=1}^N
    \dfrac{(1 - (w_i - \bt)x)(1 - (w_i + \hbar + \bt)x)}{(1 - w_i x)(1 - (w_i + \hbar)x)}, \\
    e_n \mapsto E_1[(w+\hbar)^n], \qquad f_{n+l} \mapsto
    F_1^{(l)}[(w+\hbar)^n] \quad (n\ge 0).
  \end{gather*}
  (A proof is given in \secref{sec:app2}.)
  However the target of $h(x)$ is not $1$ at $\bt=0$ nor
  $\hbar+\bt=0$. Therefore the image of $h_n$ ($n < l$) is contained
  in $\frac1{\bt(\hbar+\bt)}\cAh$, but not in $\cAh$. Similarly
  $Y_l(\vec{z})$ is \emph{almost} isomorphic to $Y_l$, which is
  independent of parameters $z_k$, but not quite yet.

  This problem does not arise for finite type shifted Yangian: $Y_l$
  and $Y_l(\vec{z})$ are isomorphic in this case. See \cite[3G]{kwy}.
\end{Remark}


%% file: Yangian2.tex

\appendix
\section{}\label{sec:app}

Let us prove (\ref{eq:f}) for $f_n = F_1[(\bar w + \hbar)^n]$.
Put
\[
	C_i = \prod_{j \neq i} \dfrac{\bar w_i - \bar w_j + \bt}{\bar w_i - \bar w_j} \begin{NB}\prod_{k=1}^l (\bar w_i + (N-1)\bt - \hbar - z_k)\end{NB}
\]
so that $F_1[(\bar w + \hbar)^n] = \sum_{i=1}^N \bar w_i^n C_i \sfu_i^{-1}$.
We define $C_i'$ and $C_{j}^{(i)}$ ($i \neq j$) by
\begin{gather*}
	C_i' = \prod_{j \neq i} \dfrac{\bar w_i - \bar w_j - \hbar + \bt}{\bar w_i - \bar w_j - \hbar} \begin{NB}\prod_{k=1}^l (\bar w_i + (N-1)\bt - 2\hbar - z_k)\end{NB}, \\
	C_{j}^{(i)} = C_j \dfrac{\bar w_j - \bar w_i}{\bar w_j - \bar w_i + \bt} = \prod_{k \neq i,j} \dfrac{\bar w_j - \bar w_k + \bt}{\bar w_j - \bar w_k}
\end{gather*}
so that
\[
	\sfu_i^{-1} C_i = C_i' \sfu_i^{-1}, \quad \sfu_i^{-1} C_j = C_j^{(i)} \dfrac{\bar w_j - \bar w_i + \hbar + \bt}{\bar w_j - \bar w_i + \hbar} \sfu_i^{-1} \ (i \neq j).
\]
We have
\begin{equation*}
	\begin{split}
		[f_m, f_n] &= \sum_{i,j} [\bar w_i^m C_i \sfu_i^{-1}, \bar w_j^n C_j \sfu_j^{-1}] \\
		&= \sum_{i} \left( \bar w_i^m (\bar w_i - \hbar)^n - \bar w_i^n (\bar w_i - \hbar)^m \right) C_i C_i' \sfu_i^{-2} \\
		& \qquad + \sum_{i \neq j} \bar w_i^m \bar w_j^n \left( C_i C_j^{(i)} \dfrac{\bar w_j - \bar w_i + \hbar + \bt}{\bar w_j - \bar w_i + \hbar} - C_j C_i^{(j)} \dfrac{\bar w_i - \bar w_j + \hbar + \bt}{\bar w_i - \bar w_j + \hbar} \right) \sfu_i^{-1} \sfu_j^{-1}.
	\end{split}
\end{equation*}
The second term is calculated as
\begin{equation*}
	\begin{split}
		& C_i C_j^{(i)} \dfrac{\bar w_j - \bar w_i + \hbar + \bt}{\bar w_j - \bar w_i + \hbar} - C_j C_i^{(j)} \dfrac{\bar w_i - \bar w_j + \hbar + \bt}{\bar w_i - \bar w_j + \hbar} \\
		=\; & - C_i^{(j)} C_j^{(i)} \dfrac{2 \hbar \bt (\hbar + \bt)}{(\bar w_i - \bar w_j)(\bar w_i - \bar w_j + \hbar)(\bar w_i - \bar w_j - \hbar)}.
	\end{split}
\end{equation*}
\begin{NB}
\[
	=\; C_i^{(j)} C_j^{(i)} \left( \dfrac{\bar w_i - \bar w_j + \bt}{\bar w_i - \bar w_j} \dfrac{\bar w_j - \bar w_i + \hbar + \bt}{\bar w_j - \bar w_i + \hbar} - \dfrac{\bar w_j - \bar w_i + \bt}{\bar w_j - \bar w_i} \dfrac{\bar w_i - \bar w_j + \hbar + \bt}{\bar w_i - \bar w_j + \hbar} \right)
\]
\end{NB}
Therefore
\begin{multline*}
	[f_m, f_n] = \sum_{i} \left( \bar w_i^m (\bar w_i - \hbar)^n - \bar w_i^n (\bar w_i - \hbar)^m \right) C_i C_i' \sfu_i^{-2} \\
	-2 \hbar \bt (\hbar + \bt) \sum_{i < j} ( \bar w_i^m \bar w_j^n - \bar w_i^n \bar w_j^m ) C_i^{(j)} C_j^{(i)} \dfrac{1}{(\bar w_i - \bar w_j)(\bar w_i - \bar w_j + \hbar)(\bar w_i - \bar w_j - \hbar)} \sfu_i^{-1} \sfu_j^{-1}.
\end{multline*}
We also have
\begin{equation*}
	\begin{split}
		f_0^2 &= \sum_{i} C_i C_i' \sfu_i^{-2} \\
		& \quad + \sum_{i < j} C_i^{(j)} C_j^{(i)} \left( \dfrac{\bar w_i - \bar w_j + \bt}{\bar w_i - \bar w_j} \dfrac{\bar w_j - \bar w_i + \hbar + \bt}{\bar w_j - \bar w_i + \hbar} + \dfrac{\bar w_j - \bar w_i + \bt}{\bar w_j - \bar w_i} \dfrac{\bar w_i - \bar w_j + \hbar + \bt}{\bar w_i - \bar w_j + \hbar} \right) \sfu_i^{-1} \sfu_j^{-1} \\
		&= \sum_{i} C_i C_i' \sfu_i^{-2} + 2 \sum_{i < j} C_i^{(j)} C_j^{(i)} \dfrac{(\bar w_i - \bar w_j)^3 - (\hbar^2 + \bt(\hbar + \bt))(\bar w_i - \bar w_j)}{(\bar w_i - \bar w_j)(\bar w_i - \bar w_j + \hbar)(\bar w_i - \bar w_j - \hbar)} \sfu_i^{-1} \sfu_j^{-1}.
	\end{split}
\end{equation*}
The coefficient of $\sfu_i^{-2}$ in the left hand side of \eqref{eq:f} is
\begin{multline*}
	C_i C_i' \\
	\times \big( 3 \left( \bar w_i^2 (\bar w_i - \hbar) - \bar w_i (\bar w_i - \hbar)^2 \right) - \left( \bar w_i^3 - (\bar w_i - \hbar)^3 \right) + (\hbar^2 + \bt(\hbar + \bt)) \left( \bar w_i - (\bar w_i - \hbar) \right) \\
	- \hbar\bt(\hbar + \bt) \big) = 0. 
\end{multline*}
The coefficient of $\sfu_i^{-1} \sfu_j^{-1}$ ($i<j$) in the left hand side of \eqref{eq:f} is
\begin{multline*}
	C_i^{(j)} C_j^{(i)} \dfrac{2\hbar \bt (\hbar + \bt)}{(\bar w_i - \bar w_j)(\bar w_i - \bar w_j + \hbar)(\bar w_i - \bar w_j - \hbar)} \\
	\times \left(-\left( 3 \left( \bar w_i^2 \bar w_j - \bar w_i \bar w_j^2 \right) - \left( \bar w_i^3 - \bar w_j^3 \right) + (\hbar^2 + \bt(\hbar + \bt)) \left( \bar w_i - \bar w_j \right) \right)\right. \\
	- \left.\left( (\bar w_i - \bar w_j)^3 - (\hbar^2 + \bt(\hbar + \bt))(\bar w_i - \bar w_j) \right)\right) = 0.
\end{multline*}


\begin{NB}
\begin{equation*}
	\begin{split}
	& C_{ij} - C_{ji} = \dfrac{\bar w_j - \bar w_i}{\bar w_j - \bar w_i + \bt} \dfrac{\bar w_j - \bar w_i + \bt + \hbar}{\bar w_j - \bar w_i + \hbar} - \dfrac{\bar w_i - \bar w_j}{\bar w_i - \bar w_j + \bt} \dfrac{\bar w_i - \bar w_j + \bt + \hbar}{\bar w_i - \bar w_j + \hbar} \\
	&= \dfrac{\bar w_i - \bar w_j}{(\bar w_i - \bar w_j + \hbar)(\bar w_i - \bar w_j - \hbar)(\bar w_i - \bar w_j + \bt)(\bar w_i - \bar w_j - \bt)} \\
	& \left( (\bar w_i - \bar w_j - (\hbar+\bt))(\bar w_i - \bar w_j + \hbar)(\bar w_i - \bar w_j + \bt) - (\bar w_i - \bar w_j + (\hbar+\bt))(\bar w_i - \bar w_j - \hbar)(\bar w_i - \bar w_j - \bt) \right) \\
	&= - 2 \dfrac{\bar w_i - \bar w_j}{(\bar w_i - \bar w_j + \hbar)(\bar w_i - \bar w_j - \hbar)(\bar w_i - \bar w_j + \bt)(\bar w_i - \bar w_j - \bt)} \hbar \bt (\hbar + \bt),
	\end{split}
\end{equation*}
\begin{equation*}
	\begin{split}
	& C_{ij} + C_{ji} = \dfrac{\bar w_j - \bar w_i}{\bar w_j - \bar w_i + \bt} \dfrac{\bar w_j - \bar w_i + \bt + \hbar}{\bar w_j - \bar w_i + \hbar} + \dfrac{\bar w_i - \bar w_j}{\bar w_i - \bar w_j + \bt} \dfrac{\bar w_i - \bar w_j + \bt + \hbar}{\bar w_i - \bar w_j + \hbar} \\
	&= \dfrac{\bar w_i - \bar w_j}{(\bar w_i - \bar w_j + \hbar)(\bar w_i - \bar w_j - \hbar)(\bar w_i - \bar w_j + \bt)(\bar w_i - \bar w_j - \bt)} \\
	& \left( (\bar w_i - \bar w_j - (\hbar+\bt))(\bar w_i - \bar w_j + \hbar)(\bar w_i - \bar w_j + \bt) + (\bar w_i - \bar w_j + (\hbar+\bt))(\bar w_i - \bar w_j - \hbar)(\bar w_i - \bar w_j - \bt) \right) \\
	&= 2 \dfrac{\bar w_i - \bar w_j}{(\bar w_i - \bar w_j + \hbar)(\bar w_i - \bar w_j - \hbar)(\bar w_i - \bar w_j + \bt)(\bar w_i - \bar w_j - \bt)} ((\bar w_i - \bar w_j)^3 - (\hbar^2 + \bt(\hbar + \bt))(\bar w_i - \bar w_j)).
	\end{split}
\end{equation*}
\end{NB}

The proof of \eqref{eq:e} is the same, hence is omitted.

\begin{NB}
Let us prove \eqref{eq:e} for $e_n = E_1[(\bar w + \hbar)^n]$.
Put
\[
	D_i = \prod_{j \neq i} \dfrac{\bar w_i - \bar w_j - \bt}{\bar w_i - \bar w_j}
\]
so that $E_1[(\bar w + \hbar)^n] = \sum_{i=1}^N (\bar w_i + \hbar)^n D_i \sfu_i$.
We define $D_i'$ and $D_{ij}$ by
\[
	\sfu_i D_i = D_i' \sfu_i, \quad \sfu_i D_j = D_j D_{ij} \sfu_i \ (i \neq j).
\]
Explicitly they are given by
\begin{gather*}
	D_i' = \prod_{j \neq i} \dfrac{\bar w_i - \bar w_j + \bt + \hbar}{\bar w_i - \bar w_j + \hbar}, \\
	D_{ij} = \dfrac{\bar w_j - \bar w_i}{\bar w_j - \bar w_i - \bt} \dfrac{\bar w_j - \bar w_i - \bt - \hbar}{\bar w_j - \bar w_i - \hbar}.
\end{gather*}
We have
\begin{align*}
	& [e_m, e_n] = \sum_{i,j} [(\bar w_i + \hbar)^m D_i \sfu_i, (\bar w_j + \hbar)^n D_j \sfu_j] \\
	=\; & \sum_{i} \left( (\bar w_i + \hbar)^m (\bar w_i + 2\hbar)^n - (\bar w_i + \hbar)^n (\bar w_i + 2\hbar)^m \right) D_i D_i' \sfu_i^{2} \\
	& \qquad + \sum_{i \neq j} (\bar w_i + \hbar)^m (\bar w_j + \hbar)^n D_i D_j (D_{ij} - D_{ji}) \sfu_i \sfu_j \\
	=\; & \sum_{i} \left( (\bar w_i + \hbar)^m (\bar w_i + 2\hbar)^n - (\bar w_i + \hbar)^n (\bar w_i + 2\hbar)^m \right) D_i D_i' \sfu_i^{2} \\
	& \qquad + \sum_{i < j} ( (\bar w_i + \hbar)^m (\bar w_j + \hbar)^n - (\bar w_i + \hbar)^n (\bar w_j + \hbar)^m) D_i D_j (D_{ij} - D_{ji}) \sfu_i \sfu_j
\end{align*}
and
\begin{align*}
	e_0^2 &= \sum_{i} D_i D_i' \sfu_i^{2} + \sum_{i \neq j} D_i D_j D_{ij} \sfu_i \sfu_j \\
	&= \sum_{i} D_i D_i' \sfu_i^{2} + \sum_{i < j} D_i D_j (D_{ij} + D_{ji}) \sfu_i \sfu_j.
\end{align*}

The coefficient of $\sfu_i^{2}$ in the left hand side of \eqref{eq:e} is
\begin{multline*}
	D_i D_i' \\
	\times \Big( 3 \left( (\bar w_i + \hbar)^2 (\bar w_i + 2\hbar) - (\bar w_i + \hbar) (\bar w_i + 2\hbar)^2 \right) - \left( (\bar w_i + \hbar)^3 - (\bar w_i + 2\hbar)^3 \right) \\
	+ (\hbar^2 + \bt(\hbar + \bt)) \left( (\bar w_i + \hbar) - (\bar w_i + 2\hbar) \right) + \hbar\bt(\hbar + \bt) \Big) = 0.
\end{multline*}
The coefficient of $\sfu_i \sfu_j$ ($i<j$) in the left hand side of \eqref{eq:e} is given by
\begin{multline*}
	D_i D_j \\
	\times \Big( \big( 3 \left( (\bar w_i + \hbar)^2(\bar w_j + \hbar) - (\bar w_i + \hbar)(\bar w_j + \hbar)^2 \right) - \left( (\bar w_i + \hbar)^3 - (\bar w_j + \hbar)^3 \right) \\
	+ (\hbar^2 + \bt(\hbar + \bt)) \left( (\bar w_i + \hbar) - (\bar w_j + \hbar) \right) \big) (D_{ij} - D_{ji}) + \hbar\bt(\hbar + \bt) (D_{ij} + D_{ji}) \Big).
\end{multline*}
This is equal to $0$ since we have
\begin{equation*}
	3 \left( (\bar w_i + \hbar)^2 (\bar w_j + \hbar) - (\bar w_i + \hbar) (\bar w_j + \hbar)^2 \right) - \left( (\bar w_i + \hbar)^3 - (\bar w_j + \hbar)^3 \right) = - (\bar w_i - \bar w_j)^3,
\end{equation*}
\begin{align*}
	& D_{ij} - D_{ji} = -( C_{ij} - C_{ji} ) \\
	=\; & 2 \dfrac{\bar w_i - \bar w_j}{(\bar w_i - \bar w_j + \hbar)(\bar w_i - \bar w_j - \hbar)(\bar w_i - \bar w_j + \bt)(\bar w_i - \bar w_j - \bt)} \hbar \bt (\hbar + \bt), \\
	&\\
	& D_{ij} + D_{ji} = C_{ij} + C_{ji} \\
	=\; & 2 \dfrac{\bar w_i - \bar w_j}{(\bar w_i - \bar w_j + \hbar)(\bar w_i - \bar w_j - \hbar)(\bar w_i - \bar w_j + \bt)(\bar w_i - \bar w_j - \bt)} \\	& \qquad ((\bar w_i - \bar w_j)^3 - (\hbar^2 + \bt(\hbar + \bt))(\bar w_i - \bar w_j)).
\end{align*}
\end{NB}

Let us prove \eqref{eq:tri} for $f_n$.
The proof for $e_n$ is similar.
We define $C_i''$ and $C_{j}'^{(i)}$ ($i \neq j$) by
\begin{gather*}
	C_i'' = \prod_{j \neq i} \dfrac{\bar w_i - \bar w_j + \bt - 2\hbar}{\bar w_i - \bar w_j - 2\hbar} \begin{NB}\prod_{k=1}^l (\bar w_i + (N-1)\bt - 3\hbar - z_k)\end{NB}, \\
	C_{j}'^{(i)} = C_j' \dfrac{\bar w_j - \bar w_i - \hbar}{\bar w_j - \bar w_i - \hbar + \bt} = \prod_{k \neq i,j} \dfrac{\bar w_j - \bar w_k - \hbar + \bt}{\bar w_j - \bar w_k - \hbar}
\end{gather*}
so that
\[
	\sfu_i^{-1} C_i' = C_i'' \sfu_i^{-1}, \quad \sfu_i^{-1} C_j' = C_j'^{(i)} \dfrac{\bar w_j - \bar w_i + \bt}{\bar w_j - \bar w_i} \sfu_i^{-1} \ (i \neq j).
\]
We have
\begin{multline}\label{eq:tri2}
	[f_0, [f_0, f_1]] = -\hbar \sum_{i,j} [ C_i \sfu_i^{-1}, C_j C_j' \sfu_j^{-2}] \\
	+ 2 \hbar \bt (\hbar + \bt) \sum_{i} \sum_{j<k} [C_i \sfu_i^{-1}, C_j^{(k)} C_k^{(j)} \dfrac{1}{(\bar w_j - \bar w_k + \hbar)(\bar w_j - \bar w_k - \hbar)} \sfu_j^{-1} \sfu_k^{-1}].
\end{multline}
Consider the case $i=j$ in the first sum of \eqref{eq:tri2}.
The summand is
\[
	(C_i C_i' C_i'' - C_i C_i' C_i'') \sfu_i^{-3} = 0.
\]
For the case $i \neq j$ in the first sum of \eqref{eq:tri2}, we have
\begin{equation*}
	\begin{split}
		C_i \sfu_i^{-1} C_j C_j' \sfu_j^{-2} &= C_i C_j^{(i)} \dfrac{\bar w_j - \bar w_i + \hbar + \bt}{\bar w_j - \bar w_i + \hbar} C_j'^{(i)} \dfrac{\bar w_j - \bar w_i + \bt}{\bar w_j - \bar w_i} \sfu_i^{-1} \sfu_j^{-2} \\
		&= C_i^{(j)} C_j C_j'^{(i)} \dfrac{\bar w_i - \bar w_j + \bt}{\bar w_i - \bar w_j} \dfrac{\bar w_i - \bar w_j - \hbar - \bt}{\bar w_i - \bar w_j - \hbar} \sfu_i^{-1} \sfu_j^{-2}.
	\end{split}
\end{equation*}
\begin{NB}
where we use
\[
	C_i = C_i^{(j)} \dfrac{\bar w_i - \bar w_j + \bt}{\bar w_i - \bar w_j}, \quad C_j^{(i)}  \dfrac{\bar w_j - \bar w_i + \bt}{\bar w_j - \bar w_i} = C_j
\]
\end{NB}
Furthermore,
\begin{equation*}
	\begin{split}
		C_j C_j' \sfu_j^{-2} C_i \sfu_i^{-1} &= C_j C_j' C_i^{(j)} \dfrac{\bar w_i - \bar w_j + 2\hbar + \bt}{\bar w_i - \bar w_j + 2\hbar} \sfu_i^{-1} \sfu_j^{-2} \\
		&= C_j C_j'^{(i)} C_i^{(j)} \dfrac{\bar w_j - \bar w_i - \hbar + \bt}{\bar w_j - \bar w_i - \hbar} \dfrac{\bar w_i - \bar w_j + 2\hbar + \bt}{\bar w_i - \bar w_j + 2\hbar} \sfu_i^{-1} \sfu_j^{-2}.
	\end{split}
\end{equation*}
\begin{NB}
where we use
\[
	C_j' = C_j'^{(i)} \dfrac{\bar w_j - \bar w_i - \hbar + \bt}{\bar w_j - \bar w_i - \hbar}
\]
\end{NB}
Thus we have
\begin{equation*}
	\begin{split}
		&  [C_i \sfu_i^{-1}, C_j C_j' \sfu_j^{-2}] \\
		=\; & C_i^{(j)} C_j C_j'^{(i)} \left( \dfrac{\bar w_i - \bar w_j + \bt}{\bar w_i - \bar w_j} \dfrac{\bar w_i - \bar w_j - \hbar - \bt}{\bar w_i - \bar w_j - \hbar} - \dfrac{\bar w_j - \bar w_i - \hbar + \bt}{\bar w_j - \bar w_i - \hbar} \dfrac{\bar w_i - \bar w_j + 2\hbar + \bt}{\bar w_i - \bar w_j + 2\hbar} \right) \sfu_i^{-1} \sfu_j^{-2}
	\end{split}
\end{equation*}
and denote this by $A_{ij}$.
Consider the case $i=j < k$ in the second sum of \eqref{eq:tri2}.
We have
\begin{equation*}
	\begin{split}
		& C_j \sfu_j^{-1} C_j^{(k)} C_k^{(j)} \dfrac{1}{(\bar w_j - \bar w_k + \hbar)(\bar w_j - \bar w_k - \hbar)}\sfu_j^{-1} \sfu_k^{-1} \\
		=\; & C_j C_j'^{(k)} C_k^{(j)} \dfrac{1}{(\bar w_j - \bar w_k)(\bar w_j - \bar w_k - 2\hbar)}\sfu_j^{-2} \sfu_k^{-1},
	\end{split}
\end{equation*}
and
\begin{equation*}
	\begin{split}
		& C_j^{(k)} C_k^{(j)} \dfrac{1}{(\bar w_j - \bar w_k + \hbar)(\bar w_j - \bar w_k - \hbar)} \sfu_j^{-1} \sfu_k^{-1}  C_j \sfu_j^{-1} \\
		=\; & C_j^{(k)} C_k^{(j)} \dfrac{1}{(\bar w_j - \bar w_k + \hbar)(\bar w_j - \bar w_k - \hbar)} C_j'^{(k)} \dfrac{\bar w_j - \bar w_k + \bt}{\bar w_j - \bar w_k} \sfu_j^{-2} \sfu_k^{-1} \\
		=\; & C_j C_k^{(j)} C_j'^{(k)} \dfrac{1}{(\bar w_j - \bar w_k + \hbar)(\bar w_j - \bar w_k - \hbar)} \sfu_j^{-2} \sfu_k^{-1}.
	\end{split}
\end{equation*}
\begin{NB}
where we use
\[
	C_j^{(k)} \dfrac{\bar w_j - \bar w_k + \bt}{\bar w_j - \bar w_k} = C_j
\]
\end{NB}
Thus we have
\begin{equation*}
	\begin{split}
		& [C_j \sfu_j^{-1}, C_j^{(k)} C_k^{(j)} \dfrac{1}{(\bar w_j - \bar w_k + \hbar)(\bar w_j - \bar w_k - \hbar)}\sfu_j^{-1} \sfu_k^{-1}] \\
		=\; & C_j C_j'^{(k)} C_k^{(j)} \left( \dfrac{1}{(\bar w_j - \bar w_k)(\bar w_j - \bar w_k - 2\hbar)} - \dfrac{1}{(\bar w_j - \bar w_k + \hbar)(\bar w_j - \bar w_k - \hbar)} \right) \sfu_j^{-2} \sfu_k^{-1}
	\end{split}
\end{equation*}
and denote this by $B_{jk}$.
Similarly the case $j < k=i$ in the second sum of \eqref{eq:tri2} is given by
\begin{equation*}
	\begin{split}
		& [C_k \sfu_k^{-1}, C_j^{(k)} C_k^{(j)} \dfrac{1}{(\bar w_j - \bar w_k + \hbar)(\bar w_j - \bar w_k - \hbar)}\sfu_j^{-1} \sfu_k^{-1}] \\
		=\; & C_k C_j^{(k)} C_k'^{(j)} \left( \dfrac{1}{(\bar w_j - \bar w_k + 2\hbar)(\bar w_j - \bar w_k)} - \dfrac{1}{(\bar w_j - \bar w_k + \hbar)(\bar w_j - \bar w_k - \hbar)} \right) \sfu_j^{-1} \sfu_k^{-2} \\
		=\; & B_{kj}.
	\end{split}
\end{equation*}
\begin{NB}
We have
\begin{equation*}
	\begin{split}
		& C_k \sfu_k^{-1} C_j^{(k)} C_k^{(j)} \dfrac{1}{(\bar w_j - \bar w_k + \hbar)(\bar w_j - \bar w_k - \hbar)}\sfu_j^{-1} \sfu_k^{-1} \\
		=\; & C_k C_j^{(k)} C_k'^{(j)} \dfrac{1}{(\bar w_j - \bar w_k + 2\hbar)(\bar w_j - \bar w_k)}\sfu_j^{-1} \sfu_k^{-2},
	\end{split}
\end{equation*}
and
\begin{equation*}
	\begin{split}
		& C_j^{(k)} C_k^{(j)} \dfrac{1}{(\bar w_j - \bar w_k + \hbar)(\bar w_j - \bar w_k - \hbar)} \sfu_j^{-1} \sfu_k^{-1} C_k \sfu_k^{-1} \\
		=\; & C_j^{(k)} C_k^{(j)} \dfrac{1}{(\bar w_j - \bar w_k + \hbar)(\bar w_j - \bar w_k - \hbar)} C_k'^{(j)} \dfrac{\bar w_k - \bar w_j + \bt}{\bar w_k - \bar w_j} \sfu_j^{-1} \sfu_k^{-2} \\
		=\; & C_j^{(k)} C_k C_k'^{(j)} \dfrac{1}{(\bar w_j - \bar w_k + \hbar)(\bar w_j - \bar w_k - \hbar)} \sfu_j^{-1} \sfu_k^{-2},
	\end{split}
\end{equation*}
where we use
\[
	C_k^{(j)} \dfrac{\bar w_k - \bar w_j + \bt}{\bar w_k - \bar w_j} = C_k.
\]
\end{NB}
Hence the coefficient of $\sfu_j^{-2} \sfu_k^{-1}$ ($j<k$) in \eqref{eq:tri2} is given by
\begin{equation*}
	\begin{split}
		& -\hbar A_{kj} + 2\hbar \bt (\hbar + \bt) B_{jk} \\
		=\; & -\hbar C_k^{(j)} C_j C_j'^{(k)} \left( \dfrac{\bar w_k - \bar w_j + \bt}{\bar w_k - \bar w_j} \dfrac{\bar w_k - \bar w_j - \hbar - \bt}{\bar w_k - \bar w_j - \hbar} - \dfrac{\bar w_j - \bar w_k - \hbar + \bt}{\bar w_j - \bar w_k - \hbar} \dfrac{\bar w_k - \bar w_j + 2\hbar + \bt}{\bar w_k - \bar w_j + 2\hbar} \right) \\
		& \qquad + 2\hbar \bt (\hbar + \bt) C_j C_j'^{(k)} C_k^{(j)} \left( \dfrac{1}{(\bar w_j - \bar w_k)(\bar w_j - \bar w_k - 2\hbar)} - \dfrac{1}{(\bar w_j - \bar w_k + \hbar)(\bar w_j - \bar w_k - \hbar)} \right).
	\end{split}
\end{equation*}
Set $x = \bar w_j - \bar w_k$.
A direct calculation shows the identity
\begin{multline*}
	-\hbar \left( \dfrac{(x - \bt)(x + \hbar + \bt)}{x(x+\hbar)} - \dfrac{(x - \hbar + \bt)(x - 2\hbar - \bt)}{(x - \hbar)(x - 2\hbar)}\right) \\
	+ 2\hbar \bt (\hbar + \bt) \left( \dfrac{1}{x(x - 2\hbar)} - \dfrac{1}{(x + \hbar)(x - \hbar)} \right) = 0. 
\end{multline*}
Hence the above vanishes.
\begin{NB}
We use
\begin{gather*}
	\hbar\dfrac{1}{x(x+\hbar)} = \dfrac{1}{x} - \dfrac{1}{x+\hbar}, \quad \hbar\dfrac{1}{(x - \hbar)(x - 2\hbar)} = \dfrac{1}{x-2\hbar} - \dfrac{1}{x-\hbar}, \\
	2\hbar\dfrac{1}{x(x-2\hbar)} = \dfrac{1}{x-2\hbar} - \dfrac{1}{x}, \quad 2\hbar\dfrac{1}{(x+\hbar)(x-\hbar)} = \dfrac{1}{x-\hbar} - \dfrac{1}{x+\hbar}.
\end{gather*}
The left hand side of \eqref{eq:tri3} is
\begin{equation*}
	\begin{split}
		&- \left( \dfrac{1}{x} - \dfrac{1}{x+\hbar} \right) (x - \bt)(x + \hbar + \bt) + \left( \dfrac{1}{x-2\hbar} - \dfrac{1}{x-\hbar} \right) (x - \hbar + \bt)(x - 2\hbar - \bt) \\
		& \qquad + \bt (\hbar + \bt) \left( \dfrac{1}{x-2\hbar} - \dfrac{1}{x} - \dfrac{1}{x-\hbar} + \dfrac{1}{x+\hbar} \right) \\
		=\; & -\dfrac{1}{x} x(x+\hbar) + \dfrac{1}{x+\hbar} x(x+\hbar) + \dfrac{1}{x-2\hbar}(x-\hbar)(x-2\hbar) - \dfrac{1}{x-\hbar}(x-\hbar)(x-2\hbar) = 0.
	\end{split}
\end{equation*}
\end{NB}
The same argument shows that the coefficient of $\sfu_j^{-1} \sfu_k^{-2}$ ($j<k$) in \eqref{eq:tri2} is
\begin{equation*}
	-\hbar A_{jk} + 2\hbar \bt (\hbar + \bt) B_{kj} =0.		
\end{equation*}
We consider the coefficient of $\sfu_i^{-1} \sfu_j^{-1} \sfu_k^{-1}$ ($i<j<k$) in \eqref{eq:tri2}.
Define $C_j^{(k,i)}$ by
\[
	C_j^{(k,i)} = C_j^{(k)} \dfrac{\bar w_j - \bar w_i}{\bar w_j - \bar w_i  + \bt} = \prod_{l \neq i,j,k} \dfrac{\bar w_j - \bar w_l + \bt}{\bar w_j - \bar w_l}
\]
so that
\[
	\sfu_i^{-1} C_j^{(k)} = C_j^{(k,i)} \dfrac{\bar w_j - \bar w_i + \hbar + \bt}{\bar w_j - \bar w_i + \hbar} \sfu_i^{-1}.
\]
Define $C_k^{(j,i)}$ and $C_i^{(j,k)}$ similarly.
Then we have
\allowdisplaybreaks{
\begin{equation*}
	\begin{split}
		& C_i \sfu_i^{-1} C_j^{(k)} C_k^{(j)} \dfrac{1}{(\bar w_j - \bar w_k + \hbar)(\bar w_j - \bar w_k - \hbar)} \sfu_j^{-1} \sfu_k^{-1} \\
		=\; & C_i C_j^{(k,i)} \dfrac{\bar w_j - \bar w_i + \hbar + \bt}{\bar w_j - \bar w_i + \hbar} C_k^{(j,i)} \dfrac{\bar w_k - \bar w_i + \hbar + \bt}{\bar w_k - \bar w_i + \hbar} \dfrac{1}{(\bar w_j - \bar w_k + \hbar)(\bar w_j - \bar w_k - \hbar)} \sfu_i^{-1} \sfu_j^{-1} \sfu_k^{-1} \\
		=\; & C_i^{(k,j)} C_j^{(k,i)} C_k^{(j,i)} \dfrac{\bar w_i - \bar w_j + \bt}{\bar w_i - \bar w_j} \dfrac{\bar w_i - \bar w_k + \bt}{\bar w_i - \bar w_k} \dfrac{\bar w_j - \bar w_i + \hbar + \bt}{\bar w_j - \bar w_i + \hbar} \dfrac{\bar w_k - \bar w_i + \hbar + \bt}{\bar w_k - \bar w_i + \hbar} \\
		& \qquad \times \dfrac{1}{(\bar w_j - \bar w_k + \hbar)(\bar w_j - \bar w_k - \hbar)} \sfu_i^{-1} \sfu_j^{-1} \sfu_k^{-1}.
	\end{split}
\end{equation*}
}
\begin{NB}
where we use
\[
	C_i = C_i^{(k,j)} \dfrac{\bar w_i - \bar w_j + \bt}{\bar w_i - \bar w_j} \dfrac{\bar w_i - \bar w_k + \bt}{\bar w_i - \bar w_k}
\]
\end{NB}
Also
\allowdisplaybreaks{
\begin{equation*}
	\begin{split}
		& C_j^{(k)} C_k^{(j)} \dfrac{1}{(\bar w_j - \bar w_k + \hbar)(\bar w_j - \bar w_k - \hbar)} \sfu_j^{-1} \sfu_k^{-1} C_i \sfu_i^{-1} \\
		=\; & C_j^{(k)} C_k^{(j)} \dfrac{1}{(\bar w_j - \bar w_k + \hbar)(\bar w_j - \bar w_k - \hbar)} C_i^{(j,k)} \dfrac{\bar w_i - \bar w_j + \hbar + \bt}{\bar w_i - \bar w_j + \hbar} \dfrac{\bar w_i - \bar w_k + \hbar + \bt}{\bar w_i - \bar w_k + \hbar} \sfu_i^{-1} \sfu_j^{-1} \sfu_k^{-1} \\
		=\; & C_j^{(k,i)} C_k^{(j,i)} C_i^{(j,k)} \dfrac{1}{(\bar w_j - \bar w_k + \hbar)(\bar w_j - \bar w_k - \hbar)} \\
		& \qquad \times \dfrac{\bar w_j - \bar w_i + \bt}{\bar w_j - \bar w_i} \dfrac{\bar w_k - \bar w_i + \bt}{\bar w_k - \bar w_i}  \dfrac{\bar w_i - \bar w_j + \hbar + \bt}{\bar w_i - \bar w_j + \hbar} \dfrac{\bar w_i - \bar w_k + \hbar + \bt}{\bar w_i - \bar w_k + \hbar} \sfu_i^{-1} \sfu_j^{-1} \sfu_k^{-1}.
	\end{split}
\end{equation*}
}
\begin{NB}
where we use
\begin{equation*}
	C_j^{(k)} = C_j^{(k,i)} \dfrac{\bar w_j - \bar w_i + \bt}{\bar w_j - \bar w_i}, \quad C_k^{(j)} = C_k^{(j,i)} \dfrac{\bar w_k - \bar w_i + \bt}{\bar w_k - \bar w_i}
\end{equation*}
\end{NB}
Thus we have
\allowdisplaybreaks{
\begin{equation*}
	\begin{split}
		& [C_i \sfu_i^{-1}, C_j^{(k)} C_k^{(j)} \dfrac{1}{(\bar w_j - \bar w_k + \hbar)(\bar w_j - \bar w_k - \hbar)}\sfu_j^{-1} \sfu_k^{-1}] \\
		=\; & C_i^{(k,j)} C_j^{(k,i)} C_k^{(j,i)} \dfrac{1}{(\bar w_i - \bar w_j)(\bar w_i - \bar w_k)(\bar w_j - \bar w_k + \hbar)(\bar w_j - \bar w_k - \hbar)} \\
		& \qquad \Bigg( \dfrac{(\bar w_i - \bar w_j + \bt)(\bar w_i - \bar w_j - \hbar - \bt)(\bar w_i - \bar w_k + \bt)(\bar w_i - \bar w_k - \hbar - \bt)}{(\bar w_i - \bar w_j - \hbar)(\bar w_i - \bar w_k - \hbar)} \\
		& \qquad \qquad -  \dfrac{(\bar w_i - \bar w_j - \bt)(\bar w_i - \bar w_j + \hbar + \bt)(\bar w_i - \bar w_k - \bt)(\bar w_i - \bar w_k + \hbar + \bt)}{(\bar w_i - \bar w_j + \hbar)(\bar w_i - \bar w_k + \hbar)} \Bigg) \sfu_i^{-1} \sfu_j^{-1} \sfu_k^{-1} \\
		=\; & C_i^{(k,j)} C_j^{(k,i)} C_k^{(j,i)} \dfrac{1}{x_{ij} x_{jk} x_{ki} (x_{ij} + \hbar)(x_{ij} - \hbar)(x_{jk} + \hbar)(x_{jk} - \hbar)(x_{ki} + \hbar)(x_{ki} - \hbar)} \\
		& \qquad \times x_{jk} \Big( (x_{ij} + \hbar)(x_{ij} + \bt)(x_{ij} - \hbar - \bt)(x_{ki} - \hbar)(x_{ki} - \bt)(x_{ki} + \hbar + \bt) \\
		& \qquad \qquad - (x_{ij} - \hbar)(x_{ij} - \bt)(x_{ij} + \hbar + \bt)(x_{ki} + \hbar)(x_{ki} + \bt)(x_{ki} - \hbar - \bt) \Big) \sfu_i^{-1} \sfu_j^{-1} \sfu_k^{-1},
	\end{split}
\end{equation*}
where} we set $x_{ij} = \bar w_i - \bar w_j$ and so on.
Put
\begin{multline*}
	A_{ijk} = x_{jk} \Big( (x_{ij} + \hbar)(x_{ij} + \bt)(x_{ij} - \hbar - \bt)(x_{ki} - \hbar)(x_{ki} - \bt)(x_{ki} + \hbar + \bt) \\
	- (x_{ij} - \hbar)(x_{ij} - \bt)(x_{ij} + \hbar + \bt)(x_{ki} + \hbar)(x_{ki} + \bt)(x_{ki} - \hbar - \bt) \Big).
\end{multline*}
It is enough to show $A_{ijk} + A_{jki} + A_{kij} = 0$.
We have
\begin{align*}
	A_{ijk} &= 2 \hbar \bt (\hbar + \bt) x_{jk} ( x_{ij}^3 - x_{ki}^3 - (\hbar^2 +\bt (\hbar + \bt))(x_{ij} - x_{ki}) ), \\
	A_{jki} &= 2 \hbar \bt (\hbar + \bt) x_{ki} ( x_{jk}^3 - x_{ij}^3 - (\hbar^2 +\bt (\hbar + \bt))(x_{jk} - x_{ij}) ),\\
	A_{kij} &= 2 \hbar \bt (\hbar + \bt) x_{ij} ( x_{ki}^3 - x_{jk}^3 - (\hbar^2 +\bt (\hbar + \bt))(x_{ki} - x_{jk}) ).
\end{align*}
A straightforward calculation shows
\begin{gather*}
	x_{jk}( x_{ij}^3 - x_{ki}^3 ) + x_{ki} ( x_{jk}^3 - x_{ij}^3 ) + x_{ij} ( x_{ki}^3 - x_{jk}^3 ) = 0, \\
	x_{jk}( x_{ij} - x_{ki} ) + x_{ki} ( x_{jk} - x_{ij} ) + x_{ij} ( x_{ki} - x_{jk} ) = 0.
\end{gather*}
This completes the proof of \eqref{eq:tri}.

\section{}\label{sec:app2}

We check that the relation (\ref{eq:3}') holds for
\begin{gather*}
    h(x) = \prod_{k=1}^l (1 - (z_k + \hbar)x) \prod_{i=1}^N
    \dfrac{(1 - (w_i - \bt)x)(1 - (w_i + \hbar + \bt)x)}{(1 - w_i x)(1 - (w_i + \hbar)x)}, \\
    e_n = E_1[(w+\hbar)^n].
\end{gather*}
We have
\begin{equation*}
	\begin{split}
		& [h(x), e_n] \\
		=\; & \sum_{i=1}^N (w_i+\hbar)^n \prod_{j \neq i} \dfrac{w_i-w_j-\bt}{w_i-w_j} \prod_{k=1}^l (1-(z_k+\hbar)x) \prod_{j \neq i} \dfrac{(1-(w_j - \bt)x)(1-(w_j + \hbar + \bt)x)}{(1-w_j x)(1-(w_j + \hbar)x)} \\
		& \quad \times \left( \dfrac{(1-(w_i - \bt)x)(1-(w_i + \hbar + \bt)x)}{(1-w_i x)(1-(w_i + \hbar)x)} - \dfrac{(1-(w_i + \hbar - \bt)x)(1-(w_i + 2\hbar + \bt)x)}{(1-(w_i + \hbar)x)(1-(w_i + 2\hbar)x)} \right) \sfu_i \\
		=\; & \sum_{i=1}^N (w_i+\hbar)^n \prod_{j \neq i} \dfrac{w_i-w_j-\bt}{w_i-w_j} \prod_{k=1}^l (1-(z_k+\hbar)x) \prod_{j \neq i} \dfrac{(1-(w_j - \bt)x)(1-(w_j + \hbar + \bt)x)}{(1-w_j x)(1-(w_j + \hbar)x)} \\
		& \quad \times \dfrac{2\hbar\bt(\hbar+\bt)x^3}{(1-w_i x)(1-(w_i + \hbar)x)(1-(w_i + 2\hbar)x)} \sfu_i.
	\end{split}
\end{equation*}
This shows $[h_0, e_n] = [h_1, e_n] =0$ and $[h_2, e_n] = -2\hbar e_n$.
Then it is enough to prove that the term with positive powers in $y$ of
\begin{multline}\label{eq:series}
  \left( (x^{-1} - y^{-1})^3 - (\hbar^2 + \bt(\hbar + \bt))(x^{-1} - y^{-1}) \right) [h(x), e(y)]
  \\ - \hbar \bt (\hbar+\bt) (h(x)e(y) + e(y)h(x))
\end{multline}
vanishes.
We have
\begin{equation*}
	\begin{split}
		& h(x) e(y) + e(y)h(x) = \sum_{n \geq 0} y^{n+1} \sum_{i=1}^N (w_i+\hbar)^n \\
		& \quad \times \prod_{j \neq i} \dfrac{w_i-w_j-\bt}{w_i-w_j} \prod_{k=1}^l (1-(z_k+\hbar)x) \prod_{j \neq i} \dfrac{(1-(w_j - \bt)x)(1-(w_j + \hbar + \bt)x)}{(1-w_j x)(1-(w_j + \hbar)x)} \\
		& \quad \quad \times \dfrac{2P_i(x)}{(1-w_i x)(1-(w_i + \hbar)x)(1-(w_i + 2\hbar)x)} \sfu_i, \\
	\end{split}
\end{equation*}
where
\begin{multline*}
	P_i(x) = 1 - 3(w_i + \hbar)x + \left( 3(w_i + \hbar)^2 - (\hbar^2 + \bt(\hbar + \bt)) \right)x^2 \\
	- \left( (w_i + \hbar)^3 - (\hbar^2 + \bt(\hbar + \bt))(w_i + \hbar) \right)x^3.
\end{multline*}
Consider the coefficient of $y^{n+1}$ ($n \geq 0$) in \eqref{eq:series}.
Since we have
\begin{multline*}
	(w_i + \hbar)^n - 3(w_i + \hbar)^{n+1}x + 3(w_i + \hbar)^{n+2}x^2 - (w_i+\hbar)^{n+3}x^3 \\
	- (\hbar^2+\bt(\hbar+\bt)) ((w_i + \hbar)^nx^2 - (w_i + \hbar)^{n+1}x^3) - (w_i + \hbar)^n P_i(x) = 0,
\end{multline*}
the assertion is proved.
The relation (\ref{eq:4}') can be checked similarly.
